\documentclass[a4paper,11pt]{amsart}
\usepackage{amsmath,amssymb,amsfonts,amsthm}
\usepackage[abbrev,nobysame]{amsrefs}
\usepackage{mathtools,cases}
\usepackage[foot]{amsaddr}
\usepackage{enumerate}
\numberwithin{equation}{section}

\numberwithin{equation}{section}
\newcommand{\p}{\partial}
\newcommand{\vphi}{\varphi}
\newcommand{\om}{\omega}
\newcommand{\tri}{\triangle}
\newcommand{\eps}{\epsilon}
\newcommand{\thmref}[1]{Theorem~\ref{#1}}

\newcommand{\lemref}[1]{Lemma~\ref{#1}}
\newcommand{\corref}[1]{Corollary~\ref{#1}}

\newcommand{\Om}{\Omega}

\def\p{\partial}
\def\b{\beta}

\def\osc{{\rm osc\,}}

\DeclareMathOperator{\tr}{Tr}
\DeclareMathOperator\supp{supp}
\DeclareMathOperator{\Tr}{Tr}

\newtheorem{theorem}{Theorem}[section]
\newtheorem{thm}[theorem]{Theorem}
\newtheorem{rem}[theorem]{Remark}
\newtheorem{defn}[theorem]{Definition}

\newtheorem{lem}[theorem]{Lemma}
\newtheorem{claim}[theorem]{Claim}

\newtheorem{prop}[theorem]{Proposition}

\newtheorem{cor}[theorem]{Corollary}
\newtheorem{ques}[theorem]{Question}

\def\tr{\mathop{\rm tr}\nolimits}

\def\cD{{\mathcal D}}

\let\ul=\underline

\let\vphi=\varphi

\def\a{\alpha}

\title[Existence, Properness and Geodesic stability]{Existence of constant scalar curvature K\"ahler cone metrics, properness and geodesic stability}

\author{Kai Zheng}
  \curraddr{Tongji University, Shanghai 200092, P.R. CHINA}
  \email[Kai Zheng]{kaizheng@amss.ac.cn}

\begin{document}
\maketitle

\begin{abstract} We show that the existence of constant scalar curvature K\"ahler (cscK) metrics with cone singularities is equivalent to the properness of log $K$-energy. We also prove their equivalence to the geodesic stability. They are extensions of the solution of the properness conjecture and Donaldson's geodesic stability conjecture of the cscK problem for smooth K\"ahler metrics by Chen-Cheng to the setting of cscK cone metrics. 

One applications of our main results is that we introduce and construct singular cscK metrics with possible degeneration in big cohomology class. As another application, we also prove both openness and approximation property for the path of cscK cone metrics, which are paralleling to Donaldson's continuity method through K\"ahler-Einstein cone metrics in the resolution of Yau-Tian-Donaldson conjecture for Fano K\"ahler-Einstein metrics.
\end{abstract}

\tableofcontents

\section{Introduction}

In K\"ahler geometry, Yau-Tian-Donaldson (YTD) conjecture states the equivalence between the existence of constant scalar curvature K\"ahler (cscK) metrics and the algebraic notion of $K$-stability. Both properness conjecture \cite{MR3858468,MR1787650} and Donaldson's geodesic stability conjecture \cite{MR1736211} are different forms of the YTD conjecture. There are many literatures on these conjectures, see for example the most recent surveys \cite{MR3858468,MR3966735,MR3966781}. Recently, both deep conjectures have been proven in Chen-Cheng \cite{arXiv:1712.06697,arXiv:1801.00656,arXiv:1801.05907}, stating that the existence of cscK metrics is equivalent to the properness of the $K$-energy, and it is also equivalent to the geodesic stability. 

One goal of this article is on the existence problem of singular cscK metrics allowing both singularities or degeneration. This article is a continuation of our previous work \cite{MR4020314} on cscK metrics with cone singularities, where we developed analytic tools and proved regularity, uniqueness and deformation of the cscK cone metrics. 

The YTD conjectures are expected to hold in the conical case, that is the equivalence between the existence of cscK cone metrics and algebraic stability \cite{MR3858468}. We will confirms these conjectures (\thmref{cone properness conjecture}, \thmref{cone geodesic conjecture}), via combining our previous works with Chen-Cheng's breakthrough on existence of smooth cscK metrics to the conical case. 

Furthermore, we would also introduce a new notion of singular cscK metrics with possible degeneration (Definition \ref{singular cscK defn}) and prove an existence theorem (\thmref{singular csck construction weak}) by applying an approximation method. The definition of singular cscK metric generalises the notion of singular K\"ahler-Einstein metric and also extends Yau's resolution \cite{MR480350} of Calabi conjecture with zero or negative first Chern classes (Aubin has an independent work for negative first Chern classes). The singular K\"ahler-Einstein metric is widely studied in the minimal model program from birational geometry and there are many literatures on finding singular K\"ahler-Einstein metrics on a minimal projective manifold of general types or log Fano varieties, such as \cite{MR3956691,MR3090260,MR2746347,MR2505296,MR2869020}.

\bigskip

Before we start to state the theorems, we recall some definitions. We are given a K\"ahler manifold $X$ and a K\"ahler class $\Om$. Let $\om_0$ be a smooth K\"ahler metric in $\Om$. The divisor $D \subset X$ is a smooth hypersurface in $X$ and the cone angle $\beta$ stays within $(0,1]$. We let $L_D$ be the associated line bundle of the divisor $D$ and assume that $C_1(L_D)$ is non-negative.
We denote $Aut(X;D)$ the identity component of the group of holomorphic automorphisms of $X$ which fix the divisor $D$.

The cscK cone metric is defined to be a smooth cscK metric in the regular part of the manifold $M:=X\setminus D$ with prescribed cone singularity of cone angle $\beta$ along the divisor $D$. The rigorous definition (Definition \ref{csckconemetricdefn}) requires introducing the appropriate weighted H\"older spaces (Definition \ref{4thspace}).
The cscK cone metric is the critical point of the log $K$-energy and shown to be unique up to $Aut(X;D)$ actions by using the cone geodesic. We will recall these results from \cite{MR4020314} in Section \ref{Constant scalar curvature Kahler cone metrics}.

\bigskip

 We now outline the structure of this paper and sketch the main ideas of the proof. Section \ref{Constant scalar curvature Kahler cone metrics} is devoted to recall the results on the cscK cone metrics in the previous paper \cite{MR4020314}, including the weighted H\"older spaces, cone geodesics, the asymptotic behaviour and uniqueness of the cscK cone metrics, and the log $K$-energy and its convexity property.
 
 \begin{defn}
	We say a energy functional $F$ is proper, if for any sequence $\{\vphi_i\}\subset \mathcal H$, we have
	\begin{align*}
	\lim_{i\rightarrow\infty}d_{1,G}(0,\vphi_i)=\infty \implies \lim_{i\rightarrow\infty}F(\vphi_i)=\infty.	
	\end{align*}
\end{defn}

We will prove that the existence of cscK cone metrics in terms of the properness of the log $K$-energy.

\begin{thm}[Log properness theorem]\label{cone properness conjecture}
Assume that $\Om$ is a K\"ahler class and the cone angle satisfies $0<\beta\leq 1$. Then $\Om$ admits a constant scalar curvature K\"ahler cone metric with cone angle $\beta$ is equivalent to the properness of the log $K$-energy.
\end{thm}

\bigskip

In Section \ref{Properness conjecture: automorphism is trivial} we prove one direction of the log properness \thmref{cone properness conjecture}, that is 'properness $\implies$ existence'. Precise statements are given in \thmref{properclosedness}, \thmref{Properness implies existence general}. The proof relies on two ingredients. One is the approximation scheme of the log $K$-energy, which leads to the construction of approximate twisted cscK metrics. Then we will show that this approximate sequence converges to a cscK cone metric by applying the a priori estimates given in Section \ref{a priori estimates}, that is the other ingredient of the proof of the log properness theorem. 
The a priori estimates are the core part of the cscK problem and have been a difficult problem for a long time. This difficult problem is solved in the very recent breakthrough of Chen-Cheng \cite{arXiv:1712.06697,arXiv:1801.00656,arXiv:1801.05907}, which are built on Chen-He's important work \cite{MR2993005} on the second order estimates via the integral method. 

\bigskip

In Section \ref{Singular constant scalar curvature Kahler metrics}, the approximation scheme is further extended to construct the singular cscK metrics with possible degeneration. The novel idea is to perturb the alpha invariant near a big class and apply the existence results in Section \ref{Properness conjecture: automorphism is trivial} to obtain a sequence of approximate cscK cone metrics (without degeneration). Then the a priori estimates obtained for the degenerate metrics in Section \ref{a priori estimates} is applied to get the smoothness of the limit metric. As last, we include an application of such construction on normal complex spaces. Other ingredients we have in the proof include to extend the alpha invariant to big classes by using pluripotent potential theory (Section \ref{Alpha invariants}), to obtain criteria for properness by running conical $J$-flow (Section \ref{Twisted J-metrics and properness}), and to construct reference metrics in degenerate case (Section \ref{Approximate reference metrics}).

\bigskip

Now let us draw attention to the PDEs of the cscK metric.
Let $\theta$ be a smooth $(1,1)$-form in $C_1(X,D)$. The reference metric $\om_\theta$ and its smooth approximation is constructed in Section \ref{Approximate reference metrics}.
Let  $\chi_0$ be a smooth closed $(1,1)$ form, $R$ be a smooth function and $f$ be a function such that $e^{-f}\in L^{p_0}(\om_0)$ with $
\sup_X f=0$. 
We set $
\Theta=\theta-\chi_0.
$ and the twisted form to be
\begin{align*}
\chi:=\chi_0+i\p\bar\p f\geq 0\text{ with }e^{-f}\in L^{p_0}(\om_0)\text{ for some large }p_0>>1.
\end{align*}

We consider the scalar curvature equation
\begin{align*}
S(\om_{\varphi})=\tr_{\vphi}(Ric(\om_{\theta})-\Theta+i\p\bar\p f)+R.
\end{align*}
It is reduced to the following system,
\begin{align}\label{2nd pde}
F=\log\frac{\om_{\vphi}^n}{\om_{\theta}^n},\quad \tri_{\vphi} F=\tr_{\vphi}(\Theta-i\p\bar\p f)-R.
\end{align}

In Section \ref{a priori estimates}, we establish a priori estimates for \eqref{2nd pde} in both the cscK cone problem and the degenerate cases. The a priori estimates of Chen-Cheng rely on fixed smooth K\"ahler metric. In the conical setting, their a priori estimates can not be applied directly, since the a priori estimates are required to be obtained with respect to a K\"ahler cone metric. The K\"ahler cone metric generally does not have bounded geometry, which is one of the difficulties in the study of cscK cone metrics \cite{MR3405866,MR3761174,MR3968885,MR4020314}. In order to overcome these difficulties, roughly speaking, we fully make use of the approximation scheme, carefully track the constants dependences in the estimates, and adopt the ideas in the previous paper \cite{MR4020314}, that is to pick a "nice" reference metric such that it has the correct asymptotic behaviour we need. With these preparation in hands, we are able to perform the strategy in \cite{arXiv:1801.05907}, combining with various weighted estimates to obtain a priori estimates. The following estimates are obtained for the approximation solution of the cscK cone metrics, regarding of the approximate reference metric of $\om_\theta$. The precise statements are given in Section \ref{a priori estimates}. 

\begin{thm}\label{A priori estimates}
Suppose that $\vphi$ is a solution to \eqref{2nd pde} with the twisted term $\chi$. Then there is a constant $C$ such that 
\begin{align*}
\|\vphi\|_{\infty}, \quad \|F+f\|_{\infty}, \quad \sup_X \|\p (F+f)\|^2_\vphi, \quad \sup_X\|\tr_{\om_{\theta}}\om_\vphi\|_{p;\om_{\theta}}\leq C,
\end{align*} where $C$ depends on $\alpha_1,\alpha_\beta, n$, $\|\frac{\om_{\theta}^n}{\om_0^n}\|_{L^q(\om_0)}$ for some $q>1$ and the following quantities
	\begin{align*}
	E_\beta=\frac{1}{V}\int_X\log\frac{\om^n_{\vphi}}{\om^n_{\theta}}\om_{\vphi}^{n}, \quad\|e^{-f}\|_{L^{p_0}(\om_0)},\quad \|R\|_\infty, \quad \|\Theta\|_\infty,\quad\inf_XRic(\om_{\theta}).
	\end{align*}
	In which, $p_0$ is sufficiently large and depends on $n$ and $p$. 
	
	Furthermore, when $f=0$, there is a constant $C$ such that 
	\begin{align*}
	\|\vphi\|_\infty, \quad \|F\|_\infty, \quad \sup_X \|\p F\|^2_{\om_{\theta}}, \quad\sup_X \tr_{\om_{\theta}}\om_\vphi\leq C,
	\end{align*} where $C$ depends on the following quantities
	\begin{align*}
	E_\beta, \quad \|R\|_\infty, \quad \|\Theta\|_\infty,\quad\inf_XRic(\om_{\theta}),\quad \alpha_1,\quad \alpha_\beta,\quad n.
	\end{align*}
\end{thm}

In the degenerate case, we perform Tsuji's trick and deal with the weights coming from the degeneration. The corresponding estimates in the degenerate case are given in Section \ref{Linfty-estimates for degenerate metrics}, Section \ref{F lower bound for degenerate metrics} and Section \ref{W{2,p}-estimate for degenerate metrics}.

\begin{rem}
These estimates are also obtained for cscK cone metric by directly making use of the conical background metric as shown in \cite{arXiv:1805.04944}, which are considered to be the limit of \thmref{A priori estimates}.
\end{rem}

\bigskip

In Section \ref{Regularity and uniqueness of log twisted energy minimisers}, we will show the regularity and uniqueness of log $\chi$-twisted $K$-energy minimisers. 
We actually adapt the direct proof in \cite{arXiv:1801.00656} for the regularity of the $\chi$-twisted $K$-energy minimisers. Again, it is essentially obtained from the key a priori estimates of (conical and twisted) cscK equation. We remark that assuming the existence of a smooth one, the regularity of the $K$-energy minimisers is proved in \cite{arXiv:1602.03114}, which it is a result of the uniqueness property. We refer the readers to the related results, different approaches and conjectures on the $K$-energy minimisers in \cite{arXiv:1602.03114,MR3600039,MR3858468,MR3010550,MR3896025,arXiv:1701.06943,arXiv:1801.00656} and references therein). 

\bigskip

Section \ref{Existence implies properness and geodesic stability}
 has two parts.
The other direction of \thmref{cone properness conjecture} will be given in Section \ref{Existence implies properness}, that is 'existence $\implies$ properness'. In the context of cscK metrics, it was proved in \cite{MR1471884,MR3600039,arXiv:1602.03114}. In our conical setting, the new input in the proof is the uniqueness of cscK cone metrics, which follows from our previous work \cite{MR4020314}.


In Section \ref{Geodesic stability}, we prove an analogue of Donaldson's geodesic stability conjecture in the setting of cscK cone metric, that is detecting the existence of cscK cone metrics by the geodesic ray.

\begin{thm}[Log geodesic stability theorem]\label{cone geodesic conjecture}Assume that $\Om$ is a K\"ahler class and the cone angle satisfies $0<\beta\leq 1$. 
The following are all equivalent,
\begin{itemize}
\item there exists no constant scalar curvature K\"ahler cone metric;
\item either $Fut_\beta\neq0$, or there exists a potential $\vphi_0\in\mathcal E^1_0$ and an $d_1$-geodesic ray $\{\rho(t);0\leq t< \infty\}$ starting with $\vphi_0$ such that the log $K$-energy is non-increasing along $\rho(t)$;
\item either $Fut_\beta\neq0$, or for all potential $\vphi\in\mathcal E^1_0$, there exists an $d_1$-geodesic ray $\{\rho(t);0\leq t< \infty\}$ starting with $\vphi$ such that the log $K$-energy is non-increasing along $\rho(t)$.
\end{itemize}
In conclusion, $(M,\Om)$ admits a constant scalar curvature K\"ahler cone metric if and only if it is geodesic stable (see Definition \ref{geodesic stability}).
\end{thm}

\begin{rem}
We remark that when $Aut(X;D)$ is discrete, those results cover the important example in the K\"ahler-Einstein problem. That is when $X$ is a Fano manifold and $D$ is a smooth divisor in the linear system $|-\lambda K_X|$ for some $\lambda\in \mathbb Z^+$, there does not exist any holomorphic vector field tangential to $D$ \cite{MR2975584,MR3107540,MR3470713}. 
\end{rem}

\bigskip

\bigskip

The other goal of this article is to prove two fundamental properties of the cscK cone path in the conical continuity method, that is it is open (\thmref{openess main}) and it could be approximated by smooth metrics (\thmref{Approximation theorem}). 

It is expected that the existence of twisted cscK metric is equivalent to the existence of cscK cone metric, see \cite{MR3858468}.
The approximation scheme provides a positive answer to one direction of this problem, that is a cscK cone metric is approximated by a sequence of twisted cscK metrics. The strategy extends the counterpart of the K\"ahler-Einstein cone metrics in the canonical class \cite{MR3264766}. The proof of this theorem is a simple combination of Proposition \ref{twisted approximation thm}, \thmref{properclosedness} and \thmref{Existence implies properness discrete}.
\begin{thm}[Approximation theorem]\label{Approximation theorem}
	Assume that $\Om$ is a K\"ahler class and $Aut(X;D)$ is discrete. Assume that $C_1(L_D)\geq 0$ and $\om_{cscK}$ is a constant scalar curvature K\"ahler cone metric in $\Om$. 
	Then the cscK cone metric $\om_{cscK}$ has a smooth approximation, which is a family of the twisted cscK metrics in $\Om$.
\end{thm}

\bigskip

In the final Section \ref{CscK cone path}, we give another application of the log properness theorem to prove the openness theorem of the cscK cone path (see Definition \ref{cscK cone path defn}), that answers Question 2.5 in Chen \cite{MR3858468} on the deformation of the cscK cone path. Although it has been a direct corollary of the linear theory developed in \cite{MR4020314}, but we would like to revisit it by a geometric proof with the aid of the log properness theorem.
\begin{thm}[Openness theorem]\label{openess main}
	Assume that $Aut(X;D)$ is discrete and $C_1(L_D)> 0$.
The cscK cone path is open when the cone angle $\beta>0$. 
\end{thm}
The precise statement is given in \thmref{openess}.
This result should be compared with Donaldson's continuity path of K\"ahler-Einstein cone metrics in the resolution of YTD conjecture for Fano K\"ahler-Einstein metrics \cite{MR3264766,MR3472831}. Donaldson's continuity path deforms the cone angles of K\"ahler-Einstein cone metrics \cite{MR2975584}. The paralleling idea is to apply an analogue conical continuity method for the cscK problem, showing the existence of smooth cscK metrics by deforming cone angles of the conical ones. Actually, the continuity path of cscK cone metrics is the limiting process of Chen's path of twisted cscK metrics \cite{MR3858468}.
The conical continuity method for the cscK problem has been explored in \cite{Zheng,MR4020314,MR3968885,MR3857693}.

\bigskip

For further application as shown in Section \ref{Singular constant scalar curvature Kahler metrics}, we propose a question concerning singular cscK metrics in arbitrary class, with possible relation to the (log) Minimal Model Program.
\begin{ques}
In big cohomology class, is existence of singular cscK metrics equivalent to geometric stability?
\end{ques}
This question is also related to Question 1.14 in \cite{MR4020314}.

\bigskip
\noindent {\bf Acknowledgments:}
The work of K. Zheng has received funding from the European Union's Horizon 2020 research and innovation programme under the Marie Sk{\l}odowska-Curie grant agreement No. 703949, and was also partially supported by EPSRC grant number EP/K00865X/1. 
\section{Constant scalar curvature K\"ahler cone metrics}\label{Constant scalar curvature Kahler cone metrics}
In this section, we recall results on cone geodesics and cscK cone metrics from \cite{MR4020314}.
We assume that $0<\b\leq1$.
A \emph{K\"ahler cone metric} $\om$ of cone angle $2\pi\b$ along the divisor $D$,
is a smooth K\"ahler metric on the regular part $M:=X\setminus D$, and quasi-isometric to the cone flat metric, 
\begin{align}\label{flat cone}
\om_{cone}
&:=\frac{\sqrt{-1}}{2}\b^2|z^1|^{2(\b-1)}dz^1\wedge dz^{\bar 1}+\sum_{2\leq j\leq n}dz^j\wedge dz^{\bar j}\, 
\end{align}
in the coordinate chart $U_p$
centred at the point $p$ on $D$ with the holomorphic coordinate $\{z^1,\cdots,z^n\}$ and the local defining function $z^1$ of $D$.
We denote by $\mathcal H_\beta(\om_0)$ the space of all K\"ahler cone potentials $\vphi$, satisfying that $\om_\vphi:=\om_0+i\p\bar\p\vphi$ is a K\"ahler cone metric.

\subsection{H\"older spaces and geometrically polyhomogeneous}

We assume that the H\"older exponent $\alpha$ satisfies 
\begin{align*} 
\a\b<1-\b.
\end{align*} 
The H\"older spaces $C^{k,\a,\b}$ were introduced in Donaldson \cite{MR2975584} for $k=0,1,2$. They were defined in the geometric way and played important role in constructing K\"ahler-Einstein metrics on Fano manifolds.

Recall that $U_p$ is a coordinate chart intersecting with the divisor. A function $u(z): U_p\rightarrow \mathbb R$ 
is said to be $C^{0,\a,\b}$, if we set  $v(|z^1|^{\b-1}z^1,z^2,\cdots,z^n):=u(z^1,z^2,\cdots,z^n)$ and $v$ is a $C^{0,\a}$ H\"older function. The space $C^{0,\a,\b}_{0}$ contains all functions $u\in C^{0,\a,\b}$ such that $$u(0, z^2,\cdots, z^n) =0.$$ A $(1,1)$-form $\chi$ is said to be $C^{0,\a,\b}$, if for any $2\leq i,j\leq n$, the following items satisfy
\begin{equation*}
\left\{
\begin{aligned}
&\chi(\frac{\p}{\p z^i},\frac{\p}{\p z^{\bar j}})\in C^{0,\a,\b}, &|z^1|^{2-2\b}\chi(\frac{\p}{\p z^1},\frac{\p}{\p z^{\bar 1}})\in C^{0,\a,\b},\\
&|z^1|^{1-\b}\chi(\frac{\p}{\p z^1},\frac{\p}{\p z^{\bar j}})\in C_0^{0,\a,\b}, &|z^1|^{1-\b}\chi(\frac{\p}{\p z^i},\frac{\p}{\p z^{\bar 1}}) \in C_0^{0,\a,\b}.
\end{aligned}
\right.
\end{equation*} 
The H\"older space $C^{2,\a,\b}$ is defined to be
\begin{align*}
C^{2,\a,\b}
:= \{u\; \vert \;  u, \p u, \sqrt{-1}\p\bar\p u\in C^{0,\a,\b}\}\; .
\end{align*}

For $k=3,4,\cdots$, higher order H\"older spaces depend on the background metrics. The precise definitions were introduced in \cite{MR4020314} for all $0<\b\leq 1$. In particular, when cone angle $0<\b< \frac{1}{2}$, the background metric has bounded geometry and then all derivatives are H\"older, as shown in \cite{MR3968885}. In general, not all derivatives are H\"older continuous, some derivatives will blow up, due to the geometric nature. 
We say a K\"ahler cone potential $\vphi$ is geometrically polyhomogeneous, if it has higher order geometric estimates and thus the expansion formula. The geometrical polyhomogeneity is referring to geometric estimates, which are better than the one called polyhomogeneity from PDE point of view, This is detailed in the author's previous articles, \cite{MR3911741} for general complex Monge-Amp\`ere equation with cone singularities and \cite{MR4020314} for cscK cone metrics.

We recall the definition of $C_{\mathbf {w}}^{3,\a,\b}$ from \cite{MR4020314}.
\begin{defn}[\cite{MR4020314}]\label{cpw3ab}
	A function $\vphi$ belongs to $C_{\mathbf {w}}^{3,\a,\b}$ with the H\"older exponent $\alpha$ satisfying	
	$
	\a\b<1-\b,
	$
	if it holds in the normal cone chart,
	\begin{itemize}
		\item $\vphi\in C^{2,\a,\b}$;
		\item the first derivatives of the corresponding metric $g$ satisfy for any $2\leq i,k,l\leq n$, the following items are $C^{0,\a,\b}$,	
		\begin{align*}
		\frac{\p g_{ k\bar l}}{\p z^i},\quad
		|z^1|^{1-\b}\frac{\p g_{k\bar{1}}}{\p z^{ i}},\quad
		|z^1|^{1-\b}\frac{\p g_{ 1\bar l}}{\p z^i},\quad
		|z^1|^{2-2\b}\frac{\p g_{1\bar 1}}{\p z^i};
		\end{align*}
		\item the following terms are $O(|z^1|^{-\kappa})$ with $\kappa=\b-\a\b$,	
		\begin{align*}
		|z^1|^{1-\b}\frac{\p g_{ k\bar l}}{\p z^1},\qquad
		|z^1|^{2-2\b}\frac{\p g_{k\bar{1}}}{\p z^{  1}},\quad
		|z^1|^{2-2\b}\frac{\p g_{1\bar l}}{\p z^1},\qquad
		|z^1|^{3-3\b}\frac{\p g_{1\bar 1}}{\p z^1} .
		\end{align*}
		Here, we set the corresponding metric $g$ to be the K\"ahler cone
metric given by $\vphi$, that is $g_{k\bar l}=g_{0k\bar l}+\vphi_{k\bar l}$.
	\end{itemize}
	When the coordinate chart does not intersect the divisor, all definitions are in the classical way.
\end{defn}


\subsection{Geodesics in the space of K\"ahler cone metrics}\label{Cone geodesics}
The $L^2$ metric is extended in the space $\mathcal H_\beta$. It is shown in \cite{MR3405866} that the geodesic equation is a degenerate complex Monge-Amp\`ere equation with cone singularities in the product manifold $\mathfrak X=X\times [0,1]\times S^1$. The existence, uniqueness and regularity of cone geodesic with cone angle in the whole interval $0<\beta\leq 1$ are completely proved in \cite{MR4020314}, by solving the boundary value problem of the approximation equation 
\begin{align}\label{per equ a}
\left\{
\begin{array}{ll}
\Om_{\Psi}^{n+1}
=\tau \cdot \Om_{\mathbf b}^{n+1} &\text{ in }\mathfrak{M}=M\times [0,1]\times S^1\; ,\\
\Psi=\Psi_0 & \text{ on }\p\mathfrak{X} \;.
\end{array}
\right.
\end{align}
The notions $\Om_\psi$, $\Psi_0$ represent the pull-back of the K\"ahler cone metric $\om_{\vphi}$ and the boundary values $\vphi_0,\vphi_1$ to the product manifold $\mathfrak X=X\times [0,1]\times S^1$ under the natural projection $\pi:\mathfrak X\rightarrow X$. The K\"ahler cone metric $\Om_{\mathbf b}$ is the background metric constructed on $\mathfrak X$ and satisfies appropriate curvature conditions.

We define the \textit{generalised cone geodesic} $\{\vphi(t), 0\leq t\leq 1\}$ to be the limit of solutions to \eqref{per equ a}
as $\tau\rightarrow 0$ under the $C_{\tri}^{\b}$-norm,
\begin{align*}
||\vphi||_{C_{\tri}^{\b}}=\sup_{(z,t)\in \mathfrak X}\{|\vphi|+|\p_t\vphi|+|\p_z\vphi|_{\om}+|\p_z{\p_{\bar z}}\vphi|_{\om}\}.
\end{align*}
Similarly, we define the \textit{$C_{\bf w}^{1,1,\b}$ cone geodesic} to be the limit of solutions to \eqref{per equ a}
as $\tau\rightarrow 0$ under the $C_{\bf w}^{1,1,\b}$-norm with the weight $\kappa=\beta-\a\b$,
	\begin{align*}
||\vphi||_{C_{\bf w}^{1,1,\b}}=||\vphi||_{C_{\tri}^{\b}}+\sup_{ \mathfrak X}\{\sum_{2\leq i\leq n}|\frac{\p^2\vphi}{\p z^{i}\p t} |_{\Om}+|s|^{\kappa}|\frac{\p^2\vphi}{\p z^{1}\p t} |_{\Om}+|s|^{2\kappa} |\frac{\p^2\vphi}{\p t^2} |\}.
\end{align*}

\begin{thm}[Cone geodesic Theorem 1.1 in \cite{MR4020314}]\label{geo existence}
	Suppose that $0<\b\leq1$ and $\{\om_i=\om_0+i\p\bar\p\vphi_{i}, i=0,1\}$ are two K\"ahler cone metrics in $\mathcal H_{\b}$. Then there exists a unique generalised cone geodesic connecting them. 
Furhthermore, if $\{\vphi_{i}, i=0,1\}$ are two $C_{\mathbf w}^{3,\a,\b}$ K\"ahler cone metrics. Then the generalised cone geodesic is $C_{\bf w}^{1,1,\b}$.
	
\end{thm}

We remark that the half angle cone geodesic ($0<\beta<\frac{1}{2}$) has better regularity, as shown in \cite{MR3405866}. 




\subsection{CscK cone metrics: definition and regularity}\label{CscK cone metrics}

We recall the definition of the cscK cone from \cite{MR4020314}. There are weaker definitions appeared in \cite{Zheng,MR3803119,MR3858468,MR2975584}.
\subsubsection{Reference metric}\label{Background metric}
We denote 
\begin{align*}
C_1(X,D):=C_1(X)-(1-\b)C_1(L_D).
\end{align*}
In \cite{MR4020314}, the reference K\"ahler cone metric $\om_\theta$ was obtained by solving the following equation of $\om_\theta:=\om_0+i\p\bar\p \vphi_\theta$
\begin{align}\label{Rictheta}
Ric(\om_\theta)=\theta+2\pi (1-\b)[D].
\end{align}
In which, $\theta\in C_1(X,D)$ is a smooth $(1,1)$-form. 

We let $s$ be the defining section of $D$ and $h$ be a smooth Hermitian metric on $L_D$. We set $\Theta_D$ to be 
\begin{align*}
\Theta_D=-i\p\bar\p\log h.
\end{align*}
It is a smooth $(1,1)$-form in $C_1(L_D)$.
According to the Poincar\'e-Lelong equation, the divisor term  is given by
\begin{align*}
2\pi[D]=i\p\bar\p\log|s|_h^2-i\p\bar\p\log h=i\p\bar\p\log|s|^2.
\end{align*}

By the cohomology condition, we set $h_0$ be a smooth function satisfying
\begin{align}\label{h0}
Ric(\om_0)=\theta+(1-\beta)\Theta_D+i\p\bar\p h_0.
\end{align}
The K\"ahler cone potential $\vphi_\theta$ (K\"ahler potential of a K\"ahler cone metric) satisfies the complex Monge-Amp\`ere equation with cone singularities
\begin{align}\label{theta}
\frac{\om^n_\theta}{\om^n_0}=\frac{e^{h_0}}{|s|_h^{2-2\b}}.
\end{align}
Here, the normalisation condition of $h_0$ is 
\begin{align}\label{thetanormalise}
V=\int_M\om^n_\theta=\int_M\frac{e^{h_0}}{|s|_h^{2-2\b}}\om^n_0.
\end{align}
\begin{lem}\label{reference metric cone}
There exists a unique solution $\vphi_\theta\in C^{2,\a,\b}$ and $\vphi_\theta$ is also geometrically polyhomogeneous, Theorem 1.1 in \cite{MR3911741} (see also \cite{MR3761174} for related references).
\end{lem}


\subsubsection{Definition of cscK cone metric}\label{Definition of cscK cone metrics}

\begin{defn}[Definition 3.1. in \cite{MR4020314}]\label{csckconemetricdefn}A cscK cone metric
	$$\om_{cscK}:=\om_0+i\p\bar\p\vphi_{cscK}$$ is a K\"ahler cone metric with K\"ahler potential $\vphi_{cscK}\in C^{2,\a,\b}$ and satisfying the equations
\begin{align}
\label{2nd equ}
&\frac{\om_{cscK}^n}{\om_\theta^n}=e^F,\\
\label{Fcsck}
&\tri_{\om_{cscK}} F=\tr_{\om_{cscK}}\theta-\underline S_\b.
\end{align}
\end{defn}
	
The constant $\underline S_\b$ is independent of the choice of $\vphi_{cscK}$ (Lemma 3.1 in \cite{MR4020314}) and equals to the topological constant
\begin{align*}
\underline S_\b=\frac{C_1(X,D)[\om_0]^{n-1}}{[\om_0]^n}.
\end{align*}
By substitution into \eqref{Fcsck} with the equation of $\om_\theta$ \eqref{theta}, we have
\begin{align}\label{2nd equ smooth}
\frac{\om_{cscK}^n}{\om_0^n}=\frac{e^{F+h_0}}{|s|_h^{2-2\b}}.
\end{align}

\subsubsection{Regularity of cscK cone metrics: geometrical polyhomogeneity}\label{Regularity of cscK cone metrics: geometrical polyhomogeneity}
\begin{defn}[\cite{MR4020314}]\label{4thspace}
The 4th order H\"older space is defined as
\begin{align*}
D_{\bf w}^{4,\a,\b}(\om_\theta)=\{\vphi\in C^{2,\a,\b}\vert \log\frac{\om_\vphi^n}{\om^n_\theta}\in C^{2,\a,\b}\}.
\end{align*}
The corresponding linearisation space at $\om$ is
\begin{align*}
C_{\bf w}^{4,\a,\b}(\om)=\{u\in C^{2,\a,\b}\vert \tri_{\om} u\in C^{2,\a,\b}\}.
\end{align*}
\end{defn}

\begin{thm}[Geometrical polyhomogeneity Theorem 1.2 in \cite{MR4020314}]\label{Geometric asymptotic}
	Assume that $0<\b\leq1$ and the H\"older exponent $\alpha$ satisfies 
$\a\b<1-\b.$ Suppose that $\omega_\vphi=\om_0+i\p\bar\p\vphi$ is a constant scalar curvature K\"ahler cone metric. Then $\vphi$ is $C^{3,\a,\b}_{\bf w}\cap D_{\bf w}^{4,\a,\b}(\om_\theta)$. 
Moreover, $\vphi$ is geometrically polyhomogeneous (Theorem 4.5 in \cite{MR4020314}).
\end{thm}
\subsection{Uniqueness of cscK cone metrics and automorphism group}
Recall that $$G:=Aut(X;D)$$ is the identity component of the group of holomorphic automorphisms of $X$ which fix the divisor $D$. We also set $\mathfrak{h}(X;D)$ to be the space of all holomorphic vector fields tangential to the divisor and $\mathfrak{h}'(X;D)$ to be the complexification of a Lie algebra consisting of Killing vector fields of $X$ tangential to $D$. The reductivity theorem of $Aut(X;D)$ for cscK cone metric was proved in \cite{MR4020314}, see also \cite{MR3968885,MR3761174,MR3264767}. 

\begin{thm}[Reductivity Theorem 1.4 in \cite{MR4020314}]\label{preceisereductivity}Suppose $\om$ is a cscK cone metric. 
	Then there exists a one-to-one correspondence between $\mathfrak{h}'(X;D)$ and the kernel of ${\mathbb{L}\mathrm{ic}}_{\om}$. 
	
	Precisely speaking, 
	the Lie algebra $\mathfrak{h} (X;D)$ has a direct sum decomposition:
	\begin{equation}
	\mathfrak{h}(X;D) = \mathfrak{a}(X;D) \oplus \mathfrak{h'} (X;D),
	\end{equation}
	where $\mathfrak{a}(X;D)$ is the complex Lie subalgebra of $\mathfrak{h}(X;D)$ consisting of all parallel holomorphic vector fields tangential to $D$,
	and $\mathfrak{h'}(X;D)$ is the ideal of $\mathfrak{h}(X;D)$ consisting of the image under $grad_g$ of the kernel of $\cD$ operator. The operator $grad_g$ is defined to be $grad_g(u)= \uparrow^{\omega}\bar\p u=g^{i\bar j}\frac{\p u}{\p z^{\bar j}}\frac{\p}{\p z^i}$.
	
	Furthermore $\mathfrak{h}'(X;D)$ is the complexification of a Lie algebra consisting of Killing vector fields of $X$ tangential to $D$. 
	In particular $\mathfrak{h}'(X;D)$ is reductive. 
	Moreover, $\mathfrak{h}(X;D)$ is reductive.
\end{thm}

\begin{thm}[Uniqueness Theorem 1.6 in \cite{MR4020314}]\label{Uniqueness}
	The constant scalar curvature K\"ahler cone metric is unique up to automorphisms. 
\end{thm}


\subsection{Log $K$-energy and convexity}\label{Log K-energy and convexity}


The log entropy and log $K$ energy are well defined in $\mathcal H_\beta$. Their properties including continuity and convexity were shown in \cite{MR4020314,MR3803119}.

\subsubsection{Entropy}
\begin{defn}
The log entropy on $\mathcal H_\b$ is defined to be
\begin{align}\label{entropy}
E_\beta(\vphi):=\frac{1}{V}\int_M\log\frac{\om^n_\vphi}{\om_0^n|s|_h^{2\b-2}e^{h_0}}\om_\vphi^{n},
\end{align}
where, $s$ is the defining section of $D$, $h_0$ is a smooth function and $h$ is a Hermitian metric on the associated line bundle $L_D$.
By substitution with \eqref{theta}, the log entropy could be rewritten in term of the reference metric $\om_\theta$,
\begin{align*}
E_\beta(\vphi)=\frac{1}{V}\int_M\log\frac{\om^n_\vphi}{\om_\theta^n}\om_\vphi^{n}.
\end{align*}
\end{defn}
\subsubsection{$J$-functionals}
The functional associated to the Monge-Amper\`e operator os 
\begin{align*}
D_{\om_0}(\vphi)
&:=\frac{1}{V}\frac{1}{n+1}\sum_{j=0}^{n}\int_{M}\vphi\om_0^{j}\wedge\om_{\varphi}^{n-j}.
\end{align*}
The computation shows that
\begin{itemize}
\item The first variation of $D_{\om_0}$ is
\begin{align*}
\p_t D_{\om_0}(\vphi)=\frac{1}{V}\int_M \p_t\vphi \om_\vphi^n.
\end{align*}
\item $D_{\om_0}$ satisfies the cocycle condition. It means that if we write $D(\om_0,\om_\vphi):=D_{\om_0}(\vphi)$ for $\om_\vphi=\om_0+i\p\bar\p\vphi$, then 
\begin{align*}
D(\om_0,\om_\vphi)=D(\om_0,\om_\psi)+D(\om_\psi,\om_\vphi).
\end{align*}
\end{itemize}

Let $\chi$ to be a closed $(1,1)$-form $\chi$. The log $J_\chi$-functional is defined to be
\begin{align*}
J_{\chi}(\vphi)&:=j_{\chi}(\vphi)-\ul{\chi}\cdot D_{\om_0}(\vphi) 
\end{align*}
with
\begin{align*}
j_{\chi}(\vphi)&:=\frac{1}{V}\int_{M}\vphi
\sum_{j=0}^{n-1}\om_0^{j}\wedge
\om_\vphi^{n-1-j}\wedge \chi,\quad \ul{\chi}
:=\frac{n\int_X\chi\wedge \om_0^{n-1}}{V}
\end{align*}
It is direct to see that
\begin{itemize}
\item The first variation of $J_{\chi}$ is
\begin{align}\label{derivatives of J chi}
\p_t J_{\chi}(\vphi)=\frac{1}{V}\int_{M}\p_t\vphi
(n\chi\wedge \om_\vphi^{n-1}-\underline\chi\om^n_\vphi).
\end{align}
\item 
\begin{align*}
j_{i\p\bar\p f}(\vphi)
=\frac{1}{V}\int_{M}f(\om^n_\vphi-\om^n_0).
\end{align*}
\end{itemize}

\subsubsection{Log $K$-energy}
\begin{defn}\label{logKenergy}
The log $K$-energy is defined on $\mathcal H_\b$ as
\begin{align}\label{log K energy}
\nu_\beta(\vphi)
&:=E_\beta(\vphi)
+J_{-\theta}(\vphi)+\frac{1}{V}\int_M (\mathfrak h+h_0)\om_0^n,
\end{align}
where, we denote 
\begin{align}\label{mathfrakh}
J_{-\theta}(\vphi)&:=j_{-\theta}(\vphi)+\ul{S}_\beta\cdot D(\vphi),\quad \mathfrak h:=-(1-\b)\log |s|_h^2
\end{align}
with 
\begin{align*}
j_{-\theta}(\vphi)&:=-\frac{1}{V}\int_{M}\vphi
\sum_{j=0}^{n-1}\om_0^{j}\wedge
\om_\vphi^{n-1-j}\wedge \theta.
\end{align*}
\end{defn}

When $\beta=1$, the log $K$-energy coincides with Mabuchi $K$-energy, which will be denoted by,
\begin{align}\label{Mabuchienergy}
\nu_1(\vphi)
=E_1(\vphi)+
j_{-Ric(\om_0)}(\vphi)+\ul{S}_1\cdot D(\vphi),
\end{align}
with the entropy term
\begin{align*}
E_1(\vphi)=\frac{1}{V}\int_M\log\frac{\om^n_\vphi}{\om_0^n}\om_\vphi^{n}.
\end{align*}


\begin{lem}[Convex and continuity of the log $K$-energy Proposition 3.10 in \cite{MR4020314}]\label{c11convexityKenergy}
	The log $K$-energy is continuous and convex along the $C^{\b}_{\tri}$ generalised cone geodesic. 
\end{lem}

\subsubsection{$I$ and $J$-functionals}

We will also use the log $I$ and log $J$-functional on $\mathcal H_\b$ as following
\begin{align}\label{IJ functionals}
J^A_{\om_0}(\vphi)
:=-D_{\om_0}(\vphi)+\frac{1}{V}\int_{M}\vphi\om_0^{n};\quad
I^A_{\om_0}(\vphi)
:=\frac{1}{V}\int_{M}\vphi(\om_0^{n}-\om_\vphi^{n}).
\end{align}
The relation between $I$ and $J$ is
\begin{align}\label{IJ equivalent}
I^A_{\om_0}\leq (n+1)J^A_{\om_0}\leq n I^A_{\om_0}.
\end{align}
The difference of their directives are
\begin{align}\label{derivatives IJ}
\p_t(I^A_{\om_0}(\vphi)-J^A_{\om_0}(\vphi))=\frac{n}{V}\int_{M}\p_t\vphi
(\om_0\wedge \om_\vphi^{n-1}-\om^n_\vphi).
\end{align}
So we see that 
\begin{align}\label{J=I-J}
J_{\om_0}(\vphi)=I^A_{\om_0}(\vphi)-J^A_{\om_0}(\vphi).
\end{align}

We denote the normalised space $\mathcal H_{\beta,0}=\{\vphi\in \mathcal H_{\beta}\vert D(\vphi)=0\}$. Let
	\begin{align*}
I_1(\psi,\vphi)=\frac{1}{V}\int_M|\psi-\vphi|\om_\psi^n+\frac{1}{V}\int_M|\psi-\vphi|\om_\vphi^n.
\end{align*} 
We will need the following inequalities from Theorem 3 in \cite{MR3406499} 
	\begin{align}\label{i1d1}
\frac{1}{2}d_1(\psi,\vphi)\leq I_1(\psi,\vphi)\leq C(n,p) d_1(\psi,\vphi).
\end{align} 
Here, the constant $C(n,p)$ only depends on $n$ and $p$.

We also need the following modification of  Lemma 4.4 in \cite{arXiv:1801.00656}. 
\begin{lem}\label{d_1andJ}Assume that $\chi_0$ is a smooth closed $(1,1)$-form and $\chi=\chi_0+i\p\bar\p f$ with $f\in C^{1,1,\beta}$. Then there exists a constant $C(n)$ such that for any $\vphi\in \mathcal H_{\beta,0}$, we have
	\begin{align*}
	|J_\chi(\vphi)|\leq C(\max_X \|\chi\|_{\om_0} d_1(0,\vphi)+\|f\|_\infty).
	\end{align*}
\end{lem}
\begin{proof}
We consider
\begin{align*}
j_{\chi}(\vphi)=\frac{1}{V}\int_{M}\vphi
\sum_{j=0}^{n-1}\om_0^{j}\wedge
\om_\vphi^{n-1-j}\wedge (\chi_0+i\p\bar\p f).
\end{align*}
We apply the elliptic estimate to $\tri_{\om_0}\vphi\geq 0$ we have 
\begin{align*}
\sup_X\vphi\leq \frac{1}{V_0}\int_X\vphi\om_0^n+C.
\end{align*}
Then we get be the definition if the $d_1$-distance,
\begin{align}\label{sup vphi d1}
\sup_X\vphi\leq C_1d_1(0,\vphi)+C_2.
\end{align}
Thus we have
\begin{align*}
&j_{\chi_0}(\vphi)\leq C(n)\max_X \|\chi_0\|_{\om_0} d_1(0,\vphi),\\
&D_{\om_0}(\vphi)\leq C(n) d_1(0,\vphi).
\end{align*}
After integration by parts, the second part becomes,
\begin{align*}
j_{i\p\bar\p f}(\vphi)=\frac{1}{V}\int_{M}f(\om^n_\vphi-\om^n_0)\leq 2\|f\|_\infty.
\end{align*}
Also,
\begin{align*}
\ul{\chi}
=\frac{\int_Xn(\chi_0+i\p\bar\p f)\wedge \om_0^{n-1}}{V}=\ul{\chi_0}\leq \|\chi_0\|_{\om_0}.
\end{align*}
Thus the lemma is proved.
\end{proof}

\section{Properness implies existence}\label{Properness conjecture: automorphism is trivial}

In this section, we prove one direction of the \thmref{cone properness conjecture}, that is properness of the log $K$-energy implies the existence of the cscK cone metric. We would also prove this theorem for the twisted cscK cone metrics, which will be used in Section \ref{Regularity of minimisers}.
\begin{defn}\label{twisted term}
We set the twisted term $\chi$ to be a smooth non-negative closed $(1,1)$-form.
\end{defn}

\begin{defn}
	We say $\om_{\vphi}=\om_0+i\p\bar\p\vphi$ is a $\chi$-twisted cscK cone metric, if it satisfies the following equations in terms of the twisted term $\chi$,
	\begin{equation}\label{twisted 2nd equ defn}
\left\{
\begin{aligned}
	&F=\log\frac{\om_{\vphi}^n}{\om_\theta^n},\\
	&\tri_{\vphi} F=\tr_{\vphi}(\theta-\chi)-(\underline S_\b-\underline{\chi}).
\end{aligned}
\right.
\end{equation} 
\end{defn}
According to this definition, the scalar curvature of a $\chi$-twisted cscK cone metric satisfies
\begin{align*}
S_\vphi=\tr_{\vphi}\chi+\underline S_\b-\underline{\chi}.
\end{align*}
\begin{defn} \label{log twisted K energy}
The log $\chi$-twisted $K$-energy $\nu_{\beta,\chi}$ is defined as following,
\begin{align*}
\nu_{\beta,\chi}(\vphi)
=\nu_\beta(\vphi)+J_\chi(\vphi).
\end{align*}
Here $\nu_\beta(\vphi)$ is the log $K$-energy in Definition \ref{logKenergy}.
\end{defn}

We prove the convexity of the log $J_\chi$-functional.
\begin{lem}\label{Jchiconvex}
	Assume $\{\vphi(t),0\leq t\leq 1\}$ is a $C_{\bf w}^{1,1,\b}$-cone geodesic and $\chi>0$ is smooth closed $(1,1)$-form. Then the log $J_\chi$-functional is strictly convex along $\vphi(t)$, otherwise $\vphi$ is a constant geodesic.
\end{lem}
\begin{proof}
The log $J_\chi$-functional is continuous along the $C_{\bf w}^{1,1,\b}$-cone geodesic by the cocycle property of $j_\chi$ and $D$, c.f. Lemma 3.7 \cite{MR4020314}.
So it is sufficient to prove the strict convexity of  $J_\chi$ along the $C_{\bf w}^{1,1,\b}$-cone geodesic in the distributional sense, using the strict positivities of $\chi$. We let $\eta$ be a smooth non-negative cut-off function supported in the interior of $[0,1]$. Recall $\Om_\Psi$ is the pull back of the geodesic $\vphi(t)$ to the product manifold $\mathfrak X=X\times [0,1]\times S^1$.  Since $\Psi$ is a psh-function in $\mathfrak X$, we approximate it by the smooth decreasing sequence $\Psi_s$ in the interior of $\mathfrak X$ and still denote by $\vphi_s$ the restriction of $\Psi_s$ on each slice $X\times \{s\}\times S^1$, c.f \eqref{per equ a}. It is direct to compute the second order derivative of $J_\chi$ along the smooth approximation $\vphi_s(t)$,
\begin{align*}
\p^2_t J_\chi(\vphi_s) = \frac{1}{V}\int_M[ (\tr_{\vphi_s}\chi-\underline\chi) \Om^{n+1}_{\Psi_s}+\chi(\p\frac{\p{\vphi}_s}{\p t},\p\frac{\p{\vphi}_s}{\p t})\om^n_{\vphi_s}].
\end{align*}
Then integrating over $[0,1]$, we have
\begin{align*}
\int_0^1 \p^2_t \eta\cdot  J_\chi(\vphi_s) dt  =\int_0^1\eta \cdot  \p^2_t J_\chi(\vphi_s) dt  =  \frac{1}{V}\int_0^1\eta\int_M\chi\wedge \Om^{n}_{\Psi_s}dt.
\end{align*}
The log $J_\chi$-functional is well defined along the $C^{\b}_{\tri}$-cone geodesic and $\p^2_t J_\chi$ is well defined along the $C_{\bf w}^{1,1,\b}$-cone geodesic.
Therefore, after taking $s\rightarrow 0$, we have by Lebesgue's dominated convergence theorem,
\begin{align*}
\int_0^1 \p^2_t \eta\cdot  J_\chi(\vphi) dt  =\int_0^1\eta \cdot  \p^2_t J_\chi(\vphi) dt  .
\end{align*}
The strict convexity in the distributional sense follows from this identity and $\chi>0$. 

\end{proof}

\begin{lem}\label{uniquechipositive}
	Assume $\chi> 0$ is smooth closed $(1,1)$-form.
	The $\chi$-twisted cscK cone metric is unique. 
\end{lem}
\begin{proof}
	When $\chi>0$, automorphism action is not involved. The proof of uniqueness is direct as following. Since $\chi$ is smooth, from \eqref{twisted 2nd equ defn} we know $F\in C^{2,\alpha,\beta}$. Then the K\"ahler cone potential of the $\chi$-twisted cscK cone metric is $C^{3,\a,\b}_{\bf w}\cap D_{\bf w}^{4,\a,\b}(\om_\theta)$, due to \thmref{Geometric asymptotic}. We connect two $\chi$-twisted cscK cone metrics by the $C_{\bf w}^{1,1,\b}$-cone geodesic (\thmref{geo existence}). Then the uniqueness required is obtained from applying the convexity of the log $K$-energy (\lemref{c11convexityKenergy}) and the strict convexity of the log $J_\chi$-functional (\lemref{Jchiconvex}) to the log $\chi$-twisted $K$-energy.
	\end{proof}

\begin{lem}\label{globalminimiser}
	Assume $\chi\geq 0$ is smooth closed $(1,1)$-form.
	The $\chi$-twisted cscK cone metrics are the global minimiser of the log $\chi$-twisted $K$-energy in $\mathcal H_\beta$. 
\end{lem}
\begin{proof}
When $\chi=0$, the uniqueness of the cscK cone metric (\thmref{Uniqueness}) and the convexity \lemref{c11convexityKenergy} implies the cscK cone metric is the global minimiser in $\mathcal H_\beta$. When $\chi>0$, we apply \lemref{uniquechipositive} instead.

\end{proof}
\subsection{Approximation of twisted cscK cone metrics}
The approximation scheme is important in the proof of Yau-Tian-Donaldson conjecture on Fano manifold \cite{MR3264766}.
An approximation of the twisted cscK cone metric was given in \cite{MR4020314}. In this section, we will apply \cite{arXiv:1801.00656} to improve such approximation.
\subsubsection{Approximation of the reference metric $\om_\theta$}\label{Approximation of the reference metric}
The background metric $\om_\theta$ has a smooth approximation (Section 3.2 \cite{MR4020314}).
We solve smooth $\om_{\theta_\eps}:=\om_0+i\p\bar\p \vphi_{\theta_\eps}$ from the following approximation equation with $\eps\in (0,1]$,
\begin{align}\label{thetaeps}
\frac{\om^n_{\theta_\eps}}{\om^n_0}=\frac{e^{h_0+c}}{(|s|^2_h+\eps)^{1-\beta}}.
\end{align}
In order to normalise the volume $\int_M \om^n_{\theta_\eps}=\int_M\frac{e^{h_0+c}}{(|s|^2_h+\eps)^{1-\beta}} \om^n_0=V$, we use \eqref{thetanormalise} and define the normalisation constant $c$ to be
\begin{align}\label{thetaeps normalisation}
e^c=\frac{\int_Me^{h_0}|s|_h^{2(\beta-1)}\om^n_0}{\int_Me^{h_0}(|s|^2_h+\eps)^{\beta-1}\om^n_0}.
\end{align} The constant $e^c$ is bounded as
\begin{align}\label{ec}
1\leq e^c\leq 
\frac{\int_Me^{h_0}|s|_h^{2(\beta-1)}\om^n_0}{\int_Me^{h_0}(|s|^2_h+1)^{\beta-1}\om^n_0}.
\end{align}
\begin{lem}[Section 3.2 \cite{MR4020314}]\label{thetaepstheta}
Assume that $\om_{\theta_\eps}$ is the approximate reference metric satisfying \eqref{thetaeps}. Then for any $\eps\in(0,1]$, we have
\begin{align*}
Ric(\om_{\theta_\eps})
\geq\tilde \theta:=\theta+\min\{(1-\beta)\Theta_D,0\}.
\end{align*}
\end{lem}
\begin{proof}
Using  \eqref{theta}, we compute that
\begin{align}\label{Ricomthetaeps}
Ric(\om_{\theta_\eps})&=Ric(\om_0)-i\p\bar\p h_0+(1-\beta)i\p\bar\p\log(|s|^2_h+\eps)\\
&=\theta+(1-\beta)\Theta_D+(1-\beta)i\p\bar\p\log(|s|^2_h+\eps)\nonumber\\
&\geq\theta+(1-\beta)\frac{\eps}{|s|^2_h+\eps}\Theta_D.\nonumber
\end{align}
The conclusion follows from the following inequality on $M$, 
\begin{align*}
i\p\bar\p\log(|s|^2_h+\eps)\geq
\frac{|s|^2_h}{|s|^2_h+\eps}i\p\bar\p\log |s|^2_h= -\frac{|s|^2_h}{|s|^2_h+\eps}\Theta_D.
\end{align*}
When $\Theta_D\geq 0$, we have $Ric(\om_{\theta_\eps})\geq \theta$. Otherwise, we have the lower bound $Ric(\om_{\theta_\eps})\geq \tilde \theta$.
\end{proof}
\subsubsection{Approximation of the log $\chi$-twisted $K$-energy}
We approximate the log $\chi$-twisted $K$-energy $\nu_{\beta,\chi}$ by the twisted $K$-energy with respect to the approximate reference metric $\om_{\theta_\eps}$.
\begin{defn}
The approximate log $\chi$-twisted $K$-energy is defined to be.
\begin{align}\label{Approximation of the log twisted K energy}
\nu^\eps_{\beta,\chi}(\vphi)
&:=\frac{1}{V}\int_M\log\frac{\om^n_\vphi}{\om^n_{\theta_\eps}}\om_\vphi^{n}
+J_{-\theta}(\vphi)+J_{\chi}(\vphi)\\
&+\frac{1}{V}\int_M [-(1-\b)\log(|s|_h^2+\eps)+h_0]\om_0^n+c\nonumber.
\end{align}
\end{defn}
\begin{lem}\label{approximate entropy and entropy lemma}
		We denote 
	\begin{align*}
	E_\beta^\eps(\vphi)=\frac{1}{V}\int_M\log\frac{\om^n_{\vphi}}{\om^n_{\theta_\eps}}\om_{\vphi}^{n}.
	\end{align*} Then we have
	\begin{align}\label{approximate entropy and entropy}
E_\beta^\eps(\vphi)
=E_\beta(\vphi)-c+\frac{1-\beta}{V}\int_M[\log (|s|^2_h+\eps)-\log|s|_h^2 ]\om_{\vphi}^{n}.
\end{align}	
\end{lem}
\begin{proof}
Comparing $E_\beta^\eps(\vphi)$ with the log entropy \eqref{entropy} and making use of \eqref{thetaeps},
we have
	\begin{align*}
	E_\beta^\eps(\vphi)&=
	E_\beta(\vphi)-\frac{1}{V}\int_M[c+\log|s|_h^{2-2\beta}+\log (|s|^2_h+\eps)^{\beta-1} ]\om_{\vphi}^{n}\\
	&=E_\beta(\vphi)-c-(1-\beta)\frac{1}{V}\int_M[\log|s|_h^2-\log (|s|^2_h+\eps) ]\om_{\vphi}^{n}.
	\end{align*}	
	The constant $c$ is defined in \eqref{ec}.
\end{proof}

\begin{lem}\label{approximate K and K lemma}
	We denote
	\begin{align*}
	 H_\eps:=\frac{1-\b}{V}\int_M \log (|s|_h^2+\eps)(\om_\vphi^n-\om_0^n), \text{ and }H:=\frac{1-\b}{V}\int_M \log |s|_h^2(\om_\vphi^n-\om_0^n).
	\end{align*} Then we have
\begin{align}
\nu^\eps_{\beta,\chi}(\vphi)=\nu_{\beta,\chi}(\vphi)
+H_\eps-H\nonumber.
\end{align}	
\end{lem}
\begin{proof}
Substituting \eqref{approximate entropy and entropy} into the formula of the approximate log $\chi$-twisted $K$-energy
\eqref{Approximation of the log twisted K energy}, we have
\begin{align*}
\nu^\eps_{\beta,\chi}(\vphi)
&=E_\beta(\vphi)+\frac{1-\beta}{V}\int_M[\log (|s|^2_h+\eps)-\log|s|_h^2 ]\om_{\vphi}^{n}\\
&+J_{-\theta}(\vphi)+J_{\chi}(\vphi)
+\frac{1}{V}\int_M [-(1-\b)\log(|s|_h^2+\eps)+h_0]\om_0^n.
\end{align*}
From the definition of the log $\chi$-twisted $K$-energy (Definition~\ref{log twisted K energy}) and \lemref{log K energy}, we have
\begin{align*}
\nu^\eps_{\beta,\chi}(\vphi)
&=\nu_{\beta,\chi}(\vphi)
-\frac{1}{V}\int_M (-(1-\b)\log |s|_h^2+h_0)\om_0^n\\
&+\frac{1-\beta}{V}\int_M[\log (|s|^2_h+\eps)-\log|s|_h^2 ]\om_{\vphi}^{n}\\
&
+\frac{1}{V}\int_M [-(1-\b)\log(|s|_h^2+\eps)+h_0]\om_0^n.
\end{align*}Therefore, the lemma is proved.

\end{proof}

\begin{lem}\label{approximation proper}
For any $\eps\in(0,1]$ and for any $\vphi\in \mathcal H$, it holds
\begin{align*}
E_\beta^\eps(\vphi)\geq E_\beta-c\text{ and }\nu^\eps_{\beta,\chi}(\vphi)
\geq \nu_{\beta,\chi}(\vphi)-C.
\end{align*}
The constant $C$ depends on $X,D,\om_0,h$.
\end{lem}
\begin{proof}
The first conclusion follows from the relation between the entropy and its approximation in \lemref{approximate entropy and entropy lemma} and using
\begin{align*}
\frac{1-\beta}{V}\int_M\log(|s|_h^2+\eps)\om_\vphi^{n}\geq \frac{1-\beta}{V}\int_M\log|s|_h^2\om_\vphi^{n}.
\end{align*}
Furthermore, we have $H_\eps\geq H-C$ (c.f. Lemma 3.5 in \cite{MR3264766}), which follows from
\begin{align*}
\frac{1-\b}{V}\int_M  (\log |s|_h^2-\log(|s|_h^2+\eps))\om_0^n=-\frac{1-\b}{V}\int_M  \log\frac{|s|_h^2+\eps}{|s|_h^2}\om_0^n\geq -C.
\end{align*}
Therefore, the second conclusion is obtained from \lemref{approximate K and K lemma} and the inequality above.
\end{proof}
\begin{lem}\label{eps twisted approximation lemma}
	The critical point of the approximate $\chi$-twisted $K$-energy satisfies the following equation
	\begin{align*}
	S(\om_{\varphi_\eps})=\tr_{\vphi_\eps}(Ric(\om_{\theta_\eps})-\theta+\chi)+\underline S_\b-\underline{\chi},
	\end{align*}
	which is equivalent to
\begin{equation}\label{eps twisted approximation}
\left\{
\begin{aligned}
F_\eps&=\log\frac{\om_{\vphi_\eps}^n}{\om_{\theta_\eps}^n},\\
\tri_{\vphi_\eps} F_\eps&=\tr_{\vphi_\eps}(\theta-\chi)-(\underline S_\b-\underline{\chi}).
\end{aligned}
\right.
\end{equation}
\end{lem}
\begin{proof}
	It is a direct computation of the first derivative of the approximate log $\chi$-twisted $K$-energy \eqref{Approximation of the log twisted K energy}. We use the derivative
 \begin{align*}
	\delta J_{\chi}(\vphi)=\frac{1}{V}\int_M \delta\vphi(\tr_\vphi\chi-\underline{\chi})\om^n_\vphi.
	\end{align*}
	So we have
	\begin{align*}
\delta	\nu^\eps_{\beta,\chi}(\vphi)
	&:=\frac{1}{V}\int_M\log\frac{\om^n_\vphi}{\om^n_{\theta_\eps}}\cdot \tri_\vphi\delta\vphi\cdot \om_\vphi^{n}\\
	&+\frac{1}{V}\int_M \delta\vphi(-\tr_\vphi\theta+\underline S_\b)\om^n_\vphi+\frac{1}{V}\int_M \delta\vphi(\tr_\vphi\chi-\underline{\chi})\om^n_\vphi.
	\end{align*} The lemma is proved by integration by parts.
\end{proof}

We check the topological condition, so that the equation in the lemma above is well-defined.
\begin{lem}
	\begin{align*}
\frac{1}{V}\int_MS(\om_{\varphi})\om_\vphi^n-\frac{1}{V}\int_M(Ric(\om_{\theta_\eps})-\theta+\chi)\wedge\om_\vphi^{n-1}=\underline S_\b-\underline{\chi}.
	\end{align*}
\end{lem}
\begin{proof}
	Since $\om_\vphi$ is a smooth K\"ahler metric in $[\om_0]$, we have $$\frac{1}{V}\int_MS(\om_{\varphi})\om_\vphi^n=\frac{C_1(X)[\om_0]^{n-1}}{[\om_0]^n}.$$ Recall that $\theta\in  C_1(X,D)$. Then we have
	\begin{align*}
	\frac{1}{V}\int_M(Ric(\om_{\theta_\eps})-\theta+\chi)\wedge\om_\vphi^{n-1}=\frac{(1-\b)C_1(L_D))[\om_0]^{n-1}}{[\om_0]^n}+\underline{\chi}.
	\end{align*}
	As a result, the lemma follows from $\underline S_\b=\frac{C_1(X,D)[\om_0]^{n-1}}{[\om_0]^n}$.
\end{proof}
\subsubsection{Solving approximate equations}\label{Solving approximate equations}
\begin{prop}[Existence of approximate solutions]\label{twisted approximation thm}
Assume that $\chi$ is a smooth nonnegative closed $(1,1)$-form and $C_1(L_D)\geq 0$. Assume the log $\chi$-twisted $K$-energy $\nu_{\beta,\chi}$ is proper. Then the $\chi$-twisted cscK cone metric has a smooth approximation $\{\vphi_\eps,\eps\in(0,1]\}$ satisfying \eqref{eps twisted approximation}.
\end{prop}
\begin{proof}
From \lemref{approximation proper}, the approximation twisted $K$-energy $\nu^\eps_{\beta,\chi}$ is also proper. From $C_1(L_D)\geq 0$, \lemref{thetaepstheta} implies that $Ric(\om_{\theta_\eps})\geq\theta$. By the assumption of  $\chi$, we have that the twisted term $Ric(\om_{\theta_\eps})-\theta+\chi$ is a smooth non-negative closed $(1,1)$-form. According to Theorem 4.1 in \cite{arXiv:1801.00656}, there exists a unique smooth twisted cscK metric satisfying equation \eqref{eps twisted approximation}.
Thus we obtain the existence of approximate solutions $\vphi_\eps$.
\end{proof}
Then we derive properties of the approximate solutions $\vphi_\eps$.
\begin{lem}[Properness implies approximate entropy bound]\label{uniform entropy}
The approximate entropy satisfies
\begin{align}
\sup_{\eps\in(0,1]} E_\beta^\eps(\vphi_{\eps}) \leq C.
\end{align}
\end{lem}
\begin{proof}
Since $\vphi_\eps\in \mathcal H$ is the global minimiser of the log twisted energy $\nu^\eps_{\beta,\chi}$ in $\mathcal H$ (c.f. \lemref{globalminimiser} but using $C^{1,1}$ geodesic \cite{MR1863016} and convexity in \cite{MR3671939,MR3582114}), we have 
\begin{align}\label{approximate K energy upper bound}
\nu^\eps_{\beta,\chi}(\vphi_\eps)\leq \nu^\eps_{\beta,\chi}(0).
\end{align}
We compute
\begin{align*}
\nu^\eps_{\beta,\chi}(0)
&=\frac{1}{V}\int_M\log\frac{\om^n_0}{\om^n_{\theta_\eps}}\om_0^{n}
+\frac{1}{V}\int_M [-(1-\b)\log(|s|_h^2+\eps)+h_0]\om_0^n+c
=0.
\end{align*} 
From \lemref{approximation proper}, we have
\begin{align*}
\nu_{\beta,\chi}(\vphi)\leq \nu^\eps_{\beta,\chi}(\vphi)+C.
\end{align*}
Then the properness of $\nu_{\beta,\chi}$ implies that
 \begin{align}
d_1(\vphi_{\eps},0)\leq C \text{ independent of }\eps.
\end{align} Thanks to \lemref{d_1andJ}, we have $J_{-\theta}(\vphi_\eps)$ and $J_\chi(\vphi_\eps)$ are uniformly bounded below. Thus, from the formula of $\nu^\eps_{\beta,\chi}$ in \eqref{Approximation of the log twisted K energy}, we obtain the uniform bound of the approximate entropy in the lemma.

\end{proof}

\begin{lem}[A priori estimates of the approximate solutions]\label{eps twisted approximation estimates}For any $\eps\in(0,1]$, there is a constant $C$ such that 
	\begin{align}\label{a prioriestimates approximation uniform}
	\|\vphi_\eps\|_\infty, \quad \|F_\eps\|_\infty, \quad  \|\p F_\eps\|_{\om_{\theta_\eps}}\leq C\text{ and }C^{-1}\om_{\theta_\eps}\leq \om_{\vphi_\eps}\leq C\om_{\theta_\eps},
	\end{align}where $C$ depends on the following quantities
	\begin{align}
	 E^\eps_\beta(\vphi_\eps), \quad \|\theta-\chi\|_\infty,\quad\inf_X\theta,\quad \alpha_1,\quad \alpha_\beta,\quad \underline S_\b,\quad n.
	\end{align}
\end{lem}
\begin{proof}
It directly follows from the a priori estimates (\thmref{a prioriestimates approximation}) and the Ricci lower bound (\lemref{thetaepstheta}). Note that the uniform bound of the entropy $E^\eps_\beta(\vphi_\eps)$ follows from \lemref{uniform entropy}.
\end{proof}

\begin{lem}[Entropy approximation]\label{properclosedness entropy convergence}
After taking $\eps\rightarrow 0$, the entropy converges if $E_\beta(\vphi)<\infty$:
\begin{align}\label{log entropy convergence}
E_\beta^\eps(\vphi_\eps)\rightarrow E_\beta(\vphi),
\end{align}
and $\sup_{\eps\leq \eps_0} E^\eps_\beta(\vphi_\eps)\leq E_\beta(\vphi)+1$ for $\eps_0$ sufficiently small.
\end{lem}
\begin{proof}
	We first see that $\om_{\theta_\eps}$ has a uniform bound. Actually, 
	from the second order estimate for \eqref{thetaeps} (c.f. \cite{MR3368100}), we have $\om_{\theta_\eps}$ is uniformly equivalent to the approximate model metric $\om_D^\eps$, that is \begin{align*}
	C_1\om_D^\eps\leq \om_{\theta_\eps}\leq C_2\om_D^\eps
	\end{align*}
	for two uniform constants $C_1, C_2$. 	Here, we recall that the model metric is $\om_D=\om_0+\delta i\p\bar \p|s|_h^{2\beta}$ with small constant $\delta$ and its approximation is $\om_D^\eps=\om_0+\delta i\p\bar \p(|s|_h^2+\eps)^\beta$.

We then use the estimates in \lemref{eps twisted approximation estimates}, to get the estimates of the approximation solutions
\begin{align*}
\|\vphi_\eps\|_\infty+\|i\p\bar\p\vphi_\eps\|_{\om_{\theta_\eps}}\leq C.
\end{align*}
The constant $C$ is independent of $\eps$. 

So $\frac{\om^n_{\vphi_\eps}}{\om_{0}^n}=\frac{\om^n_{\vphi_\eps}}{\om_{\theta_\eps}^n}\frac{\om^n_{\theta_\eps}}{\om_{0}^n}=F_\eps\frac{\om^n_{\theta_\eps}}{\om_{0}^n}$ is in $L^p(\om_0)$ for some $p>1.$
Moreover, $\vphi_\eps\rightarrow \vphi$ in $C^{0,\alpha}(X)$ and $\om_{\vphi_\eps}\rightarrow \om_\vphi$ weakly in $L^p(\om_{\theta})$ for any $p\geq 1$.

Next, we obtain the higher order estimates outside the divisor by applying the Evan-Krylov and the Schauder estimate. Hence the approximate sequence $\vphi_{\eps}$ converges to $\vphi$ point-wise outside $D$, and  $F_\eps$ also converges to $F$ point-wise outside $D$.

At last, we conclude \eqref{log entropy convergence} by applying Lebesgue's dominated convergence theorem and \lemref{uniform entropy} to the integral $
	E_\beta^\eps(\vphi_\eps)=\frac{1}{V}\int_M F_\eps\log F_
\eps\om^n_{\theta_\eps}$. As a result, we have $E_\beta^\eps(\vphi_\eps)\leq E_\beta(\vphi)+1$, when $\eps$ is sufficiently close to $0$.

\end{proof}
\subsection{Existence and a priori estimates: approximation method}
We leave the proof of the uniform a priori estimates of the approximation equation \eqref{eps twisted approximation} in Section \ref{a priori estimates}.
In this section, we will show how to apply the uniform a priori estimates to the limit equation \eqref{twisted 2nd equ defn} as $\eps\rightarrow 0$.
The resulting estimates are obtained in the 4th H\"older space $D_{\bf w}^{4,\a,\b}(\om_\theta)$, which is defined in Definition \ref{4thspace}.

\begin{thm}[Properness theorem]\label{properclosedness}
	Assume that the twisted term $\chi$ is a smooth non-negative closed $(1,1)$-form and $C_1(L_D)\geq 0$. Suppose the log $\chi$-twisted $K$-energy $\nu_{\beta,\chi}$ is proper. 
	Then there exists a $\chi$-twisted cscK cone metric $\om_\vphi=\om+i\p\bar\p\vphi$ with $\vphi\in D_{\bf w}^{4,\a,\b}(\om_\theta)$.
	\end{thm}
\begin{proof}
Let $\vphi_\eps$ be the sequence of smooth approximate solutions obtained in Proposition \ref{twisted approximation thm}. Combining \lemref{uniform entropy}, \lemref{eps twisted approximation estimates} and \lemref{properclosedness entropy convergence}, we conclude that when $\eps\rightarrow 0$, the approximate solution $\vphi_\eps$ smoothly converges to $\vphi$ in the regular part $X\setminus D$. The limit $\vphi$ is a solution to \eqref{twisted 2nd equ defn}.
\lemref{properclosedness entropy convergence} also implies the estimates
		\begin{align*}
	\|\vphi\|_\infty,  \text{ and }C^{-1}\om_{\theta}\leq \om_{\vphi}\leq C\om_{\theta}.
	\end{align*}
Since $\tri_{\vphi} F=\tr_{\vphi}(\theta-\chi)-(\underline S_\b-\underline{\chi})$,  the $C^{0,\a,\b}$ estimate of $F$ is obtained by the Nash-Moser iteration. Then the estimate of $\|\vphi\|_{C^{2,\a,\b}}$ is obtained by Evans-Krylov estimate (c.f. \cite{MR3405866,MR3668765,MR3488129,MR3264767}). Thus we have the estimate of $\|F\|_{C^{2,\a,\b}}$ from the equation of $\tri_{\vphi} F$, by the Schauder estimate \cite{MR2975584}. In conclusion, we prove that $\vphi\in D_{\bf w}^{4,\a,\b}(\om_\theta)$.
\end{proof}

The following corollary is directly obtained from the proof above.
\begin{cor}\label{properclosedness estimates}Under the assumption in \thmref{properclosedness}, there is a constant $C$ such that  
\begin{align*}
||\vphi||_{D_{\bf w}^{4,\a,\b}(\om_\theta)}=||\vphi||_{C^{2,\a,\b}}+||F||_{C^{2,\a,\b}}\leq C.
\end{align*} The constant $C$ depends on the following quantities
\begin{align*}
E_\beta(\vphi), \quad \|\theta\|_{C^{0,\a,\b}},\quad\|\chi\|_{C^{0,\a,\b}},\quad \alpha_1,\quad \alpha_\beta,\quad\underline S_\b,\quad n.
\end{align*}
\end{cor}
\begin{proof}
The up to 2nd estimates follow from \lemref{eps twisted approximation estimates} and \lemref{properclosedness entropy convergence}. Indeed, the approximate entropy is bounded by $E_\beta(\vphi)$ by \lemref{properclosedness entropy convergence}. The higher order estimates are given in \thmref{properclosedness}. 
\end{proof}
We remark that the bound of $E_\beta(\vphi)$ is given by the properness of $\nu_{\beta,\chi}(\vphi)$ and \lemref{globalminimiser}.


\subsection{Properness implies existence: general automorphism group}\label{General case}
We let $$G:=Aut(X;D)$$ be the identity component of the group of holomorphic automorphisms of $X$ which fix the divisor $D$. 
Given a K\"ahler cone potential $\vphi\in\mathcal H_\b$, we denote the $G$-orbit to be
\begin{align*}
\mathcal O_\vphi=\{\tilde{\vphi}\vert\om_{\tilde{\vphi}}=\sigma^\ast\om_ \vphi, \forall \sigma\in G\}.
\end{align*}
The $d_{1,G}$-distance between $\vphi_1,\vphi_2\in \mathcal H_\b$ is defined to be the infimum of the distance between the two orbits $\mathcal O_{\vphi_1}$ and $\mathcal O_{\vphi_2}$.
There exists $\sigma_1,\sigma_2\in G$ such that
$$
	d_{1,G}(\vphi_1,\vphi_2)=d_1(\sigma_1^\ast\vphi_1,\sigma_2^\ast\vphi_2),
$$
	c.f. Proposition 7.1 in \cite{MR4020314}.
	
	Since $d_1$-distance is invariant under $G$-action,
	letting $\sigma=(\sigma_1)^{-1}\circ\sigma_2$ we have this infimum is achieved at $\om_{\tilde\vphi_2}=\sigma^\ast\om_ {\vphi_2}$ such that
	\begin{align*}
	d_{1,G}(\vphi_1,\vphi_2)=d_1(\vphi_1,\tilde\vphi_2).
	\end{align*}
\begin{defn}\label{d1Gproper}
The log $K$-energy $\nu_\beta$ is $d_{1,G}$-proper, if 
\begin{itemize}
	\item for any sequence $\vphi_{ i}\subset\mathcal H_\beta$, if $\nu_\beta\rightarrow \infty$, then $d_{1,G}\rightarrow \infty$;
	\item the log $K$-energy is bounded below.
\end{itemize}
\end{defn}
\begin{thm}[Properness implies existence]\label{Properness implies existence general}
	Assume that the log $K$-energy is $d_{1,G}$-proper. Then there exists a constant scalar curvature K\"ahler cone metric in $\Om$.
\end{thm}
The proof is divided into the following steps.
\subsubsection{Continuity path}
We apply the continuity path
\begin{align}\label{2Twisted cone path general}
t(S-\underline S_\b)=(1-t)(\tr_\vphi\om_{\vphi_j}-n).
\end{align}
As a result of \thmref{properclosedness}, we have
\begin{prop}\label{path except 1}
	Assume  $t\in(0,1)$ and the log $K$-energy $\nu_{\beta}$ is $d_{1,G}$-proper. 
	Then there exists a twisted cscK cone metric satisfying \eqref{2Twisted cone path} with $\vphi\in D^{4,\a,\b}$. 
\end{prop}
\begin{proof}
	Since the log $K$-energy is bounded below and the $J_{\om_0}$-functional is $d_1$-proper, we have the log twisted $K$-energy $K_t$ is also $d_1$-proper. Then the conclusion follows from \thmref{properclosedness}.
\end{proof}

\subsubsection{Entropy along cone path}

Given the $\{\vphi_t;0< t<1\}$ satisfying \eqref{2Twisted cone path general}. 
	Since $\nu_\beta$ is $d_{1,G}$-proper, we have $d_{1,G}(0,\vphi_t)$ is bounded,
	\begin{align}\label{d1Gvphit}
	\sup_{0.1< t<1} d_{1,G}(0, \vphi_t)\leq C.
	\end{align}

By definition, there exists a family of $G$-action $\sigma_t\in G$ such that 
\begin{align*}
\om_{\tilde\vphi_t}:=\sigma_t^\ast\om_{\vphi_t}=\om+i\p\bar\p \tilde\vphi_t,\quad d_{1,G}(0, \vphi_t)=d_{1}(0, \tilde\vphi_t).
\end{align*}
In order to obtain a uniform estimate of $\vphi_t$, we need to control the entropy $\sup_{0.1< t<1} E_\beta(\tilde\vphi_t)$.
\begin{lem}\label{E beta tilde vphi t}Assume the log $K$-energy is $d_{1,G}$-proper. Then
	\begin{align}
	\sup_{0.1< t<1} E_\beta(\tilde \vphi_t)\leq C.
	\end{align}
\end{lem}
\begin{proof}
	From Lemma 3.7 in \cite{arXiv:1801.05907}, the log $K$-energy is controlled along the twisted cone path as following,  $\nu_\beta(\vphi_t)\rightarrow \inf_{\mathcal H_\beta} \nu_\b$, as $t\rightarrow 1$ and
	\begin{align}\label{proper existence K upper bound}
	\sup_{0.1< t<1}\nu_\beta(\vphi_t)\leq \inf_{\mathcal H_\beta} \nu_\b+1.
	\end{align}
	Since $d_{1,G}(0,\vphi_t)$ is bounded,
$
	\sup_{0.1< t<1} d_{1,G}(0, \vphi_t)\leq C, 
$
	we have equivalently, 
	\begin{align*}
	\sup_{0.1< t<1} d_{1}(0, \tilde\vphi_t)\leq C.
	\end{align*}
	
	Note that the derivative of the log $K$-energy $\nu_\beta$ is log Futaki invariant $Fut_\beta$. Since the log $K$-energy $\nu_\beta$ is bounded below, so $Fut_\beta=0$ and $\nu_\beta$ is $G$-invariant, that is $\nu_\beta(\tilde\vphi_t)=\nu_\beta(\vphi_t)$.
	By Lemma 2.12, we bound $|J_\om( \tilde\vphi_t)|\leq C(n)\|\om\|_{\infty}d_1(0, \tilde\vphi_t)$ and similarly $|J_{-\theta}( \tilde\vphi_t)|$. Using the formula of the log $K$-energy (Definition 2.9)
	\begin{align*}
	\nu_{\beta}
	=E_\beta+J_{-\theta}+\frac{1}{V}\int_M( \mathfrak h+h_0)\om_0^n,
	\end{align*} we bound the entropy $E_\beta(\tilde\vphi_t), 0.1< t<1$.
	
\end{proof}

\subsubsection{Modified cone path by automorphisms}

We define 
\begin{align*}
\tilde\om_0:=\sigma_t^\ast\om_0=\om_0+i\p\bar\p f_t \text{ and }\sup_X f_t=0.
\end{align*}

\begin{lem}\label{e-ftintegral}It holds $e^{-f_t}\rightarrow 1$ in $L^p(\om^n_0)$ for any $p>1$.
	Also,
	$e^{-f_t}$ in $L^p(\om^n_{0}) \cap L^p(\om^n_{\theta})$ for any $p>1$.
\end{lem}
\begin{proof}
	The first conclusion follows from Lemma 3.15 in \cite{arXiv:1801.05907}.
	The second conclusion follows from  $\om^n_\theta\in L^{p_0}(\om^n_0)$ for some $p_0>1$.
\end{proof}

\begin{lem}
	The modified potential $\tilde\vphi_t$ satisfies the modified cone path equation
	\begin{equation}\label{tilde G twisted 2nd equ approximation}
	\left\{
	\begin{aligned}
	&\tilde F_t=\log\frac{\om_{\tilde\vphi_t}^n}{\om_{\theta}^n},\\
	&\tri_{\tilde\vphi_t} \tilde F_t=\tr_{\tilde\vphi_t}[\theta-\frac{1-t}{t}(\om_0+i\p\bar\p f_t)]-(\underline S_\b-\frac{1-t}{t}n).
	\end{aligned}
	\right.
	\end{equation} 
\end{lem}
\begin{proof}
	Taking $\sigma$-action on \eqref{2Twisted cone path}, we have
	\begin{align*}
	\tri_{\tilde \vphi_t} \log\frac{\om_{\tilde\vphi_t}^n}{\om_{\tilde\theta}^n}=\tr_{\tilde\vphi_t}[\tilde\theta-\frac{1-t}{t}\tilde\om_0]-(\underline S_\b-\frac{1-t}{t}n).
	\end{align*}
	Recall that, 
	$
	Ric(\om_{\theta})=\theta+2\pi(1-\beta)[D].
	$
	Since $\sigma$ fixes the divisor, we have
	\begin{align*}
	Ric(\om_{\tilde\theta})=\tilde\theta+2\pi(1-\beta)[D].
	\end{align*}
	Thus we have
	\begin{align*}
	\tri_{\tilde \vphi_t} \log\frac{\om_{\tilde\vphi_t}^n}{\om_{\theta}^n}=\tr_{\tilde\vphi_t}[\tilde\theta-Ric(\om_{\tilde\theta})+Ric(\om_{\theta})-\frac{1-t}{t}\tilde\om_0]-(\underline S_\b-\frac{1-t}{t}n),
	\end{align*}
	and we have proved the lemma.
\end{proof}

\begin{lem}\label{tilde G twisted 2nd equ approximation estimate}
	Suppose that $\tilde\vphi_t$ satisfies \eqref{tilde G twisted 2nd equ approximation}. Then there is a constant $C$ such that 
	\begin{align}\label{tilde G twisted 2nd equ approximation estimate quantities}
	\|\tilde\vphi_t\|_{\infty}, \quad \|\tilde F_t+f\|_{\infty}, \quad \sup_X |\p (\tilde F_t+f)|^2_\vphi, \quad \|\tr_{\om_{\theta}}\om_{\tilde\vphi_t}\|_{p;\om_{\theta}}\leq C,
	\end{align} 
	The constant $C$ depends on $\|\theta\|_\infty$, $\|\om_0\|_\infty$, $\alpha_1$, $\alpha_\beta$, $\underline S_\b$, $n$ and
	\begin{align}\label{tilde G twisted 2nd equ approximation dependence}
	\sup_{0.1< t<1} E_\beta(\tilde\vphi_t)=\frac{1}{V}\int_M\log\frac{\om^n_{\tilde\vphi_t}}{\om^n_{\theta}}\om_{\tilde\vphi_t}^{n}, \quad \|e^{-f}\|_{p_0;\om_{\theta}}.
	\end{align}
	with sufficiently large $p_0$.
\end{lem}
\begin{proof}
	We apply \thmref{a prioriestimates approximation} to \eqref{tilde G twisted 2nd equ approximation} to obtain \eqref{tilde G twisted 2nd equ approximation estimate quantities}.
	The $L^{p_0}$ integral of $e^{-f}$ is bounded, thanks to \lemref{e-ftintegral} and the entropy $E_\beta(\tilde\vphi_t)$ is bounded, due to \lemref{E beta tilde vphi t}. 
\end{proof}


\subsubsection{Closedness: $t\rightarrow 1$}

After taking $t\rightarrow 1$, $\tilde\vphi_t$ converges to $\tilde{\vphi}$ in $C^{1,\alpha,\beta}$ and $\tilde F_t$ converge to $\tilde{F}$ in $C^{0,\alpha,\beta}$. Furthermore, $i\p\bar\p\tilde \vphi_t\rightarrow i\p\bar\p\tilde\vphi$ and $\p (F_t+f_t)\rightarrow \p\tilde F$ weakly in $L^p$.
The associate volume ratio $\tilde F=\frac{\om_{\tilde\vphi}^n}{\om^n_\theta}$ is a weak solution of the log $K$-energy (Proposition 3.14 in \cite{arXiv:1801.05907}),
\begin{align}\label{tilde G twisted 2nd equ}
-\int_M \p \tilde F\wedge \p \eta\wedge \om_{\tilde\vphi}^{n-1}=\int_M -\eta \underline S_\b \om_{\tilde\vphi}^n +n\eta \theta\wedge \om_{\tilde\vphi}^{n-1}.
\end{align} 
We could adapt Theorem \ref{1alpha estimates} as below to conclude that the limit $\tilde\vphi$ has 2nd order estimate. Moreover, by Schauder estimate, $\tilde\vphi$ and $\tilde F$ are in $C^{2,\alpha,\beta}$.


\subsubsection{A second order estimate}\label{second order estimate w1p}
In this section, we consider the complex Monge-Amp\`ere equation,
\begin{align}\label{equation second order estimate w1p}
\frac{\om_{\vphi_\eps}^n}{\om^n_{\theta_{\eps}}}=e^{F_\eps}.
\end{align}
with $F_\eps\in W^{1,p}(\om^n_{\theta_{\eps}})$. The estimate extends the work of Chen-He \cite{MR2993005} via the integral method. The the proof requires the similar strategies in Section \ref{a priori estimates}, i.e. the weighted estimates. 


\begin{thm}\label{1alpha estimates}Suppose that $\vphi$ is a classical solution to \eqref{equation second order estimate w1p}.
	There exists a constant $C$ such that
	\begin{align*}
	\sup_X \tr_{\om_{\theta_{\eps}}}\om_{\vphi_\eps}\leq C.
	\end{align*}
	The constant $C$ depends on the quantities in \eqref{backgroundmetriccurvature}, and
	\begin{align*}
	\|\vphi_\eps\|_{C^{1,\a,\b}},\quad \|F_\eps\|_{W^{1,p}(\om^n_{\theta_{\eps}})}\text{ for }p>2n.
	\end{align*}
\end{thm}
\begin{proof}
	We omit the index $\eps$ in the proof. Recall \eqref{tritrvphiom}, there is a positive constant $C_0$ such that
	\begin{align*}
	\tri_\vphi [\log (\tr_{\om}\om_\vphi)+\phi] \geq - C_0\tr_\vphi\om+\frac{\tri F}{\tr_{\om}\om_\vphi}.
	\end{align*}
	We set $u=e^{-C\vphi+\phi}\tr_{\om}\om_\vphi$ with $C=C_0+1$, and we compute that
	\begin{align*}
	\tri_\vphi u &\geq u \tri_\vphi\log u
	=u[-Cn+C\tr_{\vphi}\om-C_0\tr_{\vphi}\om+\frac{\tri F}{\tr_{\om}\om_\vphi}]\\
	&=e^{-C\vphi+\phi}\tr_{\om}\om_\vphi[-Cn+\tr_{\vphi}\om+\frac{\tri F}{\tr_{\om}\om_\vphi}].
	\end{align*}
	Since $\tr_{\vphi}\om\geq e^{\frac{-F}{n-1}}(\tr_{\om}\om_\vphi)^\frac{1}{n-1}$, we have
	\begin{align*}
	\tri_\vphi u &\geq e^{-C\vphi+\phi}[-Cn\tr_{\om}\om_\vphi+e^{\frac{-F}{n-1}}(\tr_{\om}\om_\vphi)^\frac{n}{n-1}+\tri F].
	\end{align*}
	We apply the Cauchy inequality with epsilon to see that
	\begin{align*}
	\tr_{\om}\om_\vphi\leq \eps (\tr_{\om}\om_\vphi)^\frac{n}{n-1}+C(n,\eps).
	\end{align*}We choose $\eps$ such that $-Cn\eps+e^{\frac{-\sup_X F}{n-1}}<0$.
	As a consequence, letting $A=e^{-C\inf_X\vphi+\sup_X \phi}CnC(n,\eps)$, we have
	\begin{align*}
	\tri_\vphi u 
	&\geq e^{-C\vphi+\phi}[-CnC(n,\eps)+\tri F]\geq-A+e^{-C\vphi+\phi}\tri F.
	\end{align*}
	We write it in the integral form,
	\begin{align*}
	&\int_X (p-1)u^{p-2}|\p u|^2_\vphi \om_\vphi^n=\int_X u^{p-1}(-\tri_\vphi u)\om_\vphi^n\\
	&\leq \int_Xu^{p-1}  [ A-e^{-C\vphi+\phi}\tri F]\om_\vphi^n\\
	&=A\int_Xu^{p-1}   \om_\vphi^n-\int_Xu^{p-1} e^{-C\vphi+\phi+F}\tri F\om^n.
	\end{align*}
	
	Since $u|\p u|^2_\vphi\geq e^{-C\vphi+\phi}|\p u|_\om^2$, we have the left hand side is bounded below by
	\begin{align*}
	LHS&=\int_X (p-1)u^{p-2}|\p u|^2_\vphi \om_\vphi^n\geq\int_X (p-1)u^{p-3}|\p u|_\om^2e^{-C\vphi+\phi}\om_\vphi^n\\
	&\geq\int_X \frac{4}{p-1}|\p u^{\frac{p-1}{2}}|_\om^2e^{-C\vphi+\phi+F}\om^n.
	\end{align*}
	Then there is a constant depending on $C,\|\vphi\|_\infty,\|\phi\|_\infty,\|F\|_\infty$ such that
	\begin{align*}
	LHS\geq \frac{C_1}{p-1}\int_X|\p u^{\frac{p-1}{2}}|_\om^2\om^n.
	\end{align*}
	
	We apply the integration by parts formula to the second integral on the right hand side,
	\begin{align*}
	I&=-\int_Xu^{p-1} e^{-C\vphi+\phi+F}\tri F\om^n\\
	&=\int_X(\p(u^{p-1}),\p F)_\om e^{-C\vphi+\phi+F}\om^n+\int_Xu^{p-1} e^{-C\vphi+\phi+F}(\p(-C\vphi+\phi+F),\p F)_\om\om^n\\
	&:=I_1+I_2.
	\end{align*}
	Then we use the assumption that $\vphi,\phi,F$ are bounded and $C_2=\sup_X (e^{-C\vphi+\phi+F})$ to see
	\begin{align*}
	I_1\leq 2C_2\int_X u^{\frac{p-1}{2}} |\p u^{\frac{p-1}{2}}|_\om|\p F|_\om\om^n.
	\end{align*}
	Then by the H\"older inequality, we have
	\begin{align*}
	I_1\leq \frac{2C^2_2(p-1)}{C_1}\int_X u^{p-1}|\p F|^2_\om\om^n+ \frac{C_1}{2(p-1)}|\p u^{\frac{p-1}{2}}|^2_\om\om^n.
	\end{align*}
	Thus after combining both the inequalities from the LHS and the RHS, we obtain that
	\begin{align*}
	\frac{C_1}{2(p-1)}\int_X|\p u^{\frac{p-1}{2}}|_\om^2\om^n&\leq A\int_Xu^{p-1}   \om_\vphi^n+\frac{2C^2_2(p-1)}{C_1}\int_X u^{p-1}|\p F|^2_\om\om^n+I_2.
	\end{align*}
	We then estimate the second term 
	\begin{align*}
	I_2&=\int_Xu^{p-1} e^{-C\vphi+\phi+F}(\p(-C\vphi+\phi+F),\p F)_\om\om^n\\
	&\leq C_2 \int_Xu^{p-1} |\p(-C\vphi+\phi+F)|_\om|\p F|_\om\om^n\\
	&\leq C_2 [\int_Xu^{p-1} (C|\p \vphi|_\om+|\p \phi|_\om)|\p F|_\om\om^n+\int_Xu^{p-1} |\p F|^2_\om\om^n ] .
	\end{align*}
	Since $|\p \vphi|_\om$ is bounded, and so is (see Section 2.0.1 in \cite{MR4020314}, \cite{MR3911741}), we get
	\begin{align*}
	I_2\leq C_3 \int_Xu^{p-1} (|\p F|_\om+|\p F|^2_\om)\om^n .
	\end{align*}
	The constant $C_3$ depends on $C$, $\|\vphi\|_{W^{1,\infty}(\om)}$, $\|\phi\|_{W^{1,\infty}(\om)},\|F\|_{L^{\infty}(\om)}$.

	Letting $p>2$ and $v=u^{\frac{p-1}{2}}$ and using the bounded of $F$, we have that 
	\begin{align*}
	\int_X|\p v|_\om^2\om^n&\leq C_4(p-1)^2[
	\int_X v^2(1+|\p F|_\om+|\p F|^2_\om)\om^n ].
	\end{align*}
	Substituting into the Sobolev inequality \eqref{Sobolev inequality}, we obtain
	\begin{align*}
	\|v\|_{\frac{2n}{n-1}}^2&\leq C(\|\p v\|_{2}^2+\|v\|_{2}^2)\\
	&\leq C_5(p-1)^2[
	\int_X v^2(1+|\p F|_\om+|\p F|^2_\om)\om^n ].
	\end{align*}
	Since $F\in W^{1,p_0}$ for any $p_0\geq 1$, we choose $p_0>n$ such that
	\begin{align*}
	&\int_X v^2(1+|\p F|_\om+|\p F|^2_\om)\om^n \leq \|v^2\|_{\frac{p_0}{p_0-1}}\|(1+|\p F|_\om+|\p F|^2_\om)\|_{p_0}.
	\end{align*}
	As as result, we obtain the inequality
	\begin{align*}
	\|v\|_{\frac{2n}{n-1}}\leq C_6(p-1)\|v\|_{\frac{2p_0}{p_0-1}}.
	\end{align*}
	Return back to $u$ and denote $q=\frac{2p_0}{p_0-1}$ and $b=\frac{2n(p_0-1)}{(n-1)2p_0}>1$, we have
	\begin{align*}
	\|u\|_{bq\frac{p-1}{2}}\leq C_7(p-1)^{\frac{2}{p-1}}\|u\|_{q\frac{p-1}{2}}.
	\end{align*}
	The iteration procedure as in Section \ref{Gradient estimates} shows that
	\begin{align*}
	\|u\|_{\infty}\leq C_8\|u\|_{1}.
	\end{align*}
	Recall that $u=e^{-C\vphi+\phi}\tr_{\om}\om_\vphi$. We have
	\begin{align*}
	\|\tr_{\om}\om_\vphi\|_{\infty}\leq C_9\|\tr_{\om}\om_\vphi\|_{1}.
	\end{align*}
	The $L^1$-norm is estimated as following
	\begin{align*}
	\|\tr_{\om}\om_\vphi\|_{1}&=\int_X (n+\tri_{\om}\vphi) \om^n\leq \int_X e^{\vphi-\inf_X\vphi}(n+\tri_{\om}\vphi)\om^n\\
	&=n\int_X e^{\vphi-\inf_X\vphi}\om^n
	\leq nVe^{\osc_X\vphi},
	\end{align*}
	since $\vphi\in L^\infty$. Therefore, the second order estimate is obtained.
\end{proof}


\section{Singular constant scalar curvature K\"ahler metrics}\label{Singular constant scalar curvature Kahler metrics}
In this section, we aim to study singular cscK metrics, which provides a possible canonical representative in arbitrary cohomology class. It would be an extension of Calabi's programme on cscK metrics in K\"ahler classes to big cohomology classes, see definitions below.

The problem of finding singular K\"ahler-Einstein metrics is related to the minimal model program in birational geometry.
The definition of singular cscK metric would generalise these well-studied metrics. There are many literatures studying existence of singular K\"ahler-Einstein metrics in the canonical class on a minimal projective manifold of general types, c.f. \cite{MR3956691,MR3090260,MR2746347,MR2505296,MR2869020} and references therein. In this situation, the manifolds do not have a definite first Chern class and these results generalise Yau's resolution \cite{MR480350} of Calabi conjecture in K\"ahler setting with zero first Chern classes, and also Aubin \cite{MR681859} and Yau's work \cite{MR480350} for negative first Chern classes.

At the beginning, we will include a short exposition on pluripotential theory.
\subsection{K\"ahler and Nef classes}
Recall that $X$ is a smooth K\"ahler manifold. In this Section, we use $\om_K$ to denote the smooth K\"ahler metric on $X$. We denote by $[\om_K]$ the K\"ahler class containing $\om_K$. By definition, the K\"ahler cone is the set of all K\"ahler classes.

Let $\Om$ be a real $(1,1)$-cohomology class. 
\begin{defn}A class $\Om$ is \textit{nef}, if it lies in the closure of the K\"ahler cone. That is the cohomology class $\Om+t[\om_K]$ is K\"ahler for all $t>0$.
\end{defn}

\subsection{Pseufo-effective classes}
\begin{defn}A current is defined to be a differential form with distribution coefficients. \end{defn} The space of $(p,q)$-currents is the dual of smooth $(n-p,n-q)$-forms. The positivity of a $(p,p)$-current $T$ is understood in the distribution sense, that is $T\wedge \sum_{j=1}^{n-p}iT_j\wedge\bar T_j$ is positive for all $(1,0)$-forms $T_j$, $1\leq j\leq n-p$. The closedness notion of a current are also understood in the same way.
\begin{defn}
A class $\Om$ is called \textit{pseudo-effective (psef)}, if it contains a closed positive current. 
\end{defn}Clearly, a nef class is psef. The psef cone is the space of psef classes. Note that the K\"ahler cone could be strictly smaller included in the interior of the psef cone.

The wedge product for closed positive currents is defined to be the non-pluripolar part in \cite{MR886812}. It is well-defined and closed on K\"ahler manifolds.

\begin{defn}
A positive current with \textit{minimal singularities} is defined to be less singular than any other positive currents in $\Om$, that is the upper envelope
\begin{align*}
V_{\theta}:=\sup\{\vphi\vert\vphi\in PSH(X,\theta),\quad\sup_X\vphi\leq 0\}.
\end{align*} 
\end{defn} A psef class always has a positive current with minimal singularities. We denote by $\Om_{min}$ the set of positive currents with minimal singularities.

We choose $\om_{sr}$ to be a smooth, closed $(1,1)$-form in $\Om$. Any closed positive current in $\Om$ could be written as $T=\om_{sr}+i\p\bar\p\vphi$ with a $\om_{sr}$-plurisubharmonic function $\vphi$. We denote by $PSH(X,\omega_{sr})$ the set of such $\om_{sr}$-plurisubharmonic functions $\vphi$.

\subsection{Big classes}
\begin{defn}
We say a class $\Om$ is \textit{big}, if it contains a \textit{K\"ahler current}, i.e. a closed positive $(1,1)$-current $T$ satisfying $T\geq t\om_K$ for some $t>0$.
\end{defn} Obviously, a big class is psef. In the other direction, a psef class $\Om$ is big if there exists a $T\in\Om$ such that the non-pluripolar product of $n$-copies of $T$ is not identically zero, see \cite{MR1945706}.

The volume $Vol(\Om)$ of a big class $\Om$ is defined to be  the $n$-times non-pluripolar product for currents in $\Om_{min}$. Since the integration by parts formula works for the currents in $\Om_{min}$, the volume $Vol(\Om)$ is independent of the choice of currents in $\Om_{min}$ and it is a non-negative number. 
\begin{defn}
A current $\om_{sr}+i\p\bar\p\vphi$ is said to have \textit{analytic singularities} if the $\om_{sr}$-plurisubharmonic functions $\vphi$ is locally written as the sum of a smooth function and $c\log\sum_{j=1}^N|f_j|^2$ with $c>0$ and $f_j$ being holomorphic functions. 
\end{defn} 
We denote by $\Om_{analytic}$ the set of strictly positive currents with analytic singularities. By Demailly's regularisation theorem \cite{MR1158622}, $\Om_{analytic}$ is not empty when $\Om$ is big.
\begin{defn}
The \textit{ample locus} of the big class $\Om$ is defined to be the set of points $p\in X$ such that there exists an $\om_\vphi:=\om_{sr}+i\p\bar\p\vphi\in \Om_{analytic}$ such that $\vphi$ is smooth around $p$. It is denoted by $Amp(\Om)$.
\end{defn} 
We fix an $\om_+=\om_{sr}+i\p\bar\p\vphi_+$ in $\Om_{analytic}$. Then we have
\begin{lem}
The following identity holds.
\begin{align}
PSH(X,\om_{sr})\simeq PSH(X,\om_{+})
\end{align}
\end{lem}
\begin{proof}
It is clear that any $\vphi\in PSH(X,\om_{+})$ maps to a element $\vphi+\vphi_{+}$ in $PSH(X,\om_{sr})$, and vice versa.
\end{proof}
\subsubsection{Big and nef}
On a K\"ahler manifold $(X,\om_K)$, a nef class $\Om$ is big if and only if $\int_X\Om^n>0$. By definiton, $\Om+t[\om_K]$ has a K\"ahler metric for all $t>0$, denoted by 
\begin{align*}
\om_{sr}+t\om_K+i\p\bar\p \phi_t.
\end{align*} Note that $\om_{sr}+t \om_K$ is in general not positive. It is also shown in \cite{MR2113021} that a big and nef class contains a K\"ahler current, which is smooth outside an analytic set $Z$ and with logarithmic poles along $Z$.

\subsubsection{Big and Semi-positive}
We say that the big class $\Om$ is semi-positive, if it admits a smooth closed $(1,1)$-form representative.

If we have a K\"ahler metric $\om_K$ on $X$ and let $\om_{sr}$ be a smooth represntative in $\Om$, we define 
\begin{align}\label{bigsemipositiveapp}
\om_t :=\om_{sr}+t \om_K,
\end{align}
which is K\"ahler for all $t>0$.


\subsection{Alpha invariants}\label{Alpha invariants}

When $\Om:=[\om_K]$ is a K\"ahler class, Tian defined the alpha invariant \cite{MR894378} as follows
\begin{align}\label{alpha1}
\alpha_1:=\sup\{\alpha> 0\vert\exists C\text{ such that }\sup_{\vphi\in \mathcal H(\om_K)}\int_X e^{-\alpha (\vphi-\sup_X\vphi)}\om_K^n\leq C\}.
\end{align}
Here, $\mathcal H(\om_K)$ is the space of smooth K\"ahler potentials $\vphi$ such that $\om_K+i\p\bar\p\vphi$ is a K\"ahler metric in $\Om$.

In the conical setting, the definition of the alpha invariant is generalised by replacing $\om_K$ in \eqref{alpha1} with the  K\"ahler cone metric $\om_\theta$. 
\begin{defn}\label{log alpha invariant}
	The log alpha invariant in $\mathcal H_\beta$ is defined as
	\begin{align*}
	\alpha_\beta:=\sup\{\alpha>0\vert\exists C\text{ such that}\sup_{\vphi\in \mathcal H_\beta(\om_\theta)}\int_X e^{-\alpha (\vphi-\sup_X\vphi)}\om_\theta^n\leq C\}.
	\end{align*}
\end{defn}
Recall that $\om_\theta^n=e^{h_K}|s|_h^{2\b-2}\om^n_K$, c.f. \eqref{theta}. Here $h_K$ is a smooth function, $s$ is a defining section of $D$ and $h$ is a Hermitian metric on $L_D$.
The definition of the invariant $\alpha_\beta$ does not depend on the choice of $\om_K, h_K, h$, and also $\om_\theta$. 

\begin{rem}
In the appendix in \cite{MR2484031}, Demailly showed that the alpha invariant coincides with the log canonical threshold, which was introduced in \cite{MR1852009}. This result was extended to the log case, see Appendix A in Berman \cite{MR3107540}. 
\end{rem}

\begin{rem}
The alpha invariant is extended to K\"ahler classes  on log fano varieties in \cite{MR3956691}.
\end{rem}



\subsubsection{Alpha invariants in big classes}

Motivated from Definition~\ref{log alpha invariant}, we extend the alpha invariant to a class $\Om$, which is merely big. Let $\om_{sr}$ be a smooth closed $(1,1)$-form in $\Om$. In general, $\om_{sr}$ can not be non-negative. 

We choose $\mu$ to be a positive measure on $X$, which puts no mass on pluripolar subsets. In addition, we require $\mu(X)$ equals to the volume of the big class $\Om$. According to \cite{MR2746347}, there exists a unique closed positive current $T$ in $\Om$ such that its non-pluripolar Monge-Amp\`ere measure is $\mu$. Furthermore, if $\mu$ is $L^p$ for some $p>1$ in term of a Lebesgue measure, then $T$ has minimal singularities, i.e. $T\in\Om_{min}$. The solution is then globally continuous, if $\Om$ is also semi positive, see \cite{MR2505296}. If $\mu$ is a smooth strictly positive volume form and $\Om$ is big and nef, then $T$ is smooth on the ample locus of $\Om$.

We define alpha invariant in the big classes.
\begin{defn}\label{big alpha invariant}
\begin{align*}
\alpha:=\sup\{\alpha> 0\vert\exists C\text{ such that }\sup_{\vphi\in PSH(X,\om_{sr})}\int_X e^{-\alpha (\vphi-\sup_X\vphi)}d\mu\leq C\}.
\end{align*}
\end{defn}

We remark that in the conical case, the measure $\mu$ is chosen to be 
\begin{align*}
\mu=\om_\theta^n=e^{h_K}|s|_h^{2\b-2}\om^n_K,
\end{align*}
which is $L^p$ for some $p>1$ in term of a Lebesgue measure. The log alpha invariant in big classes will also be denoted by $\alpha_\beta$.

We will show that Definition \ref{big alpha invariant} is well-defined, i.e. $\alpha>0$ is independent of the choice of $\om_{sr}$ and $C$ is uniform in the following proposition. It is a global version of uniform exponential ntegrability of the plurisubharmonic functions in big classes.
\begin{prop}\label{alpha1bigprop main}
Assume that $\mu$ is $L^p$ for some $p>1$ in term of a Lebesgue measure.
There exists two constants $\alpha_1$ and $C$ such that
\begin{align}\label{alpha1bigprop}
\int_X e^{-\alpha (\vphi-\sup_X\vphi)}d\mu\leq C,\quad \forall \vphi\in PSH(X,\om_{sr}).
\end{align}
\end{prop}
\begin{proof}
We denote $\psi=\vphi-\sup_X\vphi\leq 0$. First of all, from $\om_{sr}+i\p\bar\p\psi\geq0$, we have
\begin{align}
-\tri_{\om_K}\psi\leq\tr_{\om_K}\om_{sr}.
\end{align}
in the distribution sense. By using the elliptic estimate (see Proposition 4.8 and Corollary 4.9 in \cite{MR3761174}), we know there is a constant $A_1$ such that
\begin{align}
0=\sup_X\psi\leq \frac{1}{V_K}\int_X\psi\om_K^n+ A_1.
\end{align}
Here the volume $V_K$ is determined by the K\"ahler class $[\om_K]$.
Staying in a fixed ball $B_r(x)$ centred at $x\in X$ of radius $r\leq 1$, we have 
\begin{align}
\sup_{B_r(x)}\psi\geq \frac{-V_KA_1}{Vol_K(B_r(x))}:=A_2.
\end{align}
Then we could always find a point $y\in B_r(x)$ such that
\begin{align}\label{alpha1bigprop lower}
\psi(y)\geq \frac{A_2}{2}.
\end{align}
Since $\om_{sr}$ is closed and smooth, it has local smooth potential function $\phi$ such that $\om_{sr}=i\p\bar\p\phi$ in $B_r(x)$ and $\phi(x)=0$. Then we have $\sup_{B_r(x)}|\phi|\leq A_3$. 

We set $u=\phi+\psi-A_3$. In sum, we obtain that $u\leq 0$ on $B_r(x)$ and 
\begin{align}
u(y)\geq \frac{A_2}{2}-2A_3:=A_4.
\end{align} We notice that the constants $A_1,A_2,A_3$ only depend on $X,\om_K,\om_{sr}$. Then we make a use of H\"omander's Theorem 4.4.5 in \cite{MR1045639} to get a local estimate on a small ball $B_{\rho }(y)$ with $\rho \leq \frac{1}{2}dist_0(y,\p B_r(x))$,
\begin{align}\label{local integrability}
\int_{B_{\rho}(y)}e^{-\alpha u} d\mu \leq C. 
\end{align}
Then constant $C$ depends on $X,\om_K,\om_{sr},d\mu$, while, the radius $\rho$ depends on the exponential exponent $\alpha$ and $A_4$.

Actually, \eqref{local integrability} also follows from Skoda's integrability theorem, saying that if the Lelong number of the plurisubharmonic function $u$ at the point $y$ is less than $2$, then $e^{-u}$ is integrable with respect to the Lebesgue measure in a neighbourhood of $y$. Recall the Lelong number of $u$ at $y$ can be expressed as 
\begin{align*}
\lim_{\rho\rightarrow0^+}\frac{\max_{{|z-y|}=\rho} u(z)}{\log \rho}.
\end{align*}
By convexity of the plurisubharmonic function,
\begin{align*}
&\leq \frac{\max_{{|z-y|}=r} u(z)-\max_{{|z-y|}=\rho} u(z)}{\log r-\log \rho}\\
&\leq \frac{-A_4}{\log r-\log \rho}.
\end{align*}
Hence, \eqref{local integrability} is obtained from applying Skoda's theorem to $\alpha u$, as long as $\log\rho<\log r+\frac{1}{2}\alpha A_4$, and the H\"older inequality.

At last, we apply a standard argument of a partition of unity to glue the local estimates on balls with radius $r$ together to get a global estimate.
\end{proof}

\begin{rem}
Skoda's integrability theorem appeared in \cite{MR352517}. H\"ormander \cite{MR1045639} gave a uniform version of this estimate for plurisubharmonic functions in a unit ball in the complex $n$-space, which are bounded above and normalised at the centre of the ball. A uniform version of Skoda's theorem was proved in Zeriahi \cite{MR1857051}. Further developments of Skoda's estimates, we refer to \cite{MR2742678,MR2562773,MR3191972}.
\end{rem}

\begin{rem}
When $\Om$ is a K\"ahler class, the uniform inequality is proved in Proposition 2.1 in \cite{MR894378}. 
\end{rem}

We will show that the alpha invariant is continuous on the big cone.
\begin{lem}\label{alphacontinuity}
Let $\Om$ be a big class on a K\"ahler manifold $(X,\om_K)$.
Then for any sufficiently small $\eps_0>0$, there exists an $\tilde t>0$ such that for any $0\leq t\leq \tilde t$, we have
\begin{align}
|\alpha(\Om+t[\om_K])-\alpha(\Om)|<\eps_0.
\end{align}
\end{lem}
\begin{proof}
	By virtue of $PSH(\Om)\subset PSH(\Om+t[\om_0])$, we have
	\begin{align}\label{alphamonotonicity}
	\alpha(\Om+t[\om_K])\leq\alpha(\Om), \quad\forall t>0.
	\end{align}
	
	Let $b<1$ be a constant, which will be determined later. Since $\Om$ is big, we could choose $t$ small such that $b \Om-t [\om_K]>0$. Adding $\Om+t[\om_K]$ to $b \Om-t [\om_K]$, we have by \eqref{alphamonotonicity} that
		\begin{align*}
	\alpha((1+b)\Om)\leq\alpha(\Om+t[\om_K]).
	\end{align*}
	The scaling property of the alpha invariant implies that
			\begin{align}\label{alphaconti1}
\alpha(\Om)\leq(1+b)\alpha(\Om+t [\om_K]).
	\end{align}

On the other hand, due to psefness of a big class, we see that 
\begin{align*}
b\Om+t[\om_K]>0.
\end{align*}
Then if $\Om+t[\om_K]$ is deducted by $b \Om+t [\om_K]$, we still have a positive form $(1-b)\Om$. So $PSH((1-b)\Om)\subset PSH(\Om+t [\om_K])$. By \eqref{alphamonotonicity}, we have again
\begin{align}\label{alphaconti2}
\alpha(\Om+t [\om_K])\leq \alpha((1-b)\Om)=\frac{\alpha_1(\Om)}{1-b}.
\end{align}
Combining \eqref{alphaconti1} and \eqref{alphaconti2}, we obtain that
\begin{align*}
0\leq \alpha(\Om)-\alpha(\Om+t [\om_K])\leq\frac{b}{1-b}\alpha(\Om).
\end{align*}
The lemma is proved as long as we choose $b=(1+\frac{\alpha(\Om)}{\eps_0})^{-1}$.
\end{proof}

\begin{rem}
It is proved in Dervan \cite{MR3564131} that the alpha invariant is continuous on the K\"ahler cone. 
\end{rem}


\subsection{Twisted $J$-metrics and properness}\label{Twisted J-metrics and properness}
In \cite{MR3412393}, we introduced the twisted $J$-functional and obtained its critical points by solving the twisted $J$-flow. In this section, we extend these results to the conical setting. As a result, we will obtain criteria to check where the log $K$-energy is proper in a given cohomology class. The new feature of the result is that we will deal with cone singularities in the proof.

\begin{defn}Let $\gamma$ be a non-negative constant.
For any $\vphi\in\mathcal H_\beta(\om)$, we define
\begin{align*}
J^\gamma_{-\theta}(\vphi)=\underline S_\beta \cdot D(\vphi)+  j_{-\theta}(\vphi)+\gamma J_{\om_\theta}(\vphi).
\end{align*}
The notions above appear in \eqref{IJ functionals} and the definition of log $K$-energy, Definition \ref{logKenergy}. The references metric we choose in the definition is $\om_\theta$.
\end{defn}
Then its critical points are the twisted $J$-metrics.
\begin{defn}
We say $\om_\vphi$ is a twisted $J$-metric, if it 
satisfies the following equation
\begin{align}\label{lma}
n\cdot\theta\wedge\om^{n-1}_\vphi=c_\gamma \cdot \om^{n}_\vphi+\gamma\om_\theta^n.
\end{align}
The constant $c_\gamma$ is a topological constant determined by $$c_\gamma=n\cdot\underline S_\beta-\gamma,$$
and $\theta\in  C_1(X,D)$.
\end{defn}
Note that the smooth $(1,1)$-form $\theta$ and the background metric $\om_\theta$ are defined in Section \ref{Background metric}.


We make use of the following conical $J$-flow to find a solution to the critical equation \eqref{lma}. The proof of convergence of the conical $J$-flow is an adaption of Section 10 in \cite{MR3412393} to allow cone singularities.
\begin{theorem}\label{constgeodesicray big} We let $\om_{c}$ be a K\"ahler cone metric.
	Assume the following conditions hold
	\begin{enumerate}
		\item $-\theta>0$,
		\item $(-c_\gamma\cdot\om_c+(n-1)\theta)\wedge\om_c^{n-2}>0$.
	\end{enumerate} Then the conical $J$-flow for $\om_\vphi:=\om_c+i\p\bar\p\vphi$,
	 \begin{align}\label{jflow}
\frac{\p\vphi}{\p t}&=-c_\gamma+\tr_{\om_\vphi}\theta
 -\gamma\frac{\om_\theta^n}{\om^n_\vphi}.
 \end{align}
 converges and the functional $J^\gamma_{-\theta}(\vphi)$ has lower bound in $\mathcal H_\beta(\om_c)$.
	\end{theorem}
\begin{proof}
\textbf{Step 1: short time existence.}
We could choose $\vphi_0\in C^{2,\a,\beta}$ and linearise \eqref{jflow} at $\vphi_0$. The linearisation flow is
\begin{align}\label{Loperator}
\frac{\p^2\vphi}{\p t^2}=L_{\vphi}(\frac{\p\vphi}{\p t}):=[-g_{\vphi}^{k\bar j}g_{\vphi}^{i\bar l}\theta_{k\bar l}+\gamma\frac{\om_\theta^n}{\om^n_{\vphi}}g_{\vphi}^{ i\bar j}](\frac{\p\vphi}{\p t})_{i\bar j}.
 \end{align}
 It is a parabolic operator by condition $(1)$. The short time existence of the conical $J$-flow \eqref{jflow} follows from parabolic Schauder estimates for cone metrics \cite{MR3348827} and implicit function theorem \cite{MR3010550}. 
 
 \textbf{Step 2: a priori estimates.}
 By using the linearisation equation \eqref{Loperator}, we have from the cone maximum principle that the bound of $\frac{\p\vphi}{\p t}$. 
 
 The lower bound of $\om_\vphi$ is obtained from the following inequality, 
 \begin{align*}
 \min_M \dot\vphi(0)\leq \dot\vphi(t)&=-c_\gamma+\tr_{\om_\vphi}\theta
 -\gamma\frac{\om_\theta^n}{\om^n_\vphi}\leq -c_\gamma+\tr_{\om_\vphi}\theta.
 \end{align*}
 Then we get $-\tr_{\om_\vphi}\theta\leq C$, which implies the lower bound of $\om_\vphi$. Here we use the condition $(1)$ that $\theta<0$.

 Now we start to get the upper bound of $\om_\vphi$. We choose $\chi$ to be a K\"ahler cone metric.
We set $A=\tr_\chi\om_\vphi$ and compute 
\begin{align*}
(\p_t-L_\vphi)\log A
&=\frac{\p_t A+g_\vphi^{k\bar j}g_\vphi^{i\bar l}\theta_{i\bar j}A_{k\bar l}}{A} -\frac{g_\vphi^{k\bar j}g_\vphi^{i\bar l}\theta_{i\bar j}A_kA_{\bar l}}{A^2}\nonumber\\
&-\gamma\frac{\om_\theta^n}{\om^n_\vphi}
\frac{\tr_\vphi A}{A}+\gamma\frac{\om_\theta^n}{\om^n_\vphi}\frac{g_\vphi^{ k\bar l}A_kA_{\bar l}}{A^2}.\nonumber
\end{align*} 
After following the calculation (43-47) and Lemma 7 in \cite{MR3412393}, we have
\begin{align}\label{ptLcs}
(\p_t-L_\vphi)\log A
&\leq \frac{1}{A}\{\chi^{i\bar j}g^{p\bar q}_\vphi (\theta_{p\bar q})_{i\bar j}
+\gamma\chi^{i\bar j}R_{ i \bar j}(\om_\theta)\frac{\om_\theta^n}{\om^n_\vphi}\\
&
+ g_\vphi^{k\bar j}g_\vphi^{i\bar l}\theta_{i\bar j}{R^{p\bar q}}_{k\bar l}(\chi) g_{\vphi p\bar q}
-\gamma\frac{\om_\theta^n}{\om^n_\vphi}
g_\vphi^{ k\bar l} {R^{p\bar q}}_{k\bar l}(\chi) g_{\vphi p\bar q}\}\nonumber.
\end{align}

We examine each terms on the right hand side of \eqref{ptLcs}. The first term is bounded above, since $\om_\vphi$ has lower bound and $\theta$ is a smooth form.

We treat the third and forth terms, by choosing $\chi$ to be a K\"ahler cone metric in \thmref{Geometric asymptotic}. From the geometrically polyhomogeneous of $\chi$, there is a bounded function $\phi$ such that
$|R_{i\bar j k\bar l}(\chi )|\leq ( \chi_\phi)_{i\bar j} \cdot (\chi )_{k\bar l}
$ with $ \chi_\phi=C\cdot \chi +i\p\bar\p \phi$. As a result, these terms will be controlled by $L\phi$.

The second term involve $Ric(\om_\theta)$, which blows up only along the divisor. So we perturb the left hand side of \eqref{ptLcs} such that it could only achieve maximum outside the divisor.
We choose $S=\|s\|^{\kappa}$ with $\kappa\leq\alpha\beta$. Due to Lemma 2.4 in \cite{MR3405866}, $S$ goes faster than $$Z:=\log A-C\vphi-\phi.$$ Actually, $S$ goes to infinity as $z$ approaches the divisor $D$ and $Z$ is bounded. We also have from the same lemma that $i\p\bar\p S\geq -C_1\om_c$ that 
\begin{align*}
LS&=[-g_{\vphi}^{k\bar j}g_{\vphi}^{i\bar l}\theta_{k\bar l}+\gamma\frac{\om_\theta^n}{\om^n_{\vphi}}g_{\vphi}^{ i\bar j}]S_{i\bar j}\geq  -C_2.
\end{align*}

Now we rewrite the inequality above at the maximum point $p$ of $$\log A-C\vphi-\phi+S,$$ which must locate outside the divisor $D$. By using $Ric(\om_\theta)=\theta$ at $p$ (see \eqref{Rictheta}) and the lower bound of $\om_\vphi$ with respect to the cone metric $\om_\theta$, we get that
$\frac{1}{A}\gamma\chi^{i\bar j}R_{ i \bar j}(\om_\theta)\frac{\om_\theta^n}{\om^n_\vphi}$ is bounded above.

We then compute that
\begin{align*}
&(\p_t-L_\vphi)\vphi\\
&=-c_\gamma+2\tr_{\om_\vphi}\theta
-g_\vphi^{k\bar j}g_\vphi^{i\bar l}\theta_{i\bar j} g_{c k\bar l}
 -\gamma\frac{\om_\theta^n}{\om^n_\vphi}(n+1) 
 + \gamma\frac{\om_\theta^n}{\om^n_\vphi}\tr_{\om_\vphi}\om_c\nonumber.
\end{align*}

In conclusion, at the maximum point $p$ of $Z+S$, we obtain that
\begin{align}\label{ptLcs inequality}
0\leq&(\p_t-L)[\log A-C(\vphi-\vphi_c)-\phi+S](p)\\
&\leq C_2
-C[-c_\gamma+2\tr_{\om_\vphi}\theta
-g_\vphi^{k\bar j}g_\vphi^{i\bar l}\theta_{i\bar j} g_{c k\bar l} -\gamma\frac{\om_\theta^n}{\om^n_\vphi}(n+1) ].\nonumber
\end{align}

Diagonalise both $\om_c$ and $\om_\vphi$. By condition $(2)$, $$\sum_{i\neq k}\theta_{i\bar i}\geq 2\eps+c_\gamma.$$ 
In one case, we assume
$
-c_\gamma+2\tr_{\om_\vphi}\theta
-g_\vphi^{k\bar j}g_\vphi^{i\bar l}\theta_{i\bar j} g_{c k\bar l}\leq \eps
$. Written in terms of eigenvalues $\lambda_i$ of $\om_\vphi$, it becomes $$\sum_i(\lambda_i^{-2}-2\lambda_i^{-1})\theta_{i\bar i}\geq -c_\gamma-\eps.$$ 
Adding these two inequalities together, we have 
$-2\lambda_k^{-1}\theta_{k\bar k}\geq \eps$. So we get the upper bound of $\om_\vphi(p)$. 

In the other case, we assume 
$
-c_\gamma+2\tr_{\om_\vphi}\theta
-g_\vphi^{k\bar j}g_\vphi^{i\bar l}\theta_{i\bar j} g_{c k\bar l}>\eps.
$  Then \eqref{ptLcs inequality} implies that
\begin{align*}
-C\gamma(n+1)\frac{\om_\theta^n}{\om^n_\vphi}\leq C_2-C\eps.
\end{align*}
With the help of huge $C$ and lower bound of $\om_\vphi$, we also get the upper bound of $\om_\vphi(p)$.

Since $p$ is the maximum point of $\log A-C\vphi-\phi+S$ and the functions $ \phi, S$ are all bounded, we conclude that 
\begin{align*}
\om_\vphi(x)\leq C_3 e^{ C_4(\vphi(x)-\inf\vphi)}\om_c(x), \quad\forall x\in M.
\end{align*}

In order to obtain the $L^\infty$-estimate, it is sufficient to modify the proof in Section 10.3 in \cite{MR3412393} by replacing the smooth metric $\om$ there with a K\"ahler cone metric. Since the potential $\vphi$ is a $C^{2,\a,\beta}$ function, the integration by parts formula in the proof still works. Then we could follow the argument there to obtain the $L^\infty$-estimate. 

\end{proof}


\begin{defn}
The log $K$-energy is said to be $J$-proper in a cohomology class $[\om]$, if there are two positive constants $A$ and $B$ such that $\nu_\beta(\vphi)\geq AJ_{\om}(\vphi)-B$ for all $\vphi\in  \mathcal H_\beta(\om)$.
\end{defn}
In particular, we let $[\om]$ be K\"ahler and $\mathcal H_\beta(\om)\subset PSH(\om)$.
Then we interpret the conditions in \thmref{constgeodesicray big} into cohomology conditions and make use of it to obtain properness of the log $K$-energy.

\begin{prop}\label{properness criteria}Let $\Om=[\om_K]$ be a K\"ahler class. Assume that there is a constant $\eta$ satisfies that
\begin{align}\label{cohomology conditions Kahler}
\left\{
\begin{array}{lcl}
	&(i)&0\leq \eta<\frac{n+1}{n}\alpha_\beta,\\
	&(ii)& C_1(X,D)<\eta\Omega,\\
	&(iii)& (-n\frac{C_1(X,D)\cdot\Omega^{n-1}}{\Omega^{n}}+\eta)\Om+(n-1)C_1(X,D)>0.
\end{array}
\right.
\end{align}
Then the log $K$-energy is $J$-proper in $\Om$. Precisely, we have.
\begin{align*}
\nu_\beta(\vphi)
&\geq (\frac{n+1}{n}\alpha_\beta-\eta)J_{\om_K}(\vphi)-B, \quad \forall \vphi\in \mathcal H_\beta(\om_K).
\end{align*}
The constant $B$ depends on the constant $C$ in Definition \eqref{log alpha invariant} and the lower bound of $J_{-\theta,\eta}(\vphi)=J_{-\theta}(\vphi)+\eta J_{\om_K}(\vphi)$.
\end{prop}

\begin{proof}
We apply \thmref{constgeodesicray big} with $\gamma=0$ and choose $[\tilde\theta]=C_1(X,D)-\eta\Om$.
We need to derive the following conditions from the cohomology conditions \eqref{cohomology conditions}, i.e. there exist $\tilde\theta$ such that
	\begin{enumerate}
		\item $-\tilde\theta>0$,
		\item $(-n(\underline S_\b-\eta)\cdot\om_K+(n-1)\tilde\theta)\wedge\om_K^{n-2}>0$.
	\end{enumerate}

From $(ii)$ in \eqref{cohomology conditions}, we choose $\theta_1=Ric(\om_{\theta_1})-2\pi(1-\beta)[D]\in C_1(X,D)$ and a K\"ahler cone metric $\om_1\in\Om$ such that 
\begin{align*}
\tilde\theta_1=\theta_1-\eta\om_K<0.
\end{align*}
By $(iii)$ in \eqref{cohomology conditions}, there is $\theta_2\in C_1(X,D)$ and a K\"ahler cone metric $\om_2\in\Om$ such that
\begin{align*}
(-n\underline S_\b+\eta)\om_2+(n-1)\theta_2>0.
\end{align*}
We set 
\begin{align*}-n(\underline S_\b-\eta)\om_3=(-n\underline S_\b+\eta)\om_2+(n-1)\theta_2-(n-1)\tilde\theta_1.
\end{align*}
It is direct to verify that $\om_3$ is a K\"ahler cone metric in $\Om$ and 
\begin{align*}-n(\underline S_\b-\eta)\om_3+(n-1)\tilde\theta_1>0\end{align*}
So, the condition $(2)$ is satisfied, as long as we set $\om_c=\om_3$, $\theta=\theta_1$ and $\tilde\theta=\tilde\theta_1$. Then we could solve
\begin{align*}
n\cdot\tilde \theta\wedge\om^{n-1}_\vphi=c_0 \cdot \om^{n}_\vphi,
\quad c_0=n\cdot(\underline S_\b-\eta),
\end{align*}
which is the critical equation of the functional
\begin{align*}
J_{-\theta,\eta}(\vphi)=J_{-\theta}(\vphi)+\eta J_{\om_K}(\vphi).
\end{align*}

By \eqref{log K energy}, the log $K$-energy has the following formula
\begin{align*}
\nu_\beta(\vphi)
&:=E_\beta(\vphi)
+J_{-\theta}(\vphi)+\frac{1}{V}\int_M (\mathfrak h+h_K)\om_K^n.
\end{align*}
In which, $\mathfrak h:=-(1-\b)\log |s|_h^2$. The notation $h$ is the Hermitian metric on the line bundle $L_D$, $s$ is a section of $L_D$.

Making use of the log alpha invariant in Definition \ref{log alpha invariant}, Jensen's inequality (see Lemma 5 in \cite{MR3412393}) and \eqref{IJ equivalent}, to obtain the lower bound of the entropy $E_\beta(\vphi)$, we have
\begin{align*}
\nu_\beta(\vphi)
&\geq (\frac{n+1}{n}\alpha_\beta-\eta)J_{\om_K}(\vphi)-C +\inf_{\vphi\in\mathcal H_\beta(\om_K)} J^\eta_{-\theta}(\vphi), \quad \forall \vphi\in \mathcal H_\beta(\om_K).
\end{align*}
In which, the functional 
$
J^\eta_{-\theta}(\vphi)=J_{-\theta}(\vphi)+\eta J_{\om_K}(\vphi).
$
Then, \thmref{constgeodesicray big} implies that the functional $J^\eta_{-\theta}(\vphi)$ has lower bound. So we have proved that the log $K$-energy is $J$-proper.
\end{proof}

 \begin{rem}
 When $\gamma=0$, the twisted $J$-flow \eqref{jflow} is reduced to the $J$-flow, which was introduced in \cite{MR1772078,MR1701920}. The conditions for convergence of the $J$-flow was proved in \cite{MR2368374}.
 \end{rem}
 
 \begin{rem}
 As shown in \cite{MR3412393}, the lower bound of the twisted $J$-functional $J^\gamma_{-\chi}(\vphi)$ is related to the geodesic stability.
 \end{rem}

\subsection{Normal complex spaces}

We will construct reference metrics on normal complex spaces. The difficulties come from both the cone singularities and the degeneration of $\om_{sr}$.

Let $Y$ be a complex space, which is defined to be an analytic space over the field of complex numbers. We assume $Y$ is normal, that is at every point $p$ on $Y$, there is a neighbourhood $U_p$ of $p$ such that there exists an analytic covering from $U_p$ onto a domain of $\mathbb C^n$. 

The concepts of plurisubharmonic functions and K\"ahler currents are extended to the normal complex space \cite{MR813252,MR569410}.
\begin{defn} Let $U_i$ be a covering of $Y$ and $\sigma_i: U_i\rightarrow \mathbb C^n$ be local embeddings of $Y$. A function $\vphi$ is plurisubharmonic on $Y$, if $\vphi$ extends to a plurisubharmonic on an open neighbourhood of $\sigma_i(U_i)$. K\"ahler metric on $Y$ is defined in a similar way, i.e. locally it is written as $i\p\bar\p\vphi$ for some plurisubharmonic function $\vphi$.
\end{defn}
By normality, the plurisubharmonic function on $Y_{reg}$ could be extended to a plurisubharmonic function on $Y$.

We set $\om^Y_K$ be a K\"ahler metric on $Y$. Let $$\pi: X\rightarrow Y$$ be a resolution of singularities of $Y$, c.f. \cite{MR0199184}. The pull-back $X$ is smooth. The pull back metric $\om_{sr}=\pi^\ast\om^Y_K$ is then a semipositive and big K\"ahler metric, i.e. $\int_{X} \om_{sr}^n>0$. Let $\Om$ be the pull back of the K\"ahler class $[\om^Y_K]$ on $Y$. Then $\Om$ is a semipositive and big class.
By definition, normality implies 
\begin{align*}
\pi^\ast PSH(Y,\om^Y_K)=PSH(X,\om_{sr}).
\end{align*}
\begin{defn}Let $V=\int_Y(\om^Y_K)^n$.
A $\om^Y_K$-psh function $\varphi\in PSH(Y,\omega^Y_K)$ is said to have \textit{full Monge-Amp\`ere mass}, if 
\begin{align*}
\lim_{j\rightarrow\infty}\int_{\{\varphi>-j\}}(\omega^Y_K+i\p\bar\p\max\{\varphi,-j\})^n=V.
\end{align*}
The Monge-Amp\`ere operator is then well-defined over such $\vphi$ as
\begin{align*}
(\om_K^Y+i\p\bar\p\vphi)^n:=\lim_{j\rightarrow\infty}\mathbf 1_{\{\vphi>-j\}}(\omega^Y_K+i\p\bar\p\max\{\varphi,-j\})^n.
\end{align*}
We denote by $PSH_{full}(Y,\omega^Y_K)$ the set of $\om^Y_K$-psh functions with full Monge-Amp\`ere mass and by $\Om_{full}(Y,\tri)$ the set of corresponding currents.

The $\mathcal E^1$-space is defined to be
\begin{align*}
\mathcal E^1(Y,\om_K^Y)=\{\vphi\in PSH_{full}(Y,\omega^Y_K)\vert \vphi\in L^1(\om_K^Y+i\p\bar\p\vphi)\}.
\end{align*} 
\end{defn}

We set the pair $(Y,\tri)$ consist of a connected normal complex projective variety $Y$ and a Weil $Q$-divisor $\tri$.
Assume $K_Y+\tri$ is $Q$-Cartier, that is there is a positive integer $r$ such that $r(K_Y+\tri)$ is Cartier.
\begin{defn}
A \textit{log resolution} $\pi:X\rightarrow Y$ of $(Y,\tri)$ gives
\begin{align*}
K_{X}=\pi^\ast(K_Y+\tri)+D.
\end{align*}
In which, $a_i\in \mathbb Q$ is called the  \textit{discrepancy} of $(Y,\tri)$ along $E_i$. Actually,
\begin{itemize}
\item $D:=\sum_{i}a_iE_i$ is a $\mathbb Q$-divisor and $\cup_iE_i$ has normal crossing,
\item $\pi_\ast D=-\tri$.
\end{itemize} 
\end{defn}




\subsection{Approximate reference metrics}\label{Approximate reference metrics}


Kodaira's Lemma implies that $\Om-a_0 [E]$ is ample, in which $a_0$ is a sufficient small number and $E$ is an effective divisor. Let $S_E$ be the defining section of $E$ and $h_E$ is a smooth Hermitian metric on the line bundle associated to $E$. Then we have a K\"ahler metric
\begin{align}
\om_{K}=\om_{sr}+a_0 i\p\bar\p\log h_E>0.
\end{align}

We construct the approximate reference metrics in the classes $\Om_t:=\Om+t[\om_{K}]$ for $t>0$.
When $\Om$ is big and semipositive, we set the K\"ahler metric
\begin{align}\label{big semipositive kahler metric}
\om_t:=\om_{sr}+t\om_{K}\in \Om_t.
\end{align}


We extend Section \ref{Approximation of the reference metric} to the degenerate case. In order to construct approximate metric, we solve the conical version of the prescribed Ricci curvature problem for $\om_{\theta_t}:=\om_{t}+i\p\bar\p\vphi_{\theta_t}$ in each $\Om_t$,
\begin{align}\label{lift reference metric log Fano app both 2}
\om_{\theta_t}^n=\frac{e^{h_t}}{|s|^{2(1-\beta)}_h}\om_{t}^n, \quad Vol(\Om_t)=\int_X\frac{e^{h_{t}}}{|s|_h^{2(1-\beta)}}\om_{t}^n.
\end{align}
In which, by the cohomology condition, $h_t$ satisfies the identity
\begin{align*}
Ric(\om_{t})=\theta+(1-\beta)\Theta_D+i\p\bar\p h_t.
\end{align*} 
\begin{lem}\label{fix right hand side}We fix $t_0$.
Then there is constant $c_t$ such that $e^{h_t+c_t}\om_{t}^n=e^{h_{t_0}}\om_{t_0}^n$ for any $0<t\leq 1$. Moreover, $e^{c_t}=\frac{Vol(\Om_{t_0})}{Vol(\Om_t)}$.
\end{lem}
\begin{proof}
By \eqref{lift reference metric log Fano app both 2}, we have
\begin{align*}
Ric(\om_{\theta_t})=-i\p\bar\p \log(e^{h_t}\om_t^n)+i\p\bar\p|s|^{2(1-\beta)}_h.
\end{align*} 
Then making use of the definition of $h_t$ and the Poincar\'e-Lelong formulae, we get
\begin{align*}
Ric(\om_{\theta_t})=\theta+2\pi(1-\beta)[D].
\end{align*} 
Making use of $Ric(\om_\theta)=Ric(\om_{\theta_t})$, we have proved the lemma.
\end{proof}
As \eqref{thetaeps} in Section \ref{Approximation of the reference metric}, we approximate \eqref{lift reference metric log Fano app both 2} with a family of smooth equations. Due to \lemref{fix right hand side}, the approximate equation for $\om_{\theta_{t,\eps}}:=\om_{t}+i\p\bar\p\vphi_{\theta_{t,\eps}}\in \Om_t$ reads,
\begin{align}\label{lift reference metric log Fano app both 3}
\om_{\theta_{t,\eps}}^n=\frac{e^{h_{t_0}+c_{t,\eps}}}{(|s|^{2}_h+\eps)^{1-\beta}}\om_{t_0}^n.
\end{align}
Here $c_{t,\eps}$ is a constant determined by the normalised volume condition \eqref{thetaeps normalisation}, i.e.
\begin{align}\label{thetaeps normalisation semipositive}
e^{c_{t,\eps}}=\frac{Vol(\Om_t)}{\int_Xe^{h_{t_0}}(|s|^2_h+\eps)^{\beta-1}\om^n_{t_0}}.
\end{align} It is uniformly bounded independent of $t, \eps$.
From \lemref{thetaepstheta}, we further have
\begin{align}\label{big nef lower Ricci}
Ric(\om_{\theta_{t,\eps}})\geq\tilde \theta:=\theta+\min\{(1-\beta)\Theta_D,0\}.
\end{align}

Let $\om_{\mathbf b_\eps}$ be the background cone metric defined in Section 2 on page 10 in \cite{MR4020314}.
\begin{prop}\label{big nef approximate reference metric bound}
There exists a uniform constant $C$ independent of $t,\eps$ such that
\begin{align*}
\|\varphi_{\theta_{t,\eps}}\|_\infty\leq C, \quad 
C^{-1} |S_E|^{b}_{h_E}\cdot\om_{\bf b_\eps}\leq \om_{\theta_{t,\eps}}\leq C |S_E|^{-a}_{h_E}\cdot\om_{\bf b_\eps}.
\end{align*}
Moreover, the sequence $\varphi_{\theta_{t,\eps}}$ $L^1$-converges to $\vphi_\theta$, which is the unique solution to the equation
\begin{align}\label{lift reference metric log Fano app limit}
(\om_{sr}+i\p\bar\p\vphi_{\theta})^n=\frac{e^{h_{t_0}+c_{0,0}}}{|s|^{2(1-\beta)}_h}\om_{t_0}^n.
\end{align}
The solution $\vphi_\theta\in C^0(X)\cap C^{\infty}(X\setminus (E\cup D))\cap C^{2,\a,\b}(X\setminus E)$ and also geometrically polyhomogeneous along $D$ (see Section \ref{Regularity of cscK cone metrics: geometrical polyhomogeneity}).

\end{prop}
\begin{proof}

 \textbf{Step 1: $C^0$-estimate.}
 We set 
 \begin{align}\label{Phi degenerate}
 \om:=\om_{K}+t\om_K,\quad  \Phi:=\vphi_{\theta_{t,\eps}}-a_0 \log h_E.
 \end{align}
The upper bound of $\Phi$ is obtained by using 
$$\om+i\p\bar\p\Phi\geq0.$$ 
According to Corollary 4.9 in \cite{MR3761174}, letting $\underline{\Phi}=\frac{1}{V}\int_M\Phi\om^n$, we have 	
\begin{align}\label{zeroroughbig upper}
	\sup_M(\Phi-\underline{\Phi})\leq C_1.
\end{align}
The constant $C_1$ depends on $Vol(\om)$, $n$, the Sobolev
and Poincar\'e constant of $\om$. Since $\om$ is $C^2$ close to $\om_K$ as $t$ is close to zero, we have uniform control of the Sobolev
and Poincar\'e constant of $\om$ on parameter $t$. Thus $C_1$
 is independent of $\eps,t$.

We rewrite \eqref{lift reference metric log Fano app both 3} as
\begin{align}
(\om_t+i\p\bar\p\vphi_{\theta_{t,\eps}})^n=F_0\om_K^n \text{ with }F_0=\frac{e^{h_{t_0}+c_{t,\eps}}\om_{t_0}^n}{(|s|^{2}_h+\eps)^{1-\beta}\om_K^n}.
\end{align}
It is direct to see that $F_0\in L^p(\om_K)$ for some $p>1$ and 
$\|F_0\|^p$ is independent of $\eps,t$. By \eqref{zeroroughbig upper}, we could normalise $\vphi_{\theta_{t,\eps}}$ such its upper bound is zero.
Then the $C^0$-estimate 
\begin{align*}
\|\vphi_{t,\eps}\|_{\infty}\leq C, \quad \forall t,\eps
\end{align*}is obtained from Theorem A in \cite{MR2439574}, which is also proven independently in \cite{MR2647006}. We recall it as following.
\begin{lem}[\cite{MR2439574,MR2647006}]\label{Linfty estimate degenerate}
Let $(X,\om_K)$ and $(Y,\om_K^Y)$ be compact K\"ahler manifolds. Suppose $\pi: X\rightarrow Y$ be a non degenerate holomorphic mapping. Set $\om_{sr}=\pi^\ast \om_K^Y$ and $\om_t=\om_{sr}+t\om_K$, $0<t<1$. Suppose that $\vphi_t$ solve the equation 
\begin{align*}
(\om_t+i\p\bar\p\vphi_t)^n=c_t F_0 \om_K^n, \text{ with }\sup_X \vphi_t=0.
\end{align*}
The normalisation constant $c_t$ satisfies $\int_X\om_t^n=c_t\int_XF_0 \om_K^n$.
Suppose that $F\in L^p(\om_K)$ for some $p>1$. Then the solution $\vphi_t$ has uniform $C^0$-estimate and
\begin{align*}
\|\vphi_t\|_{\infty}\leq C, \quad \forall 0<t\leq 1.
\end{align*}
The constant $C$ depends on $p, \pi, \|F_0\|_p$.
\end{lem}

 \textbf{Step 2: Laplacian estimate.}
 The proof is similar to Lemma 3.25 in \cite{MR4020314}. As shown in Section 2 on page 10 in \cite{MR4020314}, the background metric $\om_{\mathbf b_\eps}=\om+i\p\bar\p \Phi_{\mathbf b_\eps}$ satisfies
\begin{align}\label{backgroundmetriccurvature}
\left\{
\begin{array}{ll}
&R_{i\bar j k\bar l}(\om_{\bf b_\eps} )\geq -(\tilde g_{\mathbf b_\eps})_{i\bar j} \cdot (g_{\bf b_\eps} )_{k\bar l}\\
&\tilde \om_{\mathbf b_\eps}=C_2\cdot \om_{\bf b_\eps} +i\p\bar\p \Phi^\eps_{\mathbf b}\geq 0,\\
&\|\Phi^\eps_{\mathbf b}\|_\infty\leq C_3.
\end{array}
\right.
\end{align}

We let $\tilde\Phi_{\mathbf b_\eps}=\Phi-\Phi_{\mathbf b_\eps}$ and omit the index $\eps$ from now on. We rewrite \eqref{lift reference metric log Fano app both 3} as
\begin{align}\label{lift reference metric log Fano app both 3 change}
\om_{\Phi}^n=(\om_{\mathbf b}+i\p\bar\p\Phi_{\mathbf b})^n=e^{F_{\bf b}}\om_{\mathbf b}^n \text{ with }e^{F_{\bf b}}=\frac{e^{h_{t_0}+c_{t,\eps}}\om_{t_0}^n}{(|s|^{2}_h+\eps)^{1-\beta}\om_{\mathbf b}^n}.
\end{align}
By $(2.11)$ in \cite{MR4020314}, we have 
\begin{align}\label{backgroundmetriccurvature additional}
i\p\bar\p F_{\mathbf b}\geq -(C\om_{\bf b}+i\p\bar\p\Phi^\eps_{\mathbf b}).
\end{align}

We choose 
We applying the inequality from Yau's second order estimate \cite{MR480350}, \begin{align*}
\tri_{\Phi} \log\Tr_{\om_{\bf b}}\om_{\Phi}\geq \frac{- g_{\bf b}^{i\bar j}R_{i\bar j}(\om_{\Phi})+g_\Phi^{k\bar l}{R^{i\bar j}}_{k\bar l}(\om_{\bf b}){g_\Phi}_{i\bar j}}{\Tr_{\om_{\bf b}}\om_{\Phi}}.
\end{align*}
By the equation \eqref{lift reference metric log Fano app both 3 change},
we have
\begin{align*}
Ric(\om_{\Phi})=Ric(\om_{\bf b})-i\p\bar\p F_{\bf b}
\end{align*}
and then
\begin{align*}
RHS=\frac{\tri F_{\bf b}-S(\om_{\bf b})+g_\Phi^{k\bar l}{R^{i\bar j}}_{k\bar l}(\om_{\bf b}){g_\Phi}_{i\bar j}}{\Tr_{\om_{\bf b}}\om_{\Phi}}.
\end{align*}

From \eqref{backgroundmetriccurvature additional}, we have
\begin{align*}
\tri_{\om_{\mathbf b}} F_{\mathbf b}
\geq-  \Tr_{\om_{\mathbf b}}\om_{\Phi}\cdot \Tr_{\om_{\Phi}}{(C\om_{\mathbf b}+i\p\bar\p\Phi^\eps_{{\mathbf b}})} .
\end{align*}
By $(2.28)$ in \cite{MR4020314}, we get
\begin{align}\label{RHS2ndinequality}
\frac{-S(\om_{\bf b})+g_\Phi^{k\bar l}{R^{i\bar j}}_{k\bar l}(\om_{\bf b}){g_\Phi}_{i\bar j}}{\Tr_{\om_{\bf b}}\om_{\Phi}}\geq -C(\Tr_{\om_{\Phi}}\om_{{\mathbf b}}+\tri_{\Phi}\Phi^\eps_{{\mathbf b}}).
\end{align}
In summary, we obtain that
\begin{align*}
\tri_{\Phi} (\log\Tr_{\om_{\bf b}}\om_{\Phi}+C\Phi^\eps_{{\mathbf b}})\geq -C\Tr_{\om_{\Phi}}\om_{{\mathbf b}}.
\end{align*}
Letting $$Z=\log\Tr_{\om_{\bf b}}\om_{\Phi}+C\Phi^\eps_{{\mathbf b}}-(C+1)\Phi_{\bf b},$$ we have
\begin{align*}
\tri_{\Phi} Z
\geq \Tr_{\om_{\Phi}}\om_{{\mathbf b}}-n(C+1).
\end{align*}
Since $-\Phi_{{\mathbf b}}=\vphi_{\theta_{t,\eps}}-a_0 \log h_E-\Phi^\eps_{{\mathbf b}}$ is $-\infty$ along $E$, the maximum point $p$ of $Z$ is achieved outside $E$ and we have 
\begin{align*}
\Tr_{\om_{\Phi}}\om_{{\mathbf b}}(p)\leq n(C+1).
\end{align*}
The inequality of arithmetic and geometric means implies that at $p$, 
\begin{align*}\Tr_{\om_{{\mathbf b}}  }\om_\Phi
	\leq \frac{n}{n-1}(\Tr_{\om_\Phi}\om_{{\mathbf b}})^{n-1} \cdot \frac{\om_\Phi^{n}}{\om_{{\mathbf b}}^{n}}=\frac{n}{n-1}(\Tr_{\om_\Phi}\om_{{\mathbf b}})^{n-1} e^{ F_{\bf b}}.
	\end{align*} 
So at any point $x\in X$, letting $L:=C\Phi^\eps_{{\mathbf b}}-(C+1)\tilde\Phi_{\bf b}$, we use the bound of $\Phi^\eps_{{\mathbf b}}$ in \eqref{backgroundmetriccurvature} to have
\begin{align*}
&\log\Tr_{\om_{\bf b}}\om_{\Phi}(x)\\
&\leq \log\Tr_{\om_{\bf b}}\om_{\Phi}(p)+L(p)-(2C+1)\Phi^\eps_{{\mathbf b}}+(C+1)(\vphi_{\theta_{t,\eps}}-a_0 \log h_E),
\end{align*}
or
\begin{align*}
\Tr_{\om_{\bf b}}\om_{\Phi}(x)\leq C |S_E|^{-a}_{h_E}e^{A \vphi_{\theta_{t,\eps}}}.
\end{align*}
Here $a=(C+1)a_0$.

By the inequality of arithmetic and geometric means again, 
\begin{align*}\Tr_{\om_\Phi  }\om_{{\mathbf b}}
	&\leq \frac{n}{n-1}(\Tr_{\om_{\mathbf b}}{\om_\Phi})^{n-1} \cdot \frac{\om_{{\mathbf b}}^{n}}{\om_\Phi^{n}}=\frac{n}{n-1}(\Tr_{\om_{\mathbf b}}{\om_\Phi})^{n-1}  e^{ -F_{\bf b}}\\
	&\leq C |S_E|^{-a(n-1)}_{h_E}.
	\end{align*} 
In conclusion, we have
	\begin{align*}
C^{-1}  |S_E|^{b}_{h_E}\cdot\om_{\bf b}\leq \om_{\Phi}\leq C |S_E|^{-a}_{h_E}\cdot\om_{\bf b}.
\end{align*}

 \textbf{Step 3: higher order estimates and convergence.}
Due to the Evans-Krylov estimate and Schauder estimate, we further have the $C^k$-estimates on $X\setminus E$. Especially, the estimates are uniform on each compact set contained in $X\setminus E$.
 
 We then consider the limit equation \eqref{lift reference metric log Fano app limit}. Since $[\om_{sr}]$ is big and semipositive, and the right hand side of \eqref{lift reference metric log Fano app limit} is $L^p(X)$ for some $p>1$, there exists a unique solution $\vphi_\theta\in PSH(\om_{sr})\cap L^\infty(X)$ to \eqref{lift reference metric log Fano app limit} by Theorem A in \cite{MR2505296}, which is quoted as following.
 \begin{lem}[Theorem A in \cite{MR2505296}]\label{Theorem A in 50}
 Let $X$ be a compact K\"ahler manifold, $\om$ be big and semipositive, $f\geq 0$ be a $L^p(\om), p>1$ function with $\int_X f\om^n=\int_X\om^n$. Then there is a unique bounded solution $\vphi\in PSH(\om)$ to
 \begin{align*}
(\om+i\p\bar\p\vphi)^n=f\om^n.
\end{align*}
under normalisation condition $\sup_X\vphi=0$.
 \end{lem}
 
 From Hortog's lemma and the unqiness theorem in \cite{MR2669357}, we have $\varphi_{\theta_{t,\eps}}\rightarrow \vphi_\theta$ in $L^1(X)$, and moreover $\varphi_{\theta_{t,\eps}}\rightarrow \vphi_\theta$ in $C^k(U)$ for $\forall k\geq 0$ and any compact subset $U\subset X\setminus E$. So $\varphi_\theta$ is in $C^0(X)\cap C^{\infty}(X\setminus (E\cup D))\cap C^{2,\a,\b}(X\setminus E)$. The geometrical polyhomogeneity of $\vphi_\theta$ along $D$ follows from \cite{MR3911741}.
 
 \end{proof}

According to Proposition \ref{big nef approximate reference metric bound},
\begin{cor}
On $Y_{reg}\setminus \supp \tri$, the $\om_K^Y$-psh-function $\pi_\ast\vphi_\theta$ is smooth and satisfies 
\begin{align*}
(\om_K^Y+i\p\bar\p\pi_\ast\vphi_{\theta})^n=\frac{e^{h_{t_0}+c_{0,0}}}{|s|^{2(1-\beta)}_h}\om_{t_0}^n.
\end{align*}
\end{cor}
\begin{proof}
Each fibre of the resolution $\pi$ is connected and $\om_\vphi$ is semi-positive on each fibre, so $\vphi$ is constant along the fibre. As a result, $\vphi$ is bounded on $Y$. Furthermore, $\vphi$ is smooth and satisfies \eqref{singular cscK} on $Y_{reg}\setminus \supp \tri$.
\end{proof}

\begin{rem}
We could choose different effective divisor $E$ such that $\Om-\tilde a_0 [E]$ is ample for some sufficient small number $\tilde a_0$. The intersection of such effective divisors is contained in the complement of the ample locus $Amp(\Om)$ of $\Om$. So the conclusion of Proposition \ref{big nef approximate reference metric bound} will be restated as the following, that is the solution $\vphi_\theta\in C^0(X)\cap C^{\infty}(Amp(\Om)\setminus D))\cap C^{2,\a,\b}(Amp(\Om))$ and also geometrically polyhomogeneous along $D$ (see Section \ref{Regularity of cscK cone metrics: geometrical polyhomogeneity}).

In general, when $[\om_{sr}]$ is only big, by \cite{MR2746347}, there exists a unique closed positive current $\om_{\theta}$ satisfies \eqref{lift reference metric log Fano app limit}.
Also, $\om_{\theta}$ has minimal singularity and smooth on the ample locus of $\Om$.
\end{rem}

\subsection{Construct approximate cscK cone metrics for big classes}
\label{Construct approximate cscK cone metrics for big classes}
When we work in a big class, we could slightly deform it to a K\"ahler class by Kodaira's lemma.
The condition \eqref{cohomology conditions} could be perturbed slightly as well, since the alpha invariants continuously depend on the cohomology class as shown in Section \ref{Alpha invariants}. Then we would apply \thmref{properclosedness} to obtain a family of approximate cscK cone metrics in K\"ahler classes near a big class. Furthermore, such cscK cone metrics could be also approximated by smooth cscK metrics by Proposition \ref{twisted approximation thm}. 

\begin{defn}
We say a big class $\Om$ has the cscK approximation property, if there exists a family of class $\Om_t\rightarrow \Om$ such that $\Om_t$ admits cscK metrics.
\end{defn}
\begin{prop}\label{Big classes proper}Let $\Om$ be a big class and $\om_K$ be a K\"ahler metric. Suppose that $\Om$ satisfies the cohomology condition \begin{align}\label{cohomology conditions}
\left\{
\begin{array}{lcl}
	&(i)&0\leq \eta<\frac{n+1}{n}\alpha_\beta,\\
	&(ii)& C_1(X,D)<\eta\Omega,\\
	&(iii)& (-n\frac{C_1(X,D)\cdot\Omega^{n-1}}{\Omega^{n}}+\eta)\Om+(n-1)C_1(X,D)>0.
\end{array}
\right.
\end{align} Then there exists a sufficient small positive constant $\tilde t$ such that \eqref{cohomology conditions} holds for $\Om_t=\Om+t[\om_K]$ for any $0<t\leq\tilde t$.
\end{prop}
\begin{proof}
We check each criterion for $\Om_t$. Then $(ii)$ is obvious. The first criterion $(i)$ holds after perturbing $\alpha_\beta(\Om_t)$ with sufficient small $t$, according to \lemref{alphacontinuity}. Finally, $(iii)$ is still true when $t$ is sufficiently small. 
\end{proof}

\begin{rem}\label{properness critical case}
When $\eta \Om=C_1(X,D)$, the conditions $(ii)$ and $(iii)$ become identities. The family $\Om_t$ satisfies \eqref{cohomology conditions} for sufficient small $t$.
\end{rem}

\begin{prop}\label{Big classes approximate csck}Let $\Om$ be a big and nef class on a K\"ahler manifold $(X,\om_K)$, whose automorphism group $Aut(M)$ is trivial. Suppose that $\Om$ satisfies the cohomology condition \eqref{cohomology conditions}. Then $\Om$ has the cscK approximation property. 

Precisely, setting $\Om_t=\Om+t[\om_K]$, we obtain that for $0<t\leq\tilde t$,
	\begin{enumerate}
		\item
	the log $K$-energy is $J$-proper in $\Om_t$, 
	\item there exists a cscK cone metric $\om_{t}$ in $\Om_t$,
	\item the cscK cone metric $\om_{t}$ a smooth approximation $\om_{t,\eps}$ in $\Om_{t}$.
	\end{enumerate}
\end{prop}
\begin{proof}

With the help of the approximate reference metric $\om_{\theta_{t}}$ solving \eqref{lift reference metric log Fano app both 2}, the log $K$-energy on $\mathcal H_\b(\om_t)$ is written as
\begin{align}\label{log K energy big}
\nu_\beta(\vphi)
&:=E_\beta(\vphi)
+J_{-\theta}(\vphi)+\frac{1}{V}\int_M (\mathfrak h+h_t)\om_t^n.
\end{align}
Here the entropy term is
\begin{align*}
	E_\beta(\vphi)=\frac{1}{V}\int_M\log\frac{\om^n_{\vphi}}{\om^n_{\theta_{t}}}\om_{\vphi}^{n}.
\end{align*}
The $J$-functional is defined in term of the reference metric $\om_{\theta_t}$,
\begin{align*}
J_{-\theta}(\vphi)&:=j_{-\theta}(\vphi)+\ul{S}_{\beta,t}\cdot D_{\om_t}(\vphi),\quad
\underline S_{\b,t}=\frac{C_1(X,D)\Om_t^{n-1}}{\Om_t^n},
\end{align*}
in which,
\begin{align*}
j_{-\theta}(\vphi)&:=-\frac{1}{V}\int_{M}\vphi
\sum_{j=0}^{n-1}\om_t^{j}\wedge
\om_\vphi^{n-1-j}\wedge \theta,\\
D_{\om_t}(\vphi)&:=\frac{1}{V}\frac{1}{n+1}\sum_{j=0}^{n}\int_{M}\vphi\om_t^{j}\wedge\om_{\varphi}^{n-j}.
\end{align*}

The first statement is proved, as a direct corollary of Proposition \ref{Big classes proper} and Proposition \ref{properness criteria}. Precisely speaking, there exists a constant $\eta$ such that the log $K$-energy is proper in $\Om_t$, that is
\begin{align}\label{log k energy proper}
\nu_\beta(\vphi)
&\geq (\frac{n+1}{n}\alpha_\beta(\Om_t)- \eta )J_{\om_{t}}(\vphi)-C_t, \quad \forall \vphi\in \mathcal H_\beta(\om_{{t}}).
\end{align} The constant $C_t$ depends on the uniform constant $C$ in the Definition \ref{log alpha invariant} of log alpha invariant and the lower bound of the $J^\eta_{-\theta}$-functional, which reads 
\begin{align*}
J^\eta_{-\theta}(\vphi)=J_{-\theta}(\vphi)+\eta J_{\om_{t}}(\vphi).
\end{align*}

Actually, the $J$-properness implies the $d_1$-properness. The proof goes as following.
By \eqref{IJ equivalent}, $ J_{\om_{t}}(\vphi)\geq \frac{1}{n+1} I^A_{\om_{t}}(\vphi)$. We normalise $\vphi$ such that $D_{\om_t}(\vphi)=0$. Then by Proposition 2.8 in \cite{MR3830547}, there exists constants $p$ and $A_n$ which only depend on $n$ such that
$
I_{1,\om_t}\leq A_n (I^A_{\om_{t}})^p+\frac{1}{V}\int_X\vphi\om_t^n.
$
By using the formula $J^A_{\om_{t}}=\frac{1}{V}\int_X\vphi\om_{t}^n-D_{\om_t}(\vphi)$, we have 
\begin{align*} 
I_{1,\om_{t}}\leq A_n (I^A_{\om_{t}})^p+J^A_{\om_{t}}.
\end{align*} 
Then \eqref{i1d1} implies that $$2I_{1,\om_{t}}\geq d_{1,\om_t}.$$
Therefore, the log $K$-energy is $d_1$-proper and the existence of cscK cone metrics follows from \thmref{properclosedness} and we have proved the second statement.

The third statement is a consequence of Proposition \ref{twisted approximation thm} and \thmref{properclosedness}, i.e. there is a sequence of smooth cscK metric $\om_{t,\eps}$ converges to the cscK cone metric in $\Om_t$. 

Then by \lemref{approximate K and K lemma}, \lemref{approximation proper} and \eqref{log k energy proper}, we also have that the approximate $K$-energy is also proper in $\Om_t$,
\begin{align}\label{approximate proper big}
\nu^\eps_{\beta}(\vphi)
&:=\nu_{\beta,\chi}(\vphi)
+H_\eps-H\nonumber\\
&\geq (\frac{n+1}{n}\alpha_\beta(\Om_t)- \eta )J_{\om_{t}}(\vphi)-C, \quad \forall \vphi\in \mathcal H_\beta(\om_t).
\end{align} 
The constant $C$ depends on the lower bound of the $J^\eta_{-\theta}$-functional.

\end{proof}
Furthermore, the approximate metric $\om_{\vphi_\eps}:=\om_{\eps,\eps}=\om_\eps+i\p\bar\p\vphi_{\eps}$ satisfies the cscK equations
\begin{equation}\label{approximate singular cscK}
\left\{
\begin{aligned}
	&\om_{\vphi_\eps}^n=e^{F_\eps}\om_{\theta_\eps}^n,\\
	&\tri_{\vphi_\eps} F_\eps=\tr_{\om_{\vphi_\eps}}\theta-\underline S_{\b,\eps}.
\end{aligned}
\right.
\end{equation}
Here, $\om_\eps=\om_{sr}+\eps\om_K$ and $\om_{\theta_\eps}$ is given by \eqref{lift reference metric log Fano app both 3}. We will derive the a priori estimate of $\vphi_\eps$ in Section \ref{Linfty-estimates for degenerate metrics}.

\subsection{Singular cscK metrics}
In Section \ref{Construct approximate cscK cone metrics for big classes}, we consider the cscK approximation property of a big class $\Om$. In this section, we further consider the convergence of such cscK sequence $\om_{\vphi_\eps}\in \Om_\eps=\Om+\eps[\om_K]$.
\begin{defn}\label{singular cscK defn}
We say $\omega_\vphi=\om_{sr}+i\p\bar\p\vphi$ is a singular cscK metric in a big class $\Om=[\om_{sr}]$ on a K\"ahler manifold $X$, if 
\begin{itemize}
\item $\vphi\in PSH(\om_{sr})$,
\item $\vphi$ is a $L^1$-limit of a sequence of smooth cscK potentials. 
\end{itemize}
Moreover, if $\vphi$ is $\mathcal E^1(\om_{sr})\cap L^\infty$, we say $\om_\vphi$ is a bounded singular cscK metric.
\end{defn}

The notion then is extended to a normal complex space $Y$.
\begin{defn}
We say $\omega_\vphi=\om^Y_K+i\p\bar\p\vphi$ is a singular cscK metric in a K\"ahler class $\Om=[\om^Y_K]$ on a normal complex space $Y$, if 
\begin{itemize}
\item $\vphi\in PSH(\om^Y_K)$,
\item $\vphi$ is a $L^1$-limit of a sequence of smooth cscK potentials. 
\end{itemize}
Similarly, we say is $\omega_\vphi$ is a bounded singular cscK metric, if $\vphi$ is $\mathcal E^1(\om^Y_K)\cap L^\infty$.
\end{defn}

We first show that the existence of singular cscK metric follows from Proposition Proposition \ref{Big classes approximate csck} directly.

\begin{thm}\label{singular csck construction weak}
Let $\om_{\vphi_\eps}:=\om_{\eps,\eps}=\om_\eps+i\p\bar\p\vphi_{\eps}$ be the diagonal sequence of the approximate cscK metrics obtained in Proposition \ref{Big classes approximate csck}. 

Then $\Om$ admits a singular cscK metric on $X$.
\end{thm}
\begin{proof}
We also could take $\sup_X\vphi_\eps=0$ and let $\vphi_\eps\rightarrow\vphi$ in $L^1(X)$-topology. By Hartogs' lemma, $\vphi$ is a $\om_{sr}$-psh function and $\sup_X\vphi=0$. 
\end{proof}

 We next construct \textit{bounded} singular cscK metric.
\begin{prop}\label{singular csck construction}
Given the same assumption in Proposition \ref{singular csck construction weak}. Assume $C_1(X,D)\geq 0$.
Suppose the entropy $E^\eps_\beta(\vphi_\eps)=\frac{1}{V}\int_X F_\eps e^{F_\eps}\om^n_{\theta_\eps}$ of $\om_{\vphi_\eps}$ is bounded.

Then $\Om$ admits a bounded singular cscK metric on $X$. Moreover, we have the 2nd estimates
\begin{align*}
\int_X(\tr_{\om_{\mathbf b_\eps}}\om_{\vphi_\eps})^{p}|s|^{\sigma}_{h_E}\om^n_{\mathbf b_\eps}
\leq C, \quad \forall p\geq 1 .
\end{align*}
The weight $\sigma$ is defined in \eqref{sigma L sigma}.
\end{prop}
\begin{proof}
We choose $\theta\geq 0$ in $C_1(X,D)$.
The $L^\infty$-estimate for degenerate metrics (Proposition \ref{upper F degenerate} in Section \ref{Linfty-estimates for degenerate metrics}) tells us that the approximate cscK metric \eqref{approximate singular cscK} have the following estimates,
	\begin{align*}
	\sup_X (F_\eps+\phi_l), \quad \|e^{F_\eps}\|_{p}, \quad \|\vphi_\eps\|_\infty\leq C.
	\end{align*}
	The constant $C$ depends on 	\begin{align*}
	E^\eps_\beta, \quad \sup_X \underline S_{\b,\eps}, \quad \inf_{(X,(1+\eps)\om_K)}\theta,\quad \|\phi_l\|_\infty, \quad \alpha_1,\quad \alpha_\beta,\quad n.
	\end{align*}
	Here $\phi_l= \inf_{(X,(1+\eps)\om_{K})}\theta \cdot (\vphi_{\theta_\eps}-a_0 \log h_E)
$.
By assumption, $E^\eps_\beta(\vphi_\eps)$ is bounded. So, we have the uniform bound of $\|\vphi_\eps\|_{\infty}$. Moreover, $\vphi_\eps\rightarrow\vphi$ in $L^p(X)$-topology for any $p\geq 1$.

We take the diagonal sequence 
$\om_{\theta_{\eps}}:=\om_{\eps}+i\p\bar\p\vphi_{\theta_{\eps}}\in \Om_\eps$ of the approximate reference metrics defined in \eqref{lift reference metric log Fano app both 3}, which solve
\begin{align}\label{lift reference metric log Fano app both 3 diagonal}
\om_{\theta_{\eps}}^n=\frac{e^{h_{t_0}+c_{\eps,\eps}}}{(|s|^{2}_h+\eps)^{1-\beta}}\om_{t_0}^n.
\end{align}
We let $\mu_\eps=\frac{e^{h_{t_0}+c_{\eps,\eps}}}{(|s|^{2}_h+\eps)^{1-\beta}}\om_{t_0}^n$ and $\mu=\frac{e^{h_{t_0}+c_{0,0}}}{|s|^{2(1-\beta)}_h}\om_{t_0}^n$. By \eqref{approximate singular cscK}, we have
\begin{align*}
\om_{\vphi_\eps}^n=e^{F_\eps}\om_{\theta_\eps}^n=e^{F_\eps}\mu_\eps.
\end{align*}
Since $F_\eps$ is $L^p$ bounded, we have $e^{F_\eps}\rightarrow e^F$ with $F\in L^p$. We define $\tilde\vphi_\eps$ to be the upper semicontinuous regularisation of $\sup_{\tilde\eps\geq\eps}\vphi_{\tilde\eps}$. By a result in Demailly \cite{MR1211880}, we have $\tilde\vphi_\eps$ decreases to $\vphi$ and 
\begin{align}\label{singular csck half inequality}
(\om_\eps+i\p\bar\p{\tilde\vphi_\eps})^n\geq e^{F_\eps}\mu_\eps.
\end{align}
Since we have uniform $L^\infty$ bound of $\vphi_\eps$, so we normalise it such that $\sup_X\vphi_\eps=0$. Furthermore, we apply Lemma 2.3 in \cite{MR2352488} with $\chi=id$ to get the control of $E_1=\int_X-\tilde\vphi_\eps(\om_\eps+i\p\bar\p\tilde\vphi_\eps )^n$ as following
\begin{align*}
0\leq E_1\leq \int_X-\vphi_\eps(\om_\eps+i\p\bar\p\vphi_\eps )^n.
\end{align*}
The last term in the inequality above is bounded by the $L^\infty$-estimate of $\vphi_\eps$ again. As a result, Proposition 1.2 in \cite{MR2505296} implies that
\begin{align*}
(\om_\eps+i\p\bar\p{\tilde\vphi_\eps})^n\rightarrow \om_\vphi^n.
\end{align*}
Taking $\eps\rightarrow 0$ in \eqref{singular csck half inequality}, we have achieved the following inequality
\begin{align*}
\om_\vphi^n\geq e^{F}\mu.
\end{align*}
But the normalisation condition of both sides and Proposition \ref{big nef approximate reference metric bound} tell us that the inequality is identity, i.e.
\begin{align*}
	\om_{\vphi}^n=e^{F}\om_{\theta}^n.
\end{align*}
In this identity, the right hand side is $L^p, p>1$. 
Thus, we apply \lemref{Theorem A in 50} to obtain that $\vphi$ is bounded.
In summary, we have 
\begin{align*}
\vphi\in \mathcal E^1(X,\om_{sr})\cap L^\infty(X).
\end{align*}

We could say more about the approximate metric $\om_{\vphi_\eps}$ given in Proposition \ref{singular csck construction}, regarding to the background metric $\om_{\mathbf b_\eps}$ defined in \eqref{backgroundmetriccurvature}.
According to \thmref{w2pestimates degenerate}, we have the 2nd estimates for any $p\geq 1$.
\end{proof}

\begin{ques}
Could we obtain Laplacian estimate or higher order estimates of the singular cscK metric $\omega_\vphi=\om^Y_K+i\p\bar\p\vphi$? Does it satisfy the following couple of equations on the regular part $Y_{reg}\setminus \tri$,
\begin{equation}\label{singular cscK}
\left\{
\begin{aligned}
	&\om_\vphi^n=e^F\om_\theta^n,\\
	&\tri_{\vphi} F=\tr_{\om_\vphi}\theta-\underline S_\tri.
\end{aligned}
\right.
\end{equation} 
Note that with appropriate regularity, this equation is equivalent to $S(\om_\vphi)=\underline S_\tri$ on $Y_{reg}\setminus \tri$.
Then we could lift \eqref{singular cscK} to $\pi^\ast(Y_{reg}\setminus \supp \tri)$, which stays in the ample locus $Amp(\Om)$. 
\end{ques}
\subsection{Examples}
In this Section, we construct singular cscK metrics on log Fano pairs.
We set $\tri$ to be an effective $Q$-divisor. The divisor $D$ has two parts. One part is with coefficients in $(-1,0]$. Then by definition, the cone angle $\beta_i$ is equal to $1+\alpha_i$. So we have $\beta\in (0,1]$. The other part is the $\pi$-exceptional divisor, which is effective and with integer coefficients.
\begin{defn}
The pair $(Y,\tri)$ is \textit{Kawamata log terminal (klt)}, if $a_i>-1$ for all $i$. When $\tri=0$, $Y$ is said to be \textit{log terminal}, if $(Y,0)$ is klt.
\end{defn}
\begin{defn}
A klt pair $(Y,\tri)$ is called a \textit{log Fano pair}, if $-(K_Y+\tri)$ is ample.
\end{defn}

\begin{prop}
Suppose that $(Y,\tri)$ is a log Fano pair of complex dimension $2$. Assume that $[\om_K^Y]$ is a K\"ahler class on $Y$ such that $[\pi^\ast\om_K^Y]+\eps\om_K$ satisfies the cohomology conditions in Proposition \ref{singular csck construction} for all $0<\eps\leq\tilde\eps$.  
Then there is a singular cscK metric in $[\om_K^Y]$.
\end{prop}
\begin{proof}
Let $(X,D)$ be the log resolution of $(Y,\tri)$.
Then $(X,\Om)$ satisfies the assumption in Proposition \ref{Big classes approximate csck}. So we have a diagonal sequence of the approximate cscK metrics $\om_{\vphi_\eps}=\om_\eps+i\p\bar\p\vphi_{\eps}$. 

The conclusion is a consequence of pushing forward the singular cscK metric on the covering space $X$ to $Y$.
Since the fibre $F$ of the log resolution $\pi$ is connected and the pull-back metric $\om_{sr}=0$ along $F$, we have $i\p\bar\p \pi^\ast \vphi\geq0$. We then see that $\vphi$ is a constant along $F$. As a result, $\vphi$ is also $\mathcal E^1$ and bounded on $Y$.

In order to apply Proposition \ref{singular csck construction} to construct a singular cscK metric in $[\om_{sr}]$ with $\om_{sr}=\pi^\ast\om_K^Y$, it is sufficient to show that there exists a constant $C$ such that
$E^\eps_\beta(\vphi_\eps)\leq C$ for all $\eps$ on $X$.

From \lemref{approximation proper} and $\nu^\eps_{\beta}(0)=0$, we have the upper bound of 
$
\nu_{\beta}(\vphi_\eps)$. From \eqref{log K energy big}, it is sufficient to show $J^\eta_{-\theta}(\vphi)=J_{-\theta}(\vphi)+\eta J_{\om_{\eps}}(\vphi)$ has uniform lower bound in $\mathcal H_\beta(\om_\eps)$ for all $\eps>0$.

The critical equation of $J^\eta_{-\theta}$ is
\begin{align*}
2\cdot\theta_\eps\wedge\om^{n-1}_\vphi=c_\eps \cdot \om^{n}_\vphi,
\quad c_\eps=2(\ul{S}_{\beta,\eps}-\eta).
\end{align*}
Here
$\underline S_{\b,\eps}=\frac{C_1(X,D)\Om_\eps^{n-1}}{\Om_\eps^n}$ and $\theta_\eps=\theta-\eta\om_\eps<0$. When $n=2$, as in \cite{MR1772078}, the critical equation could be rewritten as a Monge-Amper\`e equation
\begin{align}\label{J equation dim 2}
(c_\eps\om_\vphi-\theta_\eps)^2=\theta_\eps^2.
\end{align}

From $(iii)$ in \eqref{cohomology conditions},  we get
 $-c_0\om_{sr}+\theta-\eta\om_{sr}+i\p\bar\p \phi_0\geq0$. By assumption, 
$-c_\eps\om+\theta_\eps+i\p\bar\p \phi_0> 0$. Also, $\theta_\eps^2$ is $L^p(\om_K)$, $p>1$. Then we apply \lemref {Linfty estimate degenerate} to conclude that the solutions $\vphi^J_\eps$ to \eqref{J equation dim 2}
 have uniform bound for any $\eps$.

Therefore, making use of the definitions of 
$
J_{-\theta}(\vphi)
$ and $\om_{sr}\leq A \om_K$ for some positive constant $A>0$, we have that
$
|J_{-\theta}(\vphi^J_\eps)|\leq B
$ uniformly. Therefore, the convexity of $J_{-\theta}$ implies the uniform lower bound of $J_{-\theta}$ and we have completed the proof.

\end{proof}
\begin{rem}
Another way to define the reference metic for the log Fano pair $(Y,\tri)$ is to utilise the \textit{adapted measure}.
Let $\phi$ be a smooth Hermitian metric on the $Q$-line bundle $-(K_Y+\tri)$,
and $\sigma$ be a nowhere vanishing section of $r(K_Y+\tri)$.
The adapted measure of $\phi$ is defined to be
$
m_\phi:=\frac{(i^{rn^2}\sigma\wedge\bar\sigma)^{\frac{1}{r}}}{|\sigma|_{r\phi}^{\frac{2}{r}}}.
$ Let $\phi_i$ be a smooth Hermitian metric on the line bundle associated to $E_i$ and $s_i$ be a section vanishing along $E_i$. 
The lift of $\pi^\ast m_\phi$ writes as following
$
\pi^\ast m_\phi=\Pi_i |s_i|^{2a_i}_{\phi_i}\Pi_j |s_j|^{2(\beta_j-1)}_{\phi_j}dV:=e^fdV.
$
 By the klt condition, $\pi^\ast m_\phi$ is $L^p$ for some $p>1$.
On the regular part $Y_{reg}$, it holds
$
-i\p\bar\p\log m_{\phi_0}=\theta+[\tri].
$

Then the reference metric is defined to be a current $\omega_\theta:=\om^Y_K+i\p\bar\p\pi_\ast\vphi_\theta\in \Om_{full}(Y,\tri)$ satisfying the following equation
\begin{align}\label{reference metric log Fano}
\omega_\theta^n=\frac{m_{\phi_0}}{V^{-1} \int_Y m_{\phi_0} }.
\end{align}
After lifted to $X$, the equation becomes 
\begin{align}\label{lift reference metric log Fano}
(\omega_{sr}+i\p\bar\p\varphi_\theta)^n=e^{c}\pi^\ast m_{\phi_0}
\end{align}
for some constant $c$. Here, the lift metric $\omega_{sr}$ is semipositive and big, $a_i$ are integers and $0<\b_i\leq 1$. 
The equation \eqref{lift reference metric log Fano} has a bounded solution on $X$ and smooth outside $D$, according to Theorem B in \cite{MR2505296}. It is also a unique continuous solution $\pi_\ast\vphi_\theta\in  PSH(\om^Y_{K})$ to the equation \eqref{reference metric log Fano}. The solution $\pi_\ast\varphi_\theta$ is smooth on $Y_{reg}\setminus \supp \tri$ and satisfies that
\begin{align*}
Ric(\om_{\theta})=\theta+[\tri].
\end{align*}


One way to construct the approximate metric is to solve the following equation of $\om_{\theta_\eps}:=\om_\eps+i\p\bar\p\vphi_{\theta_\eps}\in\Om_\eps$,
\begin{align}\label{lift reference metric log Fano app both 1}
(\omega_{\eps}+i\p\bar\p\varphi_{\theta_\eps})^n=e^{c_\eps+f_\eps}dV.
\end{align}
Here $f_\eps$ is chosen to be a smooth function on $X$ and converges to $f$ with uniform estimates in an appropriate sense.
For example, we could set $f_\eps:=\log \frac{\Pi_i (|s_i|^{2}_{\phi_i}+\eps)^{a_i}}{\Pi_j( |s_j|^{2}_{\phi_j}+\eps)^{1-\beta_j}}$.
According to Yau's theorem \cite{MR480350}, there exists smooth solution $\varphi_{\theta_\eps}$ to \eqref{lift reference metric log Fano app both 1}. 
\end{rem}

\begin{rem}
The existence of K\"ahler-Einstein metrics on log Fano pair $(Y,\tri)$ is studied in \cite{MR3956691,MR3830547}, which is defined to be a $\theta$-psh function with full Monge-Amp\`ere mass such that
\begin{align*}
\theta_\varphi^n=\frac{e^{-\varphi}m_{\phi_0}}{V^{-1} \int_Y e^{-\varphi} m_{\phi_0} }.
\end{align*}
It could be rewritten regarding of Ricci curvature on the regular part $Y_{reg}$ as
$$
Ric(\theta_\varphi)=\theta_\varphi+[\tri].
$$
\end{rem}

\begin{rem}About the assumption on the automorphism group in \thmref{singular csck construction}, we remark that in the K\"ahler-Einstein case, Theorem 5.4 in \cite{MR3956691} shows that the $J$-properness of the Mabuchi energy implies the triviality of the automorphism group $Aut(X,D)$.
\end{rem}

\begin{rem}[log canonical pair]
Another important case is log canonical pair $(Y,\tri)$, which satisfies $a_i\geq-1$ for all $i$. When $K_Y+\tri$ is ample, the existence of K\"ahler-Einstein metric is proven in \cite{MR3283927}. Similar to the log Fano case, a smooth Hermitian metric is fixed on the $Q$-line bundle $K_Y+\tri$ and the K\"ahler-Einstein equation reads $$
Ric(\theta_\varphi)=-\theta_\varphi+[\tri]
$$ on the regular part $Y_{reg}$. Similar strategies could be adopted to define singular cscK metrics on the log canonical pair $(Y,\tri)$. 
\end{rem}

\section{A priori estimates}\label{a priori estimates}

In this section, we derive the a priori estimates of the cscK metrics with both cone singularities and degeneration. 
\subsection{CscK equations}
Recall that $\theta$ is a smooth closed $(1,1)$-form in the cohomology class $C_1(X,D)$.

In the K\"ahler case, $\om_0$ is a smooth K\"ahler metric. The K\"ahler cone metric $\om_\theta$ is defined in \eqref{theta} and $\om_{\theta_\eps}$ is the smooth approximation metric of $\om_\theta$ in \eqref{thetaeps} with Ricci curvature bounded below $Ric(\om_{\theta_\eps})\geq \theta$. Note that the K\"ahler cone potential $\vphi_{\theta_\eps}$ is bounded, since $\om^n_{\theta_\eps}\in L^p(\om_0)$ for some $p>1$ from \eqref{thetaeps}.

In the degenerate case, $\om_K$ is a K\"ahler metric and $\om_\eps=\om_{sr}+\eps\om_K$. The approximate metrics $\om_{\theta_\eps}$ is defined to be the solution to \eqref{lift reference metric log Fano app both 3} with $t=\eps$. They are smooth K\"ahler metrics.
From \eqref{big nef lower Ricci}, the Ricci curvature of $\om_{\theta_\eps}$ is also bounded below uniformly. According to Proposition \ref{big nef approximate reference metric bound}, $\vphi_{\theta_\eps}$ is also uniformly bounded.

Let  $\chi_0$ be a smooth closed $(1,1)$ form. 
We are given a smooth function $R$ and a closed $(1,1)$-form 
\begin{align}\label{chi}
\chi:=\chi_0+i\p\bar\p f\geq 0\text{ with }e^{-f}\in L^{p_0}(\om_0)\text{ for some large }p_0>>1.
\end{align} 
We normalise $f$ such that
$
\sup_X f=0.
$
We denote
$
\Theta=\theta-\chi_0.
$

The equations we consider are,
\begin{align}
F_\eps&=\log\frac{\om_{\vphi_\eps}^n}{\om_{\theta_\eps}^n},\label{F}\\
\tri_{\vphi_\eps} F_\eps&=\tr_{\vphi_\eps}(\Theta-i\p\bar\p f)-R.\label{triF}
\end{align} 
Equivalently, the scalar curvature of $\om_{\vphi_\eps}$ satisfies
\begin{align*}
S(\om_{\varphi_\eps})=\tr_{\vphi_\eps}(Ric(\om_{\theta_\eps})-\Theta+i\p\bar\p f)+R.
\end{align*}

In this section, we define the K\"ahler potential $\vphi_\eps$ in terms of the background metric $\om_{\theta_\eps}$,
\begin{align}
\om_{\vphi_\eps}=\om_{\theta_\eps}+i\p\bar\p\vphi_\eps.
\end{align}
We remark that this is different from the previous sections, where the K\"ahler potential is defined respect to the smooth metric $\om_0$ or $\om_\eps$. 

\subsubsection{Convention}
The quantitative conditions of $\chi$ and $R$ will be specified in this section. When we write $\inf_X \chi$, it means $\chi\geq \inf_X\chi \cdot \om_{\theta_\eps}$. If $\chi>0$, we set $\inf_X\chi=0$. It is similar to define $\sup_X \chi$ and $\|\chi\|_\infty$, with respect to $\om_{\theta_\eps}$.

From now on, we use $\om$ to denote $\om_{\theta_\eps}$ and write $\vphi$ instead of $\vphi_\eps$ for short. We denote $E_\beta=\frac{1}{V}\int_X\log\frac{\om^n_{\vphi_\eps}}{\om^n_{\theta_\eps}}\om_{\vphi_\eps}^{n}=\frac{1}{V}\int_X e^FF\om^n$.

\subsection{Main estimates}

\begin{thm}[Nondegenrate]\label{a prioriestimates approximation}
Suppose that $\vphi$ is a solution to \eqref{F} and \eqref{triF} with the twisted term satisfying \eqref{chi}. Then there is a constant $C$ such that 
\begin{align*}
\|\vphi\|_{\infty}, \quad \|F+f\|_{\infty}, \quad \sup_X \|\p (F+f)\|^2_\vphi, \quad \sup_X\|\tr_{\om_{\theta}}\om_\vphi\|_{p;\om_{\theta}}\leq C,
\end{align*} where $C$ depends on $\alpha_1,\alpha_\beta, n$, $\|\frac{\om_{\theta}^n}{\om_0^n}\|_{L^q(\om_0)}$ for some $q>1$ and the following quantities
	\begin{align}\label{a prioriestimates approximation dependence}
	E_\beta, \quad\|e^{-f}\|_{L^{p_0}(\om_0)},\quad \|R\|_\infty, \quad \|\Theta\|_\infty,\quad\inf_XRic(\om_{\theta}).
	\end{align}
	In which, $p_0$ is sufficiently large and depends on $n$ and $p$. 
	
	Furthermore, when $f=0$, there is a constant $C$ such that 
	\begin{align*}
	\|\vphi\|_\infty, \quad \|F\|_\infty, \quad \sup_X \|\p F\|^2_{\om_{\theta}}, \quad\sup_X \tr_{\om_{\theta}}\om_\vphi\leq C,
	\end{align*} where $C$ depends on the following quantities
	\begin{align}\label{a prioriestimates approximation dependence f=0}
	E_\beta, \quad \|R\|_\infty, \quad \|\Theta\|_\infty,\quad\inf_XRic(\om_{\theta}),\quad \alpha_1,\quad \alpha_\beta,\quad n.
	\end{align}
\end{thm}

\begin{proof}
	We apply the H\"older inequality to control $\|e^{-f}\|_{L^{p_0}(\om_{\theta_\eps})}$ by $\|e^{-f}\|_{L^p(\om_{0})}$ with $p> p_0$ and $\|\frac{\om_{\theta}^n}{\om_0^n}\|_{L^q(\om_0)}$ with some $q>1$.
	Then the first conclusion is a combination of the $L^\infty$-estimates (\thmref{Linftyestimates}), the $W^{2,p}$-estimates (\thmref{w2pestimates}) and the gradient estimate of $F+f$ (\thmref{Gradient estimates}). The second conclusion follows from \thmref{Linftyestimates}, \thmref{w2pestimates} and the $C^{1,1}$-estimate (\thmref{C11estimates}). 
\end{proof}

The estimates in \cite{arXiv:1712.06697,arXiv:1801.00656,arXiv:1801.05907} were obtained in terms of fixed smooth reference metric $\om_0$. In the non-degenerate setting, the reference metric is $\om_{\theta_\eps}$, which converges to the K\"ahler cone metric $\om_\theta$, as $\eps\rightarrow 0$. But, in general the K\"ahler cone metric does not have bounded geometry. Therefore, we need to apply the delicate analysis of the approximate metric $\om_\theta$ proved in \cite{MR4020314} and choose appropriate weighted functions to obtain the estimates. In the degenerate case, we will obtain estimates separately.

\subsection{$L^\infty$-estimates: nondegenrate}

\subsubsection{$L^\infty$-estimates of K\"ahler cone potential and volume ratio}
The $L^\infty$-estimates include three parts, the lower bound of $F+f$, the upper bound of $F+f$, and the $L^\infty$-bound of $\vphi$.

\begin{thm}\label{Linftyestimates}
	Assume that $\vphi$ is a solution of \eqref{F} and \eqref{triF} Then there exists a constant $C$ such that
	\begin{align*}
	\|F+f\|_\infty, \quad \|\vphi\|_\infty\leq C.
	\end{align*}
	The constant $C$ depends on 	\begin{align*}
	E_\beta, \quad\|e^{-f}\|_{p_0},p_0\geq2,\quad \|R\|_\infty, \quad \|\Theta\|_\infty,\quad \alpha_1,\quad \alpha_\beta,\quad n.
	\end{align*}
\end{thm}
\begin{proof}
	The lower bound of $F$ and $F+f$ are obtained from Proposition \ref{Flower}.
Then the theorem follows from the $L^\infty$-estimates of $\vphi$ and the $F+f$ upper bound (Proposition \ref{upper F}).
\end{proof}

\subsubsection{Basic inequalities}
We will use the following fundamental inequalities many times in this section,
\begin{align}
&n e^{-\frac{F}{n}}\leq \tr_\vphi\om\leq \frac{1}{n-1}e^{-F}(\tr_\om\om_\vphi)^{n-1},\label{trivphiom}\\
&n e^{\frac{F}{n}}\leq \tr_\om\om_\vphi\leq \frac{1}{n-1}e^{F}(\tr_\vphi\om)^{n-1}.\label{triomvphi}
\end{align}
We normalised $\vphi$ such that
$\sup_X \vphi=0.$
The definition of $\om_\vphi$ gives us two useful formulas
\begin{align}\label{vphivphi}
\tri_{\vphi}\vphi=n-\tr_\vphi\om,\quad
\tri\vphi=\tr_\om\om_\vphi-n.
\end{align}

We compute that
\begin{lem}\label{differential identity}
Let $A_1,A_2$ be two constants. We have
\begin{align*}
\tri_{\vphi}(F+A_2 f+A_1\vphi)
=\tr_{\vphi}\Theta-R+(A_2-1)\tri_\vphi f+A_1(n-\tr_\vphi\om).
\end{align*}
\end{lem}

\subsubsection{Auxiliary function}
We let $\Phi(t)=\sqrt{t^2+1}$ and \begin{align*}
E_\Phi=\int_Xe^F\Phi(F) \om^n.
\end{align*}
According to the Calabi-Yau theorem, there exists a solution $\psi$ solving the following equation
\begin{align}\label{psi}
\frac{(\om+i\p\bar\p\psi)^n}{\om^n}=\frac{e^F\Phi(F)}{E_\Phi}.
\end{align}
We also normalize $\psi$ such that $\sup_X \psi=0.$
We use \eqref{trivphiom}, and also \eqref{psi} to get
\begin{align}\label{tripsi}
\tri_{\vphi}\psi&=\tr_\vphi\om_\psi-\tr_\vphi\om\geq n(\frac{\om^n_\psi}{\om_\vphi^n})^{\frac{1}{n}}-\tr_\vphi\om\\
&=n(\Phi(F)E_\Phi^{-1})^{\frac{1}{n}}-\tr_\vphi\om.\nonumber
\end{align}
We call $\psi$ the auxiliary function.

\begin{lem}The integral $E_\Phi$ is controlled by the entropy $E_\beta=\frac{1}{V}\int_X e^FF\om^n$ as
	\begin{align*}
E_\Phi\leq E_\beta+2e^{-1}+1.
\end{align*}
\end{lem}
\begin{proof}
We prove by direct computation
	\begin{align*}
	E_\Phi&\leq \frac{1}{V}\int_X e^F(|F|+1)\om^n= \frac{1}{V}\int_X e^F|F|\om^n+1\\
	&=  \frac{1}{V}\int_X e^FF\om^n- \frac{2}{V}\int_{F<0} e^FF\om^n+1.
	\end{align*}
	The lemma follows from $e^FF\geq -e^{-1}$.
\end{proof}
\subsubsection{Cutoff function}Given a point $z\in X$.
The cutoff function $\eta$ is chosen to be a function in $B_{d}(z)$ such that $\eta=1$ in the half ball $B_{d/2}(z)$ with radii $\frac{d}{2}$ centred at the maximum point $z$ of $u$, and equals to $1-\eps$ outside $B_d(z)$. Then we have
\begin{align}\label{cutofffunction}
\tri_\vphi\log\eta
&=-\frac{1}{\eta^2}|\p\eta|^2_\vphi+\frac{1}{\eta}\tri_\vphi\eta\\
&\geq -\frac{4\eps^2}{d^2(1-\eps)^2}\tr_\vphi\om-\frac{4\eps}{d^2(1-\eps)}\tr_\vphi\om.\nonumber
\end{align}

\subsubsection{$L^\infty$-estimates of $\vphi$ and $F+f$ upper bound}
\begin{prop}\label{upper F}
	Assume $\vphi$ is a solution of \eqref{triF}. Then there exists $p>1$ such that
	\begin{align*}
	\sup_X(F+f), \quad \|e^{F}\|_{p},\quad \|\vphi\|_\infty\leq C.
	\end{align*}
	The constant $C$ depends on 	\begin{align*}
	E_\beta, \quad\|e^{-f}\|_{p_0;\om_0},p_0> 1,\quad \sup_X R, \quad \inf_X\Theta,\quad \alpha_1,\quad \alpha_\beta,\quad n.
	\end{align*}
\end{prop}
\begin{proof}
Let $A>0$, $1>B>0$ be two positive constants to be determined.
	
\textbf{Step 1: differential inequality}
We set $u=F+f+A \vphi+B \psi$. We apply \lemref{differential identity} with $A_1=A$ and $A_2=B$, and \eqref{tripsi} to have
	\begin{align*}
\tri_{\vphi}u
&=\tr_{\vphi}\Theta-R+A(n-\tr_\vphi\om)+ B \tri_{\vphi}\psi\\
&\geq \tr_{\vphi}\Theta-R+A(n-\tr_\vphi\om)+ B [n(\Phi(F)E_\Phi^{-1})^{\frac{1}{n}}-\tr_\vphi\om].
	\end{align*}
We set $A_\Theta:=\inf_X \Theta-A-B >1$ with the choice of $A=-2(1+|\inf_X \Theta|)$, and also set
\begin{align}\label{AR}
A_R:=-\sup_X R+An+B n(\Phi(F)E_\Phi^{-1})^{\frac{1}{n}}.
\end{align}
We denote by $\delta$ a constant to be determined, which will be related to the log alpha invariant.
In conclusion, we have
\begin{align}\label{DI infinity}
\tri_{\vphi}(\delta u)
\geq \delta A_\Theta \tr_{\vphi}\om+\delta A_R.
\end{align}

\textbf{Step 2: localisation}
In  order to annihilate the the first term on the right hand side of \eqref{DI infinity}, it sufficient to add $\delta u$ with a cutoff function \eqref{cutofffunction}.
Choosing $\eps$ small enough such that $\delta A_\Theta - \frac{4\eps^2}{d^2(1-\eps)^2}-\frac{4\eps}{d^2(1-\eps)}>0$, we thus have
\begin{align*}
\tri_{\vphi}(\delta u+\log\eta)
\geq \delta A_R.
\end{align*}

Then we apply the Alexandroff maximum principle to the equation of $e^{\delta u}\eta$, which is obtained from above,
\begin{align*}
\tri_{\vphi}(e^{\delta u}\eta)
\geq \delta e^{\delta u}\eta A_R.
\end{align*}
We get
\begin{align*}
\sup_{B_d(z)}(e^{\delta u}\eta)
\leq \sup_{\p B_d(z)}(e^{\delta u}\eta)+\delta\int_{B_d(z)} e^{2n\delta u+2F}(A_R)_{-}^{2n}\om^n.
\end{align*}
The function $(A_R)_{-}(x)$ is equal $A_R(x)$ when $A_R(x)<0$, and vanishing when $A_R(x)>0$. So in the region $B_-=\{x\in B_d(z)\vert A_R(x)\leq0\}$,  we have by \eqref{AR},
\begin{align*}
B n(\Phi(F)E_\Phi^{-1})^{\frac{1}{n}}\leq \sup_X R-An.
\end{align*} From this inequity together with $F\leq\Phi(F)=\sqrt{F^2+1}$, we have $F$ is bounded above by $C_F$ depending on 
\begin{align*}
E_\Phi,\quad\sup_X R, \quad A,\quad B,\quad  n.
\end{align*} 

\textbf{Step 3: using cone reference metric}
We then prove the bound of \begin{align*}
I:=\delta\int_{B_d(z)} e^{2n\delta u+2F}(A_R)_{-}^{2n}\om^n=\delta\int_{B_{-}} e^{2n\delta (F+f+A \vphi+B \psi)+2F} (A_R)_{-}^{2n} \om^n.
\end{align*}
From \eqref{AR}, $-\sup_X R +An\leq A_R\leq 0$ on $B_{-}$.
By $\psi\leq 0$ and $f\leq 0$, there exists a constant depending on $C_F,\sup_X R, A, n$ such that
\begin{align}\label{I degenerate}
I&\leq C\delta  (\sup_X R-An)^{2n} \int_{B_{-}} e^{2n\delta A \vphi} \om^n.
\end{align}
In the case $\eps=0$, we remark that
\begin{align*}
\om^n=\om^n_{\theta_\eps}=\frac{e^{h_0+c}\om^n_0}{(|s|^2_h+\eps)^{1-\beta}}\leq \frac{e^{h_0}\om^n_0}{|s|_h^{2(1-\beta)}}e^c=\om^n_\theta e^c.
\end{align*}
From \eqref{ec}, the normalisation constant $e^c$ is bounded above. We have $I$ is bounded, when we choose appropriate $\delta$ according to the log alpha invariant $\alpha_\beta$. 

Thus the upper bound of $u=F+f+A \vphi+B \psi$ is obtained from $\eta=1-\eps$ on $\p B_d(z)$.
By using $\vphi\leq 0$ and $A< 0$, we get $A\vphi\geq 0$.
So, we conclude that 
	\begin{align}\label{infF}
F+f+B \psi\leq C.
\end{align} 

\textbf{Step 4: $L^\infty$-estimates}
	Rewriting \eqref{infF} and using the alpha invariant $\alpha_1$ \eqref{alpha1}, we get
	\begin{align*}
	\int_Xe^{\frac{\alpha_1}{B}(F+f)}\om_0^n\leq e^{\frac{\alpha_1 C}{B}}\int_Xe^{-\alpha_1\psi}\om_0^n\leq C,
	\end{align*}
	by $\psi\leq 0$ and the definition of the alpha invariant $\alpha_1$.
	
		We choose $p$ such that $p_0>p>1$ and also $p_1$ satisfying $p_1=\frac{p_0p}{p_0-p}>p$. 
We further choose $B$ small enough such that $\frac{\alpha_1}{B}=p_1$. Then we have 
	\begin{align*}
	e^{F+f}\in L^{p_1}(\om_0) .
	\end{align*} 
	The H\"older's inequality implies that
	\begin{align*}
	\|e^F\|_p=\|e^{F+f}e^{-f}\|_p\leq \|e^{F+f}\|_{p_1} \|e^{-f}\|_{p_0}.
	\end{align*}
	Then by Ko{\l}odziej's estimate, we get the $L^\infty$-bound of $\vphi$ from \eqref{F}. 
	
	Furthermore, the $L^p$ bound of $e^F$ implies that $e^F\Phi(F)$ is also $L^q(\om_0)$ for $q\geq 2$. So we get bound of $\|\psi\|_\infty$ by applying Ko{\l}odziej's estimate again to \eqref{psi}.

	At last, \eqref{infF} implies that
	$\sup_X (F+f)$ is bounded.
\end{proof}

\subsection{$L^\infty$-estimates for degenerate metrics}\label{Linfty-estimates for degenerate metrics}
Recall that $\om_K$ is a K\"ahler and $\om_{sr}$ is big and semipositive. Recall the notations
\begin{align*}
&\om_{K}=\om_{sr}+a_0 i\p\bar\p\log h_E,\quad \phi_E=a_0\log h_E,\\
&\om_\eps=\om_{sr}+\eps\om_K,\quad \tilde \om_\eps=(1+\eps )\om_K=\om_\eps+i\p\bar\p \phi_E,\\
&\Phi_{\theta_{\eps}}=\vphi_{\theta_{\eps}}-\phi_E,\\
&\om_{\theta_{\eps}}=\om_{\eps}+i\p\bar\p\vphi_{\theta_{\eps}}=\tilde\om_\eps+i\p\bar\p \Phi_{\theta_{\eps}},\\
&\om_{\vphi_\eps}=\om_{\theta_\eps}+i\p\bar\p\vphi_{\eps}=\tilde\om_\eps+i\p\bar\p \Phi_{\theta_{\eps}}+i\p\bar\p\vphi_{\eps},\\
&\om_{\mathbf b_\eps}=\tilde\om_\eps+i\p\bar\p \Phi_{\mathbf b_\eps}.
\end{align*} 
Note that we have the following estimates from Proposition \ref{big nef approximate reference metric bound} and \eqref{backgroundmetriccurvature},
\begin{align*}
\|\varphi_{\theta_{\eps}}\|_\infty\leq C, \quad 
C^{-1} |S_E|^{b}_{h_E}\cdot\om_{\bf b_\eps}\leq \om_{\theta_{\eps}}\leq C |S_E|^{-a}_{h_E}\cdot\om_{\bf b_\eps}.
\end{align*}
and 
\begin{align*}
\left\{
\begin{array}{ll}
&R_{i\bar j k\bar l}(\om_{\bf b_\eps} )\geq -(\tilde g_{\mathbf b_\eps})_{i\bar j} \cdot (g_{\bf b_\eps} )_{k\bar l}\\
&\tilde \om_{\mathbf b_\eps}=C_2\cdot \om_{\bf b_\eps} +i\p\bar\p \phi^\eps_{\mathbf b}\geq 0,\\
&\|\phi^\eps_{\mathbf b}\|_\infty\leq C_3.
\end{array}
\right.
\end{align*}

We now derive immediately the $L^\infty$-estimates of $\vphi_\eps$ for the solution $\om_{\vphi_\eps}$ to the equation \eqref{approximate singular cscK} by directly applying Proposition \ref{upper F}.

The new difficulty is that $\om_{sr}$ could be degenerate somewhere. We need to measure $\Theta$ in terms of a K\"ahler metric. We fix a reference metric $(1+\eps)\om_{K}$ to measure $\Theta$.
\begin{lem}\label{control Theta}We set 
\begin{align*}
\phi_l:= \inf_{(X,(1+\eps)\om_{K})}\Theta \cdot (\vphi_{\theta_\eps}-\phi_E),\\
\phi_u:= \sup_{(X,(1+\eps)\om_{K})}\Theta \cdot (\vphi_{\theta_\eps}-\phi_E).
\end{align*} Then we have
\begin{align*}
\Theta+i\p\bar\p\phi_l&\geq \inf_{(X,(1+\eps)\om_{K})}\Theta\cdot \om_{\theta_\eps},\\
\Theta+i\p\bar\p\phi_u&\leq \sup_{(X,(1+\eps)\om_{K})}\Theta\cdot \om_{\theta_\eps}.
\end{align*}
\end{lem}
\begin{proof}By \eqref{lift reference metric log Fano app both 3}, we have
\begin{align*}
\Theta&\geq \inf_{(X,(1+\eps)\om_{K})}\Theta \cdot (1+\eps)\om_{K}\\
&= \inf_{(X,(1+\eps)\om_{K})}\Theta \cdot (\om_{\theta_\eps}-i\p\bar\p\vphi_{\theta_\eps}+a_0i\p\bar\p \log h_E).
\end{align*}
The proof of the other direction of the inequality is identical.

\end{proof}

\begin{rem}
We could use the metric $\om_{\mathbf b_\eps}$ instead of $\om_{\theta_\eps}$ in the lemma.
\end{rem}

We know $\phi_E$ is bounded above. So, 
\begin{lem}$\phi_l\leq C$ and $\phi_u\geq C$.
\end{lem}

\begin{lem}\label{weights power} Let $A,B$ be two non-negative constants. Then
$e^{A \phi_l}$ and $e^{-B \phi_u}$ is bounded.
\end{lem}
\begin{proof}
We compute and use the boundedness of $\vphi_{\theta_\eps}$ (Proposition \ref{big nef approximate reference metric bound}) that
\begin{align*}
e^{A \phi_l}&=e^{A \inf_{(X,(1+\eps)\om_{K})}\Theta \cdot \vphi_{\theta_\eps}-A a_0 \inf_{(X,(1+\eps)\om_{K})}\Theta \cdot \log h_E)}\\
&\leq C e^{-A a_0 \inf_{(X,(1+\eps)\om_{K})}\Theta \cdot \log h_E}\\
&= C |s|^{-a}_{h_E}.
\end{align*}
In which, $a=A a_0 \inf_{(X,(1+\eps)\om_{K})}\Theta$ is nonpositive. 
So, $|s|^{-a}_{h_E}$ is bounded.  
\end{proof}

\begin{lem}\label{phil negative}Assume positive constant $A$ is small enough. Then
$e^{-A\phi_l}$ is $L^p$ for some $p>1$.
\end{lem}
\begin{proof}
By boundedness of $\vphi_{\theta_\eps}$ in Proposition \ref{big nef approximate reference metric bound}, we get
\begin{align*}
e^{-A\phi_l}=e^{-A\inf_{(X,(1+\eps)\om_{K})}\Theta \cdot (\vphi_{\theta_\eps}-a_0 \log h_E)}\leq C |s|^{2b}_{h_E}
 \end{align*} 
with $b=a_0 A \inf_{(X,(1+\eps)\om_{K})}\Theta\leq 0$. Once we choose $A$ sufficiently small, we have $e^{-A\phi_l}$ is $L^p,p>1$.
\end{proof}

\begin{prop}\label{upper F degenerate weak}
Assume $\vphi$ is a solution of \eqref{triF}. Then there exists $p>1$ such that
	\begin{align*}
	 \|e^{F+\phi_l}\|_{p}\leq C.
	\end{align*}
	The constant $C$ depends on 	\begin{align*}
	E_\beta, \quad\|e^{-f}\|_{p_0;\om_0},p_0> 1,\quad \sup_X R 	\end{align*}
	and $\inf_{(X,(1+\eps)\om_{K})}\Theta ,\quad \alpha_1,\quad \alpha_\beta,\quad n,\quad \|\vphi_{\theta}\|_\infty.
$
\end{prop}

\begin{proof}
 We will omit the lower index $\eps$ in the proof.
We now consider $\tilde u=u+\phi_l$ and have from \textbf{Step 1} that
\begin{align}\label{DI infinity degenerate}
\tri_{\vphi}(\delta \tilde u)
\geq \delta A_\Theta \tr_{\vphi}\om+\delta A_R.
\end{align}
Then the Alexandroff maximum principle in \textbf{Step 2} implies that
\begin{align*}
\sup_{B_d(z)}(e^{\delta \tilde u}\eta)
\leq \sup_{\p B_d(z)}(e^{\delta\tilde u}\eta)+I.
\end{align*}
From \eqref{I degenerate} in \textbf{Step 3}, we get
\begin{align*}
I&:=\delta\int_{B_d(z)} e^{2n\delta \tilde u+2F}(A_R)_{-}^{2n}\om^n\\
&=\delta\int_{B_{-}} e^{2n\delta (F+f+A \vphi+B \psi+\phi_l)+2F} (A_R)_{-}^{2n} \om^n.
\end{align*}

Due to \lemref{weights power}, 
$e^{2n\delta \phi_l}$ is bounded.
Since $F,f,\psi$ are all bounded above, we further have
\begin{align}
I&\leq C\delta  (\sup_X R-An)^{2n} \int_{B_{-}} e^{2n\delta A \vphi} |s|^{-a}_{h_E}\om^n.
\end{align}
Note that $\om=\om_{\theta_\eps}=\om_{\eps}+i\p\bar\p\vphi_{\theta_\eps}$.
We have $I$ is bounded, because of the continuity of the alpha invariant $\alpha([\om_\eps])$(Proposition \ref{alphacontinuity}). So we have \begin{align}\label{infF d}
F+f+B \psi+\phi_l \leq C.
\end{align} 

As in \textbf{Step 4}, we then have
	\begin{align*}
	\int_Xe^{\frac{\alpha_1}{B}(F+f+\phi_l)}\om_0^n\leq e^{\frac{\alpha_1 C}{B}}\int_Xe^{-\alpha_1\psi}\om_0^n\leq C,
	\end{align*}
So $\|e^{F+f+\phi_l}\|_p$ is bounded for some $p>1$. By $e^{-f}\in L^{p_0}$, we get $\|e^{F+\phi_l}\|\in L^p, p>1$.

\end{proof}

\begin{prop}\label{upper F degenerate}
Assume $\vphi$ is a solution of \eqref{triF} and $e^{-\phi_l}\in L^{p_0}$ for some $p_0>1$. Then there exists $p>1$ such that
	\begin{align*}
	\sup_X(F+f+\phi_l), \quad \|e^{F}\|_{p},\quad \|\vphi\|_\infty\leq C.
	\end{align*}
	The constant $C$ depends on 	\begin{align*}
	E_\beta, \quad\|e^{-f}\|_{p_0;\om_0},p_0> 1,\quad \sup_X R,\quad \|e^{-\phi_l}\|_{p_0} 	\end{align*}
	and $\inf_{(X,(1+\eps)\om_{K})}\Theta ,\quad \alpha_1,\quad \alpha_\beta,\quad n,\quad \|\vphi_{\theta}\|_\infty.
$
\end{prop}

\begin{proof}
We further have 
\begin{align*}
e^F=e^{F+\phi_l} e^{-\phi_l}\leq Ce^{F+\phi_l} .
\end{align*}
then $e^{F}$ is also $L^p,p>1$, due to Proposition \ref{upper F degenerate weak}.

 Recall that \begin{align*}
\om^n_{\vphi_\eps}=e^F\om^n_{\theta_\eps}.
 \end{align*} 
Here $\om_{\vphi_\eps}=\om_\eps+i\p\bar\p(\vphi_{\theta_\eps}+\vphi_\eps)$. From Proposition \ref{big nef approximate reference metric bound}, $\vphi_{\theta_\eps}$ is bounded and $\om^n_{\theta_\eps}$ is $L^p, p>1$. 
So the $L^\infty$-estimate $\|\vphi_\eps\|_\infty$ follows from \lemref{Linfty estimate degenerate}.

Then we have $e^F\Phi(F)$ is also $L^q(\om_0)$ for $q\geq 2$ and get bound of $\|\psi\|_\infty$ by the same reason. By \eqref{infF d}, the upper bound of $F+f$ is obtained in terms of $-\phi_l$.
\end{proof}

\begin{rem}
When $\Theta\geq 0$, we have $e^{-\phi_l}$ is bounded.
\end{rem}
\subsection{$F$ lower bound: nondegenrate}
\begin{prop}\label{Flower}There holds
	\begin{align*}
	\inf_{X}F\geq \inf_{X}( F+f)\geq C.
	\end{align*}
	The constant $C$ depends on 
	\begin{align*}
	\|\vphi\|_\infty,  \quad\|e^{-f}\|_{L^{p_0}(\om)}, p_0\geq 2, \quad \inf_X R, \quad \sup_X \Theta .
	\end{align*}
\end{prop}
\begin{proof}
	We set $u=F+f+A\vphi$. We have by \lemref{differential identity} with $A_1=A$ and $A_2=1$, 
	\begin{align}\label{triFvphi}
	\tri_{\vphi} u
	&=\tr_{\vphi}\Theta-R+A(n-\tr_\vphi\om)\leq A_{\Theta}\tr_\vphi\om+A_R.
	\end{align}
	We choose $A=\sup_X \Theta + \frac{|\inf_X R |}{n}+1$ such that 
	\begin{align*}A_R:=An-\inf_X R>0,\quad A_{\Theta}:=\sup_X \Theta-A<0.
	\end{align*} Let $\delta$ to be a positive constant to be determined. After multiplied with $-\delta$, \eqref{triFvphi} becomes
	\begin{align*}
	\tri_{\vphi}(-\delta u)\geq -\delta A_{\Theta}  \tr_\vphi\om-\delta A_R.
	\end{align*} 
	By adding the cutoff function \eqref{cutofffunction}, which is defined in $B_d(z)$ near the maximum point $z$ of $-u$, and choosing $\eps$ small enough such that $-\delta A_{\Theta} - \frac{4\eps^2}{d^2(1-\eps)^2}-\frac{4\eps}{d^2(1-\eps)}>0$, we have
\begin{align*}
\tri_{\vphi}(-\delta u+\log\eta)
\geq -\delta A_R.
\end{align*}
Moreover,
\begin{align*}
\tri_{\vphi}(e^{-\delta u}\eta)
\geq -\delta  A_R e^{-\delta u}\eta.
\end{align*}

Then we apply the Alexandroff maximum principle to this differential inequality, 
\begin{align*}
\sup_{B_d(z)}(e^{-\delta u}\eta)
\leq \sup_{\p B_d(z)}(e^{-\delta u}\eta)+\delta\int_{B_d(z)} e^{-2n\delta u+2F}A_R^{2n}\om^n.
\end{align*}
Choose $0<\delta\leq n^{-1}$ and check the boundedness of the last term. Then there exists a constant depending on $\|\vphi\|_\infty,\sup_X(F+f), A_R, \delta$ such that 
\begin{align*}
I&:=\delta\int_{B_d(z)} e^{-2n\delta u+2F}A_R^{2n}\om^n
=\delta\int_{B_d(z)} e^{-2n\delta (F+f+A \vphi)+2F} A_R^{2n} \om^n\\
&\leq C\delta\int_{B_d(z)} e^{(2-2n\delta)F- 2n\delta f}\om^n\leq C\delta\int_{B_d(z)} e^{-2f}\om^n.
\end{align*}
We then conclude that $I$ is bounded, since $e^{-f}\in L^{p_0}(\om)$ for $p_0\geq 2$.
We complete the proof by using $\eta=1-\eps$ on $\p B_d(z)$.
\end{proof}

\subsection{$F$ lower bound for degenerate metrics}\label{F lower bound for degenerate metrics}
\begin{prop}\label{Flower}There holds
	\begin{align*}
	\inf_{X}(F+\phi_u)\geq \inf_{X}( F+f+\phi_u)\geq C.
	\end{align*}
	The constant $C$ depends on 
	\begin{align*}
	\|\vphi\|_\infty,  \quad\|e^{-f}\|_{L^{p_0}(\om)}, p_0\geq 2, \quad \inf_X R, \quad \sup_{(X,(1+\eps)\om_{K})}\Theta,\quad \|\vphi_{\theta_\eps}\|_\infty.
	\end{align*}
\end{prop}
\begin{proof}
	Set $u:=F+f+A\vphi+\phi_u$. By \lemref{differential identity}, 
	\begin{align}\label{triFvphi D}
	\tri_{\vphi} u
	&=\tr_{\vphi}\Theta-R+A(n-\tr_\vphi\om)
+\tri_\vphi\phi_u	.
	\end{align}
	From \lemref{control Theta}, 
		\begin{align}\label{triFvphi D 1}
	\tri_{\vphi} u
	\leq \sup_{(X,(1+\eps)\om_{K})}\Theta\cdot \om-R+A(n-\tr_\vphi\om).
	\end{align}
	Now we could choose $A=\sup_{(X,(1+\eps)\om_{K})}\Theta + \frac{|\inf_X R |}{n}+1$ such that 
	\begin{align*}A_R:=An-\inf_X R>0,\quad A_{\Theta}:=\sup_{(X,(1+\eps)\om_{K})}\Theta-A<0.
	\end{align*} 
	By adding the cutoff function \eqref{cutofffunction}, and applying the Alexandroff maximum principle, we have 
\begin{align*}
\sup_{B_d(z)}(e^{-\delta \tilde u}\eta)
\leq \sup_{\p B_d(z)}(e^{-\delta\tilde  u}\eta)+I.
\end{align*}
Here,
\begin{align*}
I&:=\delta\int_{B_d(z)} e^{-2n\delta u+2F}A_R^{2n}\om^n\\
&=\delta\int_{B_d(z)} e^{-2n\delta (F+f+\phi_l+A \vphi+\phi_u-\phi_l)+2F} A_R^{2n} \om^n\\
&=\delta\int_{B_d(z)} e^{(2-2n\delta) (F+f+\phi_l)+(-2n\delta)(A \vphi+\phi_u-\phi_l)-2(f+\phi_l)} A_R^{2n} \om^n\\
&=\delta\int_{B_d(z)} e^{(2-2n\delta) (F+f+\phi_l)-2n\delta(A \vphi+\phi_u)-2f+(2n\delta-2)\phi_l} A_R^{2n} \om^n.
\end{align*}

Choosing $\delta$ such that $2-2n\delta$ is positive and sufficiently small, we have there exists a constant $C$ depending on $\|\vphi\|_\infty$, $\sup_X(F+f+\phi_l)$, $A_R$, $A$, $\delta$ such that 
\begin{align*}
I \leq C\delta\int_{B_d(z)} e^{-2n\delta\phi_u-2f+(2n\delta-2)\phi_l} \om^n.
\end{align*}
By \lemref{weights power}, $e^{-2n\delta\phi_u} $ is bounded. Note that $e^{-2f}$ is $L^p, p>1$. Also, by \lemref{phil negative}, $e^{(2n\delta-2)\phi_l}$ is $L^p,p>1$.
Thus $I$ is bounded and the proof is completed by using the property of the cutoff function.
\end{proof}

\subsection{$W^{2,p}$-estimate: nondegenrate}

\begin{thm}\label{w2pestimates}
For any $p\geq 1$, there exits a constant $C(p)$ such that
\begin{align}
\int_M (\tr_{\om}\om_\vphi)^p \om^n\leq C(p).
\end{align}
Here, $C(p)$ depends on the quantities in \eqref{backgroundmetriccurvature}, and
\begin{align*}
\|F+f\|_\infty,\quad\|\vphi\|_\infty,\quad\|e^{-f}\|_{L^{(n-1)p^2-np+2p}(\om)},\quad\sup_X \Theta,\quad\inf_X R,\quad n .
\end{align*}
When $0<p<1$, the same inequality holds with constant $C(1)$.
\end{thm}
\begin{proof}
\textbf{Step 1: using cone reference metric.}
Following Yau's computation, we have
\begin{align*}
\tri_\vphi \log (\tr_{\om}\om_\vphi) \geq \frac{g_\vphi^{k\bar l}{R^{i\bar j}}_{k\bar l}(\om)g_{\vphi i\bar j}-g^{i\bar j}R_{\vphi i\bar j}}{\tr_{\om}\om_\vphi}.
\end{align*}
Since $Ric(\om_\vphi)=Ric(\om)-i\p\bar\p  F$,
\begin{align*}
\tri_\vphi \log (\tr_{\om}\om_\vphi) \geq \frac{g_\vphi^{k\bar l}{R^{i\bar j}}_{k\bar l}(\om)g_{\vphi i\bar j}-S(\om)+\tri F}{\tr_{\om}\om_\vphi}.
\end{align*}
Then we follow the same argument in Section 2.0.1 and 2.2 in \cite{MR4020314} to deal with the curvature terms.
According to the geometrical polyhomogeneity of the reference metric, there is a function $\phi$ such that
\begin{align}\label{backgroundmetriccurvature}
\tilde \om=C\cdot \om +i\p\bar\p \phi\geq 0,\quad\|\phi\|_\infty\leq C_\phi,\quad
R_{i\bar j k\bar l}(\om )\geq -(\tilde g)_{i\bar j} \cdot (g )_{k\bar l},
\end{align}
for some fixed constants $C,C_\phi$.
Then by Paun's trick we have
		\begin{align}\label{tritrvphiom}
		\tri_\vphi \log (\tr_{\om}\om_\vphi) \geq - (C\tr_\vphi\om+\tri_\vphi\phi)+\frac{\tri F}{\tr_{\om}\om_\vphi}.
		\end{align}

\textbf{Step 2: differential inequality with weight $\phi$.}
We compute by using \eqref{triFvphi} and \eqref{tritrvphiom},
\begin{align*}
&\tri_{\vphi} [-\alpha(F+Bf+A\vphi)+\log \tr_{\om}\om_\vphi]\\
&\geq-\alpha[\tr_{\vphi}\Theta-R-(1-B)\tri_{\vphi}f+A(n-\tr_\vphi\om)]- (C\tr_\vphi\om+\tri_\vphi\phi)+\frac{\tri F}{\tr_{\om}\om_\vphi}.
\end{align*}
We choose $A=| \sup_X\Theta|+\alpha^{-1}(C+2)+n^{-1}|\inf_X R|$ such that
\begin{align*}
A_{ \Theta}&:=-\alpha \sup_X\Theta+\alpha A-C>1.
\end{align*} 
We also denote $A_R:=\alpha\inf_X R-\alpha An<0$.
We define 
\begin{align}\label{udefn}
u=-\alpha(F+Bf+A\vphi)+\phi.
\end{align}
The inequality above is rewritten as
\begin{align*}
\tri_{\vphi} [u+\log \tr_{\om}\om_\vphi]
\geq A_R+\alpha(1-B)\tri_{\vphi}f+A_{ \Theta}\tr_\vphi\om +\frac{\tri F}{\tr_{\om}\om_\vphi}.
\end{align*}
We set $v=e^{u}\tr_{\om}\om_\vphi.$
Then using $\tr_{\om} \om_\vphi\cdot\tr_\vphi\om\geq e^{\frac{-F}{n-1}}(\tr_{\om}\om_\vphi)^{1+\frac{1}{n-1}}$, we have
\begin{align*}
\tri_\vphi v\geq A_R v+\alpha(1-B)\tri_{\vphi}f v+A_{ \Theta}e^{u-\frac{F}{n-1}}(\tr_{\om}\om_\vphi)^{1+\frac{1}{n-1}}+\tri Fe^{u}.
\end{align*}

\textbf{Step 3: integral inequality with weight $\phi$.}
We multiply the differential inequality above  with $v^{p-1}$ and integrate by parts 
\begin{align*}
&\int_X (p-1)v^{p-2}|\p v|^2_\vphi \om_\vphi^n
=\int_X v^{p-1}(-\tri_\vphi v)\om_\vphi^n\\
&\leq-\int_X v^{p-1}\{A_R v+\alpha(1-B)\tri_{\vphi}f v+A_{\Theta}e^{u-\frac{ F}{n-1}}(\tr_{\om}\om_\vphi)^{1+\frac{1}{n-1}}+\tri Fe^{u}\}\om_\vphi^n\\
&:=I+II+III +IV.
\end{align*}
So,
\begin{align}\label{integral inequality W2p}
(p-1)\int_X v^{p-2}|\p v|^2_\vphi \om_\vphi^n-III\leq I+II+IV.
\end{align}
This integral inequality will be used to derive the following iteration inequality.
\begin{prop}\label{W2p iteration}There is a consonant $C$ depending on 
\begin{align*}
\|F+f\|_\infty,\quad\|\vphi\|_\infty,\quad A,\quad A_{\Theta},\quad A_R,\quad n 
\end{align*} and constants in \eqref{backgroundmetriccurvature} such that
\begin{align*}
\int_Xe^{(p+\frac{1}{n-1}-1)f} (\tr_{\om}\om_\vphi)^{p+\frac{1}{n-1}} \om^n \leq C \int_X e^{(p-1)f}(\tr_{\om}\om_\vphi)^{p} \om^n.
\end{align*}
\end{prop}
Before we show the proof of this iteration inequality, we use it to obtain the $W^{2,p}$-estimate we want.
Setting $h=\tr_{\om}\om_\vphi$ and choosing $k=p+\frac{1}{n-1}$, we may rewrite the iteration inequality as
\begin{align*}
\|e^{(k-1)f} h^{k}\|_1 \leq C\|e^{(p-\frac{p}{k})f}h^{p} \cdot e^{\frac{p-k}{k}f} \|_1.
\end{align*}
By H\"older inequality, 
\begin{align*}
\|e^{(k-1)f} h^{k}\|_1 \leq \|e^{(k-1)f} h^{k}\|_1^{\frac{p}{k}} \|e^{-f}\|_1^{\frac{k-p}{k}}.
\end{align*}
Consequently, we obtain the estimate
\begin{align*}
\|e^{(k-1)f} h^{k}\|_1
\leq  \|e^{-f}\|_1.
\end{align*}
At last, by the H\"older inequality to obtain the estimate and the integrability of $e^{-f}$,
\begin{align*}
\|h^{p}\|_1=\|e^{\frac{k-1}{k}pf}h^{p}\cdot e^{-\frac{k-1}{k}pf}\|_1\leq \|e^{(k-1)f}h^{k} \|_1^{\frac{p}{k}}\|e^{-\frac{k-1}{k-p}pf}\|_1^{\frac{k-p}{k}}.
\end{align*} Hence, we complete the proof for any $p\geq 1$.
\end{proof}

Then we return back to prove Proposition \ref{W2p iteration}.
\begin{proof}[Proof of Proposition \ref{W2p iteration}]
The main term is the fourth term. The difficulty comes from removing $\tri F$. The idea is to move the operator $\tri$ on $f$ by integration by parts, and then use the boundedness of $F+ f$. We compute that
\begin{align*}
IV=&-\int_X  v^{p-1}\tri Fe^{u} \om^n_{\vphi}=-\int_X  v^{p-1}\tri Fe^{u+  F} \om^n\\
&=\frac{1}{\alpha-1}\int_X v^{p-1}e^{u+F}\tri[(u+F)-(u+\alpha F)] \om^n\\
&:=IV_1+IV_2.
\end{align*}
Proceed with the integration by parts, we reduce $IV_1$ to 
\begin{align*}
IV_1&=\frac{1}{\alpha-1}\int_X  v^{p-1}e^{u+F}\tri[u+F] \om^n\\
&=\frac{-1}{\alpha-1}\int_X  v^{p-1}e^{u+ F}|\p (u+ F)|^2 \om^n
-\frac{p-1}{\alpha-1}\int_X v^{p-2}e^{u+ F}(\p v,\p (u+ F)) \om^n.
\end{align*}
By H\"older inequality, 
\begin{align}\label{W2p IV1}
IV_1&\leq \frac{(p-1)^2}{4(\alpha-1)}\int_X  v^{p-3}e^{u+ F}|\p v|^2 \om^n\\
&\leq \frac{(p-1)^2}{4(\alpha-1)}\int_X  v^{p-3}e^{u}|\p v|^2_{\vphi}(\tr_{\om}\om_\vphi)\om^n_{\vphi}\nonumber.
\end{align}
We choose $\alpha\geq \max\{p,2\}$ such that $\frac{(p-1)^2}{4(\alpha-1)}\leq \frac{p-1}{4}$.
Substituting with $v=e^{u}\tr_{\om}\om_\vphi$, we have
\begin{align*}
IV_1\leq \frac{p-1}{4}\int_X v^{p-2}|\p v|^2_{\vphi}\om^n_{\vphi}.
\end{align*}
So, $IV_1$ could be absorbed by the LHS of \eqref{integral inequality W2p}.

We now consider $IV_2$. We make use of $\tri\vphi=\tr_{\om}\om_\vphi-n\leq \tr_{\om}\om_\vphi$ and $\tri\phi=\tr_{\om}\tilde \om-nC \geq-nC$ to see that
\begin{align*}
\tri(u+\alpha F)=\tri[-\alpha Bf-\alpha A \vphi+\phi] 
&\geq -\alpha B \tri f-\alpha A\tr_{\om}\om_\vphi-nC.
\end{align*}
Substituting into $IV_2$ and using $\alpha-1\geq\frac{\alpha}{2}\geq 1$ and also $\tr_{\om}\om_\vphi\geq 1$ (otherwise we are done), we obtain that
\begin{align*}
IV_2&=\frac{-1}{\alpha-1}\int_X  v^{p-1}e^{u+F}\tri(u+\alpha F) \om^n\\
&\leq \frac{1}{\alpha-1}\int_X v^{p-1}e^{u}[\alpha B \tri f+\alpha A\tr_{\om}\om_\vphi+nC] \om^n_{\vphi}\\
&\leq \frac{\alpha B}{\alpha-1}\int_X  v^{p-1}e^{u} \tri f\om^n_{\vphi}+(2A+nC)\int_X  v^{p} \om^n_{\vphi}.
\end{align*}
We use the second term 
\begin{align*}
II=-\alpha(1-B)\int_X  v^{p-1}e^{u} \tri_{\vphi}f \tr_{\om}\om_\vphi \om_\vphi^n,
\end{align*}
and choose $B=1-\alpha^{-1}$ such that $1-B=\frac{B}{\alpha-1}>0$. We see that
\begin{align*}
&II+\frac{\alpha B}{\alpha-1}\int_X  v^{p-1}e^{u} \tri f \om^n_{\vphi}\\
&=-\alpha(1-B)\int_X  v^{p-1}e^{u} \tri_{\vphi}f \tr_{\om}\om_\vphi \om_\vphi^n+\frac{\alpha B}{\alpha-1}\int_X  v^{p-1}e^{u} \tri f \om^n_{\vphi}\\
&=-\alpha(1-B)\int_X v^{p-1}e^{u} [\tri_{\vphi}f \tr_{\om}\om_\vphi- \tri f] \om_\vphi^n
\leq 0.
\end{align*}
Here, we use $\tri_{\vphi}f \tr_{\om}\om_\vphi- \tri f\geq 0$. Note that $
I=-A_R\int_X  v^{p} \om^n_{\vphi}$.
Inserting the inequalities above into \eqref{integral inequality W2p}, we thus obtain that  
\begin{align}\label{integral inequality W2p short}
-III&\leq I+II+IV_2=(2A+nC-A_R)\int_X  v^{p} \om^n_{\vphi}.
\end{align}

From \eqref{udefn}, the third term becomes 
\begin{align*}
-III&= A_{\Theta}\int_X e^{pu-\frac{F}{n-1}+F}(\tr_{\om}\om_\vphi)^{p+\frac{1}{n-1}} \om^n\\
&=A_{\Theta}\int_X e^{(1-\frac{1}{n-1}-p\alpha)F-p\alpha Bf-p\alpha A\vphi+p\phi}(\tr_{\om}\om_\vphi)^{p+\frac{1}{n-1}} \om^n.
\end{align*}
We make use of $v=e^{-\alpha(F+Bf+A\vphi)+\phi}\tr_{\om}\om_\vphi$ and $\alpha B=\alpha-1$, so we have
\begin{align*}
-III\leq(2A+nC-A_R)\int_X e^{(1-p\alpha )F+(p-p\alpha) f-p\alpha A\vphi+p\phi}(\tr_{\om}\om_\vphi)^p\om^n.
\end{align*}
The iteration inequality is thus obtain directly by using $\|F+f\|_\infty$, $\|\vphi\|_\infty$, $\|\phi\|_\infty$.

\end{proof}

\subsection{$W^{2,p}$-estimate for degenerate metrics}\label{W{2,p}-estimate for degenerate metrics}

From now on, we still use $h$ to denote $\tr_{\om_{\mathbf b}}\om_\vphi$.

\begin{thm}\label{w2pestimates degenerate}

For any $p\geq 1$, there exits a constant $C$ such that
\begin{align}\label{w2p iteration inequality iteration}
\int_Xe^{( p-1)f}h^{p}|s|^{\sigma_L}_{h_E}\om^n_{\mathbf b}
\leq C  \| |s|^{\sigma_R} \|_{{\frac{p_0}{p_0-2p\alpha+1}},\om_{\mathbf b}}.
\end{align}
Here, \begin{align}\label{sigma R sigma}
\sigma_R&=\sigma+(p\alpha+p-1)a_0\inf_{(X,(1+\eps)\om_{K})}\Theta,\\
\sigma&=p\alpha A a_0+p\alpha a_0\sup_{(X,(1+\eps)\om_{K})}\Theta,\nonumber\\
A&=| \sup_{(X,(1+\eps)\om_{K})}\Theta|+\alpha^{-1}(C+2)+n^{-1}|\inf_X R|\nonumber
\end{align} is given in \eqref{sigma R}, 
\begin{align}\label{sigma L sigma}
\sigma_L=\sigma-(p\alpha-1)a_0\inf_{(X,(1+\eps)\om_{K})}\Theta
\end{align} is in \eqref{sigma L} and $C$ depends on the quantities in \eqref{backgroundmetriccurvature}, and
\begin{align*}
&\sup_X(F+f+\phi_l),\quad
\|e^{-f}\|_{p_0,\om_{\mathbf b}},p_0>2p\alpha-1,\\
&\|\vphi\|_\infty, \quad \sup_{(X,(1+\eps)\om_{K})}\Theta,\quad \inf_{(X,(1+\eps)\om_{K})}\Theta,\quad\inf_X R,\quad n , \quad p
\end{align*}
and the bound from the weights $$ \|\phi\|_\infty,\quad\|\vphi_\theta\|_\infty,\quad \|\Phi_{\mathbf b}\|_\infty,\quad\|\log\frac{\om^n}{\om^n_{\mathbf b}}\|_\infty,\quad\inf_X\tr_{\om_{\mathbf b}}\om_K,\quad\sup_X\tr_{\om_{\mathbf b}}\om_{sr}.$$
\end{thm}
\begin{proof}
\textbf{Step 1: differential inequality with weights.}
We choose $\om_{\mathbf b}=\om_K+\eps\om_K+i\p\bar\p \Phi_{\mathbf b}$ (omitting lower index $\eps$) in \eqref{backgroundmetriccurvature}
as the background metric. As in \textbf{Step 1}, we have
\begin{align}\label{tritrvphiom degenerate}
\tri_\vphi \log (\tr_{\om_{\mathbf b}}\om_\vphi) \geq - (C\tr_\vphi\om_{\mathbf b}+\tri_\vphi\phi)+\frac{\tri \tilde F}{\tr_{\om_{\mathbf b}}\om_\vphi}.
\end{align}
Here $\tilde F=F+\log\frac{\om^n}{\om^n_{\mathbf b}}$ and $\tri$ is defined regarding to $\om_{\mathbf b}$. In order to obtain the differential inequality as in \textbf{Step 2}, we recall from \eqref{Phi degenerate} that 
\begin{align*}
\om=\om_{\vphi_{\theta_\eps}}=\om_K+\eps\om_K+i\p\bar\p \Phi,\quad
\om_\vphi=\om+i\p\bar\p\vphi.
\end{align*}
Similar to \lemref{control Theta}, we have
\begin{align*}
\Theta+i\p\bar\p\phi^{\mathbf b}_u&\leq \sup_{(X,(1+\eps)\om_{K})}\Theta\cdot \om_{\mathbf b},
\end{align*}
with $
\phi^{\mathbf b}_u:= \sup_{(X,(1+\eps)\om_{K})}\Theta \cdot (\Phi_{\mathbf b}-\phi_E).
$

We let $\tilde f=f-\log\frac{\om^n}{\om^n_{\mathbf b}}$ and set
\begin{align*}
u=-\alpha[\tilde F+B\tilde f+A(\vphi+\vphi_\theta-\phi_E-\Phi_{\mathbf b})+\phi_u^{\mathbf b}]+\phi.
\end{align*}
Then 
\begin{align*}
&\tri_{\vphi} u\\
&=-\alpha[\tr_{\vphi}\Theta-R-(1-B)\tri_{\vphi}\tilde f+A(n-\tr_\vphi\om_{\mathbf b})+\tri_\vphi\phi^{\mathbf b}_u]+\tri_\vphi\phi.
\end{align*}
Letting 
\begin{align}
A&:=| \sup_{(X,(1+\eps)\om_{K})}\Theta|+\alpha^{-1}(C+2)+n^{-1}|\inf_X R|,\label{W2p constant A}\\
A_{ \Theta}&:=-\alpha \sup_{(X,(1+\eps)\om_{K})}\Theta+\alpha A-C>1,\nonumber\\
A_R&:=\alpha\inf_X R-\alpha An<0,\nonumber
\end{align} 
we get by \eqref{tritrvphiom degenerate},
\begin{align*}
\tri_{\vphi} [u+\log \tr_{\om_{\mathbf b}}\om_\vphi]
\geq A_R+\alpha(1-B)\tri_{\vphi}\tilde f+A_{ \Theta}\tr_\vphi\om_{\mathbf b} +\frac{\tri \tilde F}{\tr_{\om_{\mathbf b}}\om_\vphi}.
\end{align*}
Taking exponentiation $v=e^{u}\tr_{\om_{\mathbf b}}\om_\vphi$, we have
\begin{align}\label{W2p degenerate differential inequality}
\tri_\vphi v\geq A_R v+\alpha(1-B)\tri_{\vphi}\tilde f v+A_{ \Theta}e^{u-\frac{\tilde F}{n-1}}(\tr_{\om_{\mathbf b}}\om_\vphi)^{1+\frac{1}{n-1}}+\tri \tilde Fe^{u}.
\end{align}

\textbf{Step 2: integral inequality with weights.}
It is written in the integral form 
\begin{align*}
&\int_X (p-1)v^{p-2}|\p v|^2_\vphi \om_\vphi^n
=\int_X v^{p-1}(-\tri_\vphi v)\om_\vphi^n\\
&\leq-\int_X v^{p-1}\{A_R v+\alpha(1-B)\tri_{\vphi}\tilde f v+A_{\Theta}e^{u-\frac{ \tilde F}{n-1}}(\tr_{\om_{\mathbf b}}\om_\vphi)^{1+\frac{1}{n-1}}+\tri \tilde Fe^{u}\}\om_\vphi^n\\
&:=I+II+III +IV.
\end{align*}
The next step is to simplify this integral inequality by using the bounds of $\tilde F$ and $\vphi$.
The forth term is then divided into two terms.
\begin{align*}
IV&=\frac{1}{\alpha-1}\int_X v^{p-1}e^{u+\tilde F}\tri[(u+\tilde F)-(u+\alpha\tilde F)] \om_{\mathbf b}^n\\
&:=IV_1+IV_2.
\end{align*}
As \eqref{W2p IV1}, with $\alpha\geq \max\{p,2\}$, we get
\begin{align*}
IV_1\leq \frac{p-1}{4}\int_X v^{p-2}|\p v|^2_{\vphi}\om^n_{\vphi}.
\end{align*}

For $IV_2$, we compute
\begin{align*}
\tri(u+\alpha\tilde F)
&=-\alpha B \tri \tilde f-\alpha A(\tr_{\om_{\mathbf b}}\om_\vphi-n)-\alpha\tri \phi_u^{\mathbf b}+\tri\phi\\
&\geq -\alpha B \tri \tilde f-\alpha A\tr_{\om_{\mathbf b}}\om_\vphi-\alpha\tri \phi_u^{\mathbf b}-nC.
\end{align*}
Also, by $\om_{sr}+i\p\bar\p \phi_E=\om_K>0$, 
\begin{align*}
\tri \phi_u^{\mathbf b}
&=\sup_{(X,(1+\eps)\om_{K})}\Theta \cdot \tri (\Phi_{\mathbf b}-\phi_E) \\
&=\sup_{(X,(1+\eps)\om_{K})}\Theta \cdot \tr_{\om_{\mathbf b}}[\om_{\mathbf b}-(1+\eps)\om_K+\om_{sr}]
\end{align*}
is bounded above by some constant $C$. In summary, we have
\begin{align*}
-\tri(u+\alpha\tilde F)
\leq \alpha B \tri \tilde f+\alpha A\tr_{\om_{\mathbf b}}\om_\vphi+C.
\end{align*}
So,
\begin{align*}
IV_2\leq \frac{\alpha B}{\alpha-1}\int_X  v^{p-1}e^{u} \tri \tilde f\om^n_{\vphi}+C(n,A)\int_X  v^{p} \om^n_{\vphi}.
\end{align*}
Similar to \eqref{integral inequality W2p short}, we thus have
\begin{align}\label{integral inequality W2p short degenerate}
-III&\leq I+II+IV_2=(C(n,A)-A_R)\int_X  v^{p} \om^n_{\vphi}.
\end{align}

Furthermore, there exists a constant $C$ depending on $n$, $A$, $A_R$, $\|\vphi\|_\infty$, $\|\phi\|_\infty$, $\|\vphi_\theta\|_\infty$, $\|\Phi_{\mathbf b}\|_\infty$, $\|\log\frac{\om^n}{\om^n_{\mathbf b}}\|_\infty$ (bounded by \eqref{lift reference metric log Fano app both 3 change}) such that
\begin{align*}
RHS&=C(n,A,A_R)\int_X  e^{(1-p\alpha)\tilde F-p\alpha B\tilde f-p\alpha A(\vphi+\vphi_\theta-\phi_E-\Phi_{\mathbf b})-p\alpha\phi_u^{\mathbf b}+p\phi}h^{p} \om^n_{\mathbf b}\\
&\leq C \int_X  e^{(1-p\alpha) F-p\alpha B f+p\alpha A\phi_E-p\alpha\phi_u^{\mathbf b}}h^{p} \om^n_{\mathbf b}.
\end{align*}
Using $\alpha B = \alpha -1$, we have
\begin{align}\label{w2p rhs}
RHS&= C \int_X  e^{(1-p\alpha) F+p(1-\alpha ) f+p\alpha A\phi_E-p\alpha\phi_u^{\mathbf b}}h^{p} \om^n_{\mathbf b}\nonumber\\
&\leq C \int_X e^{(1-p\alpha) F+p(1-\alpha ) f}h^{p}|s|^{\sigma}_{h_E} \om^n_{\mathbf b}.
\end{align}
The power \begin{align}\label{sigma}
\sigma=p\alpha A a_0+p\alpha a_0\sup_{(X,(1+\eps)\om_{K})}\Theta.
\end{align}

Then we compute the LHS of \eqref{integral inequality W2p short degenerate},
\begin{align*}
-III&=A_{\Theta}\int_X v^{p-1}e^{u-\frac{ \tilde F}{n-1}}h^{1+\frac{1}{n-1}}\om_\vphi^n\\
&=A_{\Theta}\int_X  e^{(-p\alpha -\frac{1}{n-1}+1)\tilde F-p\alpha  B\tilde f-\alpha pA(\vphi+\vphi_\theta-\phi_E-\Phi_{\mathbf b})-p\alpha \phi_u^{\mathbf b}+p\phi}h^{p+\frac{1}{n-1}}\om_\vphi^n.
\end{align*}
The is a constant $C$ depending on $\alpha$, $A$, $A_{\Theta}$, $\|\vphi\|_\infty$, $\|\phi\|_\infty$, $\|\vphi_\theta\|_\infty$, $\|\Phi_{\mathbf b}\|_\infty$, $\|\log\frac{\om^n}{\om^n_{\mathbf b}}\|_\infty$ such that
\begin{align*}
-III&\geq C \int_X e^{(-p\alpha -\frac{1}{n-1}+1)F-p\alpha B f+p\alpha A\phi_E-p\alpha \phi_u^{\mathbf b}}h^{p+\frac{1}{n-1}}\om_{\mathbf b}^n\\
&=C \int_X e^{(-p\alpha -\frac{1}{n-1}+1)F+p(1-\alpha) f+p\alpha A\phi_E-p\alpha \phi_u^{\mathbf b}}h^{p+\frac{1}{n-1}}\om_{\mathbf b}^n.
\end{align*}
Inserting the formulas of $\phi_u^{\mathbf b}$, we have 
\begin{align}\label{w2p lhs}
-III\geq C \int_X e^{(-p\alpha -\frac{1}{n-1}+1)F+p(1-\alpha) f} h^{p+\frac{1}{n-1}}|s|^{\sigma}_{h_E}\om_{\mathbf b}^n.
\end{align}

In conclusion, letting $k=p+\frac{1}{n-1}$, we have obtained the iteration inequality
\begin{align}\label{w2p iteration inequality}
\|e^{(-p\alpha +1-\frac{1}{n-1})F} h^{k}\cdot e^{p(1-\alpha) f}|s|^{\sigma}_{h_E}\|_1
\leq C  \|  e^{(-p\alpha+1) F} h^{p}\cdot e^{p(1-\alpha ) f} |s|^{\sigma}_{h_E} \|_1.
\end{align}

\textbf{Step 3: iteration}
Actually, we could apply the iteration inequality to prove the following weighted inequality. It follows from applying the H\"older inequality to the RHS of \eqref{w2p iteration inequality}, 
\begin{align*}
& \|  e^{(-p\alpha+1) F} h^{p}\cdot e^{p(1-\alpha ) f} |s|^{\sigma}_{h_E} \|_1\\
 &=\|[e^{(-p\alpha +1-\frac{1}{n-1})F} h^{k}\cdot e^{p(1-\alpha) f}|s|^{\sigma}_{h_E} ]^\frac{p}{k}\cdot [e^{\frac{p\alpha+p-1}{(n-1)k}F}(e^{p(1-\alpha)f}|s|^{\sigma})^{\frac{k-p}{k}}]\|_1\\
&\leq \|e^{(-p\alpha +1-\frac{1}{n-1})F} h^{k}\cdot e^{p(1-\alpha) f}|s|^{\sigma}_{h_E} \|_1^{\frac{p}{k}}\cdot\|e^{(p\alpha+p-1)F}e^{p(1-\alpha)f}|s|^{\sigma}\|_1^{\frac{k-p}{k}}.
\end{align*} 
In summary, we have
\begin{align*}
\|  e^{(-p\alpha+1) F} h^{p}\cdot e^{p(1-\alpha ) f} |s|^{\sigma}_{h_E} \|_1\leq C \|e^{(p\alpha+p-1)F}e^{p(1-\alpha)f}|s|^{\sigma}\|_1.
\end{align*} 

\textbf{Step 4: simplification}
It remains to further simply the both sides. Due to Proposition \ref{upper F degenerate}, we use $\sup_X(F+f+\phi_l)$ to have
\begin{align*}
RHS&\leq C \|e^{(p\alpha+p-1)(-f-\phi_l)}e^{p(1-\alpha)f}|s|^{\sigma}\|_1\\
&= \|e^{(-2p\alpha+1)f} e^{-(p\alpha+p-1)\phi_l} |s|^{\sigma} \|_1.
\end{align*}
By the definition of $\phi_l$ in \lemref{control Theta},
\begin{align}
RHS\leq \|e^{(-2p\alpha+1)f} |s|^{ \sigma_R} \|_1.
\end{align}
Here, from $\sigma$ define in \eqref{sigma} and $A$ defined in \eqref{W2p constant A}, we have the power
\begin{align}\label{sigma R}
&\sigma_R=\sigma+(p\alpha+p-1)a_0\inf_{(X,(1+\eps)\om_{K})}\Theta.\end{align}
At last, we choose $p_0>2p\alpha-1$ such that $e^{-f}$ is $L^{p_0}$ for $p_0>1$ to conclude that 
\begin{align*}
RHS\leq \| |s|^{\sigma_R} \|_{\frac{p_0}{p_0-2p\alpha+1}}.
\end{align*}
We further compute by using $\sup_X(F+f+\phi_l)$ again,
\begin{align*}
LHS&\geq C\|  e^{(-p\alpha+1) (-f-\phi_l)} h^{p}\cdot e^{p(1-\alpha ) f} |s|^{\sigma}_{h_E} \|_1\\
&= C \|  e^{(p-1)f } h^{p}\cdot e^{(p\alpha-1 ) \phi_l} |s|^{\sigma}_{h_E} \|_1
\end{align*}
Set the power
\begin{align}\label{sigma L}
&\sigma_L=\sigma-(p\alpha-1)a_0\inf_{(X,(1+\eps)\om_{K})}\Theta.\end{align}
We have
\begin{align}
LHS&\geq C \|  e^{(p-1)f } h^{p} |s|^{\sigma_L}_{h_E} \|_1
\end{align}
Therefore, we have completed the proof.

\end{proof}


\subsubsection{A different treatment of the iteration inequality}
If we utilise both the upper and lower bound of $F$ in \eqref{w2p rhs} and \eqref{w2p lhs}, we could also have the following iteration inequality.
\begin{lem}For any $p\geq 1$, it holds
\begin{align}\label{w2p iteration inequality}
\|e^{( k-1)f}h^{k}|s|^{\tilde \sigma_L(p)}_{h_E}\|_1
\leq C  \| e^{(p-1)f}h^{p}|s|^{\tilde \sigma_R(p)}_{h_E} \|_1.
\end{align}
Here the powers \begin{align*}
\tilde\sigma_R(p)
&=a_0\sup_{(X,(1+\eps)\om_{K})}\Theta+p\alpha A a_0>0,\\
\tilde\sigma_L(p)=&(p \alpha +\frac{1}{n-1}-1)(-a_0\inf_{(X,(1+\eps)\om_{K})}\Theta)\\
&+p\alpha a_0\sup_{(X,(1+\eps)\om_{K})}\Theta
+p\alpha A a_0>0.
\end{align*}
In particular, when $p=1$, it holds that
\begin{align}\label{iteration p=1}
&\|e^{\frac{1}{n+1}f} h^{1+\frac{1}{n+1}} |s|^{\tilde\sigma_L(1)}_{h_E} \|_1
\leq C.
\end{align}
\end{lem}
\begin{proof}
We prove the inequality for $p=1$.
When $p=1$, we have from integrating both sides of \eqref{W2p degenerate differential inequality} that
\begin{align*}
0\leq-\int_X \{A_R v+\alpha(1-B)\tri_{\vphi}\tilde f v+A_{\Theta}e^{u-\frac{ \tilde F}{n-1}}(\tr_{\om_{\mathbf b}}\om_\vphi)^{1+\frac{1}{n-1}}+\tri \tilde Fe^{u}\}\om_\vphi^n.
\end{align*}
Following the same argument as the proof above, we will arrive at
\begin{align*}
\|e^{\frac{1}{n+1}f} h^{1+\frac{1}{n+1}} |s|^{\tilde\sigma_L(1)}_{h_E} \|_1
\leq C  \| {h} |s|_{h_E}^{\tilde\sigma_R}\|_1 .
\end{align*}
Since $\tilde\sigma_R>0$, we get
\begin{align}\label{p=1}
&  \| {h} |s|_{h_E}^{\tilde\sigma_R}\|_1 \leq   \| {h} \|_1
=C\int_X \tr_{\om_{\mathbf b}}\om_\vphi \om_{\mathbf b}^n\\
&=C\int_X n+\tri_{ \om_{\mathbf b}}(\vphi+\Phi_\theta-\Phi_{\mathbf b})  \om_{\mathbf b}^n
\leq Cn Vol.\nonumber
\end{align} 
Consequently, it implies an integral inequality of $h$,
that is \eqref{iteration p=1}.
It is slightly different from what we got in \eqref{w2p iteration inequality iteration} when $p$ is chosen to be 1.
\end{proof}
\begin{rem}
The iteration inequality \eqref{w2p iteration inequality} also leads to weighted estimates of $h$ for all $p\geq 1$, we leave this further discussion to interested readers.
\end{rem}
\subsection{$C^{1,1}$ estimate: nondegenrate}
\subsubsection{Gradient estimate of $F+f$}
\begin{thm}\label{Gradient estimates}
	There holds
	\begin{align*}
	\sup_X \|\p (F+f)\|^2_\vphi\leq C.
	\end{align*}
	The constant $C$ depends on the quantities in \eqref{backgroundmetriccurvature}, $\|R\|_\infty, \|\Theta\|_\infty, \inf_XRic(\om), n$, and
	\begin{align*}
	\|F+f\|_\infty,\quad\|\vphi\|_\infty,\quad \|e^{-f}\|_{L^{\frac{b}{b-2}}(\om)},\quad \|\tr_{\om}\om_{\varphi}\|_{L^{\frac{(2n-2)b}{b-2}}(\om)},\quad b>2 .
	\end{align*}
\end{thm}
We let $Z$ be any given function and set $\tilde F:=F+Z$. We also set \begin{align*}
	w:=e^{\frac{\tilde F}{2}}|\p\tilde F|^2_\vphi.
	\end{align*}
We will need the following inequality on $\tilde F$.
\begin{lem}
	\begin{align}\label{gradientW2pinequalitypre pre}
	\tri_\vphi w
	&\geq e^{\frac{\tilde F}{2}}\{
	(\nabla\tri_\vphi\tilde F,\nabla\tilde F)_\vphi+\frac{\tri_\vphi \tilde F}{2}|\p\tilde F|^2_\vphi\\
	&+
	(R_{i\bar j}(\om)+Z_{i\bar j})g_\vphi^{i\bar l}g_\vphi^{k\bar j}\tilde F_k\tilde F_{\bar l}+|\p\bar\p\tilde F|^2_\vphi\}.\nonumber
	\end{align}
\end{lem}
\begin{proof}
	We also denote 
	$(\nabla a,\nabla b)_\vphi
	=g_\vphi^{i\bar j}(a_ib_{\bar j}+a_ib_{\bar j})$.
	Then we have
	\begin{align}\label{gradientW2pinequalitypre pre identity}
	e^{-\frac{\tilde F}{2}}\tri_\vphi w=\frac{1}{2}\tri_\vphi\tilde F|\p\tilde F|^2_\vphi
	+\frac{1}{4}|\p\tilde F|_\vphi^4
	+\frac{1}{2}(\nabla\tilde F,\nabla |\p\tilde F|^2_\vphi)_\vphi+\tri_\vphi|\p\tilde F|^2_\vphi.
	\end{align}

	We compute the forth term in \eqref{gradientW2pinequalitypre pre identity}
	\begin{align*}
	\tri_\vphi|\p\tilde F|^2_\vphi
	&=(\nabla\tri_\vphi\tilde F,\nabla\tilde F)_\vphi+Ric(\om_{\varphi})(\p\tilde F,\bar \p\tilde F)+|\p\p\tilde F|^2_\vphi+|\p\bar\p\tilde F|^2_\vphi.
	\end{align*}
	We use $Ric(\om_{\vphi})=Ric(\om)-i\p\bar\p F$ to see in local coordinate that
	\begin{align}\label{gradientricvphi}
	Ric(\om_{\varphi})(\p\tilde F,\bar \p\tilde F)
	&=
	R_{i\bar j}(\om)g_\vphi^{i\bar l}g_\vphi^{k\bar j}\tilde F_k\tilde F_{\bar l}-g_\vphi^{i\bar l}g_\vphi^{k\bar j}F_{i\bar j}\tilde F_k\tilde F_{\bar l}\nonumber\\
	&=(R_{i\bar j}(\om)+Z_{i\bar j})g_\vphi^{i\bar l}g_\vphi^{k\bar j}\tilde F_k\tilde F_{\bar l}-g_\vphi^{i\bar l}g_\vphi^{k\bar j}\tilde F_{i\bar j}\tilde F_k\tilde F_{\bar l}.
	\end{align}
	
	The third term in \eqref{gradientW2pinequalitypre pre identity} reads
	\begin{align*}
	\frac{1}{2}(\nabla\tilde F,\nabla |\p\tilde F|^2_\vphi)_\vphi
	=\frac{1}{2}g_\vphi^{i\bar j}g_{\vphi}^{k\bar l}[\tilde F_i\tilde F_{k\bar j}\tilde F_{\bar l} +\tilde F_{\bar j}\tilde F_{k}\tilde F_{\bar l i}+\tilde F_i\tilde F_{k}\tilde F_{\bar l\bar j}+\tilde F_{\bar j}\tilde F_{ki}\tilde F_{\bar l}],
	\end{align*}
	whose mixed derivatives could cancel the second term in \eqref{gradientricvphi} and the pure ones cancel the second term $\frac{1}{4}|\p\tilde F|_\vphi^4$ in \eqref{gradientW2pinequalitypre pre identity} and $|\p\p\tilde F|^2_\vphi$, since 
	\begin{align*}
	|\p\p\tilde F|^2_\vphi+\frac{1}{4}|\p\tilde F|_\vphi^4+\frac{1}{2}g_\vphi^{i\bar j}g_{\vphi}^{k\bar l}[\tilde F_i\tilde F_{k}\tilde F_{\bar l\bar j}+\tilde F_{\bar j}\tilde F_{ki}\tilde F_{\bar l}]\geq 0.
	\end{align*}
	Therefore, combining the identities above, we have proved the lemma.
\end{proof}

We use the assumption that $\tr_{\om}\om_\vphi\geq 1$, otherwise we are done.

\begin{proof}[proof of \thmref{Gradient estimates}]\textbf{Step 1: differential inequality.}
	We further simplify \eqref{gradientW2pinequalitypre pre} with $Z=f$. 	We let \begin{align*}
	u=w +1=e^{\frac{\tilde F}{2}}|\p\tilde F|^2_\vphi +1.
	\end{align*}	 
	
We rewrite \eqref{gradientW2pinequalitypre pre} as following
\begin{lem}
	\begin{align}\label{gradientW2pinequality}
	\tri_\vphi u 
	&\geq e^{\frac{\tilde F}{2}}(\nabla\tri_\vphi\tilde F,\nabla\tilde F)_\vphi+\tilde A\cdot[1+ e^{-F}(\tr_{\om}\om_\vphi)^{n-1}]u.
	\end{align}
	Here \begin{align*}
	\tilde A&:=\frac{1}{n-1}[\frac{\inf_X\Theta}{2}+\inf_XRic(\om) ]-\sup_X \chi_0-\frac{\sup_X R}{2},
	\end{align*}
\end{lem}	
\begin{proof}	By using $\tri_\vphi \tilde F=\tr_{\vphi}\Theta-R$ and $\tr_{\vphi}\om\leq \frac{1}{n-1}e^{-F}(\tr_{\om}\om_\vphi)^{n-1}$, we see that the second term in \eqref{gradientW2pinequalitypre pre} is
	\begin{align*}
	\frac{\tr_{\vphi}\Theta-R}{2}u
	&\geq 
	\frac{\inf_X\Theta\cdot\tr_{\vphi}\om-R}{2}u\\
	&\geq \frac{\inf_X\Theta}{2(n-1)}\cdot e^{-F}(\tr_{\om}\om_\vphi)^{n-1} u -\frac{\sup_X R}{2}u.
	\end{align*}
Similarly, we simply the third term in \eqref{gradientW2pinequalitypre pre},
	\begin{align*}
	e^{\frac{\tilde F}{2}} R_{i\bar j}(\om)g_\vphi^{i\bar l}g_\vphi^{k\bar j}\tilde F_k\tilde F_{\bar l}
	&\geq \inf_XRic(\om) \cdot \tr_\vphi\om \cdot u\\
	&\geq \frac{\inf_X Ric(\om) }{n-1}  \cdot e^{-F} (\tr_{\om}\om_\vphi)^{n-1} \cdot u.
	\end{align*}
We obtain with the aid of the relation $i\p\bar\p f\geq -\chi_0$ that
		\begin{align*}
	e^{\frac{\tilde F}{2}}f_{i\bar j}g_\vphi^{i\bar l}g_\vphi^{k\bar j}\tilde F_k\tilde F_{\bar l}
	&\geq -e^{\frac{\tilde F}{2}} (\chi_0)_{i\bar j}g_\vphi^{i\bar l}g_\vphi^{k\bar j}\tilde F_k\tilde F_{\bar l}
	\geq -\sup_X \chi_0   \cdot \tr_\vphi\om \cdot u\\
	&\geq -\sup_X \chi_0 \cdot  e^{-F} (\tr_{\om}\om_\vphi)^{n-1}\cdot  u.
	\end{align*}
	
Substituting all these inequalities into \eqref{gradientW2pinequalitypre pre} and setting the constant $\tilde A$, we prove this lemma. 
\end{proof}

	\textbf{Step 2: integral inequality.}
	Setting \begin{align*} 
v=u^{\frac{p}{2}}\text{ and }h=\tr_{\om}\om_\vphi,
\end{align*}  we rewrite the differential inequality in the integral form, 
	\begin{lem}
		\begin{align}\label{gradient2ndkeyintegral} 
	\int_X \frac{2(p-1)}{p^2}|\p v|^2_\vphi \om^n
	\leq C p \{\int_X [e^{-F}h^{2n-2}v^2+ e^Fv^2]\om^n\}.
	\end{align} 
	\end{lem}
\begin{proof}	We multiply the differential inequality above with $-u^{p-1}$ for some $p\geq 1$. After applying the integration by parts, we get
	\begin{align}\label{gradient2nddiffinequ}
	&\int_X (p-1)u^{p-2}|\p u|^2_\vphi \om_\vphi^n\\
	&\leq \int_X [-e^{\frac{\tilde F}{2}}(\nabla\tri_\vphi\tilde F,\nabla\tilde F)_\vphi u^{p-1} -\tilde Au^p-\tilde A e^{-F} (\tr_{\om}\om_\vphi)^{n-1}u^p]\om_\vphi^n\nonumber.
	\end{align}
	Applying the integration by parts again, the first term $I$ in \eqref{gradient2nddiffinequ} becomes,
	\begin{align}\label{gradient2ndmixedterm}
	I&=\int_X [2e^{\frac{\tilde F}{2}}(\tri_\vphi\tilde F)^2 u^{p-1}
	+ e^{\frac{\tilde F}{2}}|\p\tilde F|^2_\vphi \tri_\vphi\tilde F u^{p-1}\\
	&+2(p-1)e^{\frac{\tilde F}{2}}\tri_\vphi\tilde F(\p u,\p \tilde F)_\vphi u^{p-2}]\om_\vphi^n:=I_1+I_2+I_3\nonumber.
	\end{align}
	By the H\"older inequality, the last term $I_3$ in \eqref{gradient2ndmixedterm} is bounded by
	\begin{align*}
	I_3\leq\int_X \frac{p-1}{2}u^{p-2}|\p u|^2_\vphi \om_\vphi^n+\int_X 2(p-1)e^{\tilde F} (\tri_{\vphi}\tilde F)^2 |\p\tilde F|^2_\vphi u^{p-2}\om_{\varphi}^n.
	\end{align*}
	Using $e^{\frac{\tilde F}{2}}|\p\tilde F|^2_\vphi\leq u$, we get
	\begin{align*}
	\int_X 2(p-1)e^{\tilde F} (\tri_{\vphi} \tilde F)^2 |\p\tilde F|^2_\vphi u^{p-2}\om_{\varphi}^n&\leq 2(p-1)\int_Xe^{\frac{\tilde F}{2}}(\tri_\vphi\tilde F)^2u^{p-1}\om_{\varphi}^n\\
	&=(p-1)I_1 .
	\end{align*}
	So we have
	\begin{align*}
	I_3\leq\int_X \frac{p-1}{2}u^{p-2}|\p u|^2_\vphi \om_\vphi^n+(p-1)I_1 .
	\end{align*}
	By $e^{\frac{\tilde F}{2}}|\p \tilde F|^2_\vphi\leq u$ again, we have 
	\begin{align*}
	I_2\leq \int_X  |\tri_\vphi\tilde F | u^{p}\om_\vphi^n.
	\end{align*}
	Substituting into \eqref{gradient2nddiffinequ}, we obtain that
	\begin{align}\label{gradient2ndkeyintegralpre} 
	\int_X \frac{p-1}{2}u^{p-2}|\p u|^2_\vphi \om_\vphi^n&\leq \int_X [2pe^{\frac{\tilde F}{2}}(\tri_\vphi\tilde F)^2 u^{p-1}
	+ |\tri_\vphi\tilde F | u^{p}]\om_\vphi^n\\
	&-\tilde A\int_X [ u^p+ e^{-F} (\tr_{\om}\om_\vphi)^{n-1}u^p]\om_\vphi^n.\nonumber
	\end{align}
	Since $\tri_{\vphi}\tilde F=\tr_{\vphi}\Theta-R$, we have by \eqref{trivphiom} that 
	\begin{align*}
	|\tri_{\vphi}\tilde F|
	&\leq \|\Theta\|_\infty\tr_{\vphi}\om+\|R\|_\infty\\
	&\leq \frac{\|\Theta\|_\infty}{n-1}e^{-F}(\tr_{\om}\om_\vphi)^{n-1}+\|R\|_\infty.
	\end{align*}
	We now substitute it into \eqref{gradient2ndkeyintegralpre} and make use of the assumption $\tr_{\om}\om_\vphi\geq 1$.
	Accordingly, we see that there is a constant $C$ depending on $\|\tilde F\|_\infty, \|\Theta\|_\infty, \|R\|_\infty$, $\tilde A$, $n$ such that
	\begin{align*}
	&\int_X \frac{p-1}{2} u^{p-2}|\p u|^2_\vphi \om_\vphi^n\\
	&\leq C p \{\int_X [e^{-2F}(\tr_{\om}\om_\vphi)^{2n-2}u^{p-1}+u^{p-1}+e^{-F} (\tr_{\om}\om_\vphi)^{n-1}u^p+u^p]e^F\om^n\}.
	\end{align*}
Making use of 
\begin{align*} 
u^{p-1}\leq u^p,\quad e^{-F} (\tr_{\om}\om_\vphi)^{n-1}u^p\leq e^{-2F}(\tr_{\om}\om_\vphi)^{2n-2}u^{p}+u^p,
\end{align*} 
we conclude that the right hand side of the inequality above is bounded by
\begin{align*} 
C p \{\int_X [e^{-2F}(\tr_{\om}\om_\vphi)^{2n-2}u^{p}+u^p]e^F\om^n\}.
\end{align*} 
We use the lower bound of $F$ (Proposition \ref{Flower}) to the left hand side, that is $\om^n_\vphi=e^F\om^n\geq e^{\inf_X F}\om^n $. In conclusion, we have 	the lemma.
\end{proof}	

	\textbf{Step 3: iteration.}
	In order to proceed the iteration procedure, we change $|\p v|_\vphi$ to $|\p v|$ and use the H\"older's inequality,
	\begin{align*}
	\int_X|\p v|^a \om^n
	\leq \int_X|\p v|_\vphi^a h^{\frac{a}{2}} \om^n
	\leq (\int_X|\p v|_\vphi^2 \om^n)^{\frac{a}{2}}
	\cdot(\int_Xh^{\frac{a}{2-a}} \om^n)^{\frac{2-a}{2}}.
	\end{align*} 
	The inequality \eqref{gradient2ndkeyintegral} is rewritten as
	\begin{align}\label{gradient2ndkeyintegral pre}
	\|\p v\|_a^2
	\leq C\|h\|_{\frac{a}{2-a}}p^2\int_X (e^{-F}h^{2n-2}v^{2}+e^{F}v^{2})\om^n.
	\end{align}
	We apply the H\"older's inequality to the terms on the right hand side
	\begin{align*}
	\int_X e^Fv^2\om^n\leq e^{\sup_X(F+f)}\int_X e^{-f}v^2\om^n\leq e^{\sup_X(F+f)}(\int_X e^{\frac{-b}{b-2}f}\om^n)^{\frac{b-2}{b}}  (\int_X v^{b}\om^n)^{\frac{2}{b}},
	\end{align*}
	and also by using the $F$ lower bound (Proposition \ref{Flower}), 
	\begin{align*}
	\int_X e^{-F}h^{2n-2}v^2\om^n &\leq e^{-\inf_X F}(\int_X h^{\frac{(2n-2)b}{b-2}}\om^n)^{\frac{b-2}{b}} (\int_X v^{b}\om^n)^{\frac{2}{b}}.
	\end{align*}
	We denote the Sobolev exponent $a^\ast=\frac{2na}{2n-a}$ and let $a<2<b<a^\ast$. The Sobolev inequality gives that
	\begin{align*}
	\|v\|_{a^\ast}^2\leq C(\|\p v\|_{a}^2+\|v\|_{a}^2).
	\end{align*}
	In conclusion, by the $W^{2,p}$-estimates (\thmref{w2pestimates}), there is a constant $C$ depending on $\sup_X(F+f)$, $\inf_X F$, $\|e^{-f}\|_{\frac{b}{b-2};\om}$, $\|h\|_{\frac{(2n-2)b}{b-2};\om}$ such that
	\begin{align*}
	\|v\|_{a^\ast}^2\leq C p^2\|v\|_{b}^2.
	\end{align*}
	Rewrite it in $u$, we get
	\begin{align*}
	\|u\|_{a^\ast\frac{p}{2}}\leq (C(\frac{p}{2})^2)^\frac{2}{p}\|u\|_{b\frac{p}{2}}.
	\end{align*}
	We choose $\chi=\frac{a^\ast}{b}>1$, $\frac{p}{2}=\chi^{i}$, $i=0,1,\cdots$. The iteration procedure tells us
	\begin{align*}
	\|u\|_{b\chi^{m}} \leq \Pi_{i=0}^{m-1}( C \chi^{i})^{\chi^{-i}}\|u\|_{b}.
	\end{align*}
	Letting $m\rightarrow\infty$, $\tilde C:=\Pi_{i=0}^{\infty}( C \chi^{i})^{\chi^{-i}}\leq  C^{\sum_{i=0}^{\infty}\chi^{-i}}\chi^{\sum_{i=0}^{\infty}i\chi^{-i}} <\infty$, we obtain the estimate
	\begin{align*}
	\|u\|_{\infty} \leq \tilde C\|u\|_{b}\leq \|u\|^{1-\frac{1}{b}}_{\infty}\|u\|^{\frac{1}{b}}_{1}.
	\end{align*}
	In order to bound $\|u\|_{1}$, we use the integration by parts to see
	\begin{align}\label{L1tildeF}
	\int_X|\p \tilde F|^2_\vphi\om_\vphi^n
	=\int_X \tilde F (R-\tr_{\vphi}\Theta)\om_\vphi^n
	\leq C ( 1+\int_X (\tr_{\om}\om_{\varphi}))^{n-1}\om^n.
	\end{align}
	The constant $C$ depends on $\|\tilde F\|_\infty$, $\|R\|_\infty$ and $\inf_X \Theta$.
	Therefore, the $\|u\|_{1}$ norm is bounded and the required estimate follows from the $W^{2,p}$-estimates (\thmref{w2pestimates}).
	
\end{proof}
\subsubsection{$f=0$}
\begin{thm}\label{C11estimates}Suppose that $f=0$.
	There holds
	\begin{align*}
	\sup_X \|\p F\|^2_\vphi+\sup_X \tr_{\om}\om_\vphi\leq C.
	\end{align*}
	The constant $C$ depends on the quantities in \eqref{backgroundmetriccurvature}, and
	\begin{align*}
	\|F\|_\infty,\quad\|\vphi\|_\infty,\quad \|R\|_\infty,\quad\|\Theta\|_\infty,\quad\inf_X Ric(\om) ,\quad  n.
	\end{align*}
\end{thm}

\begin{proof}\textbf{Step 1: differential inequality.}
	We denote 
	\begin{align*}
	w=e^{\frac{ F}{2}}|\p F|^2_\vphi,\quad 
u=w+ e^\phi \tr_{\om}\om_\vphi +1.
	\end{align*}
	The weight $\phi$ is added to control the cone singularities. The differential inequality is still the same form as \eqref{gradient2ndkeyintegral}.
\begin{lem}
	\begin{align*}
\tri_\vphi u 
&\geq e^{\frac{F}{2}}(\nabla\tri_\vphi F,\nabla F)_\vphi+\tilde A\cdot[ (\tr_{\om}\om_\vphi)^{n-1}u+u].
\end{align*}
Here \begin{align*}
\tilde A&:=A_{\Theta,F,R}=\frac{1}{n-1}[\frac{\inf_X\Theta}{2}+\inf_XRic(\om) ]e^{-\sup_X F}-\frac{\sup_X R}{2}\\
&-\frac{C}{n-1}e^{-\inf_XF}-\frac{1}{4}e^{-\frac{\inf_XF}{2} +\sup_X \phi}.
\end{align*}
\end{lem}	
\begin{proof}	
Recall \eqref{tritrvphiom} that
\begin{align*}
\tri_\vphi \log\tr_{\om}\om_\vphi \geq - (C\tr_\vphi\om+\tri_\vphi\phi)+\frac{\tri F}{\tr_{\om}\om_\vphi}.
\end{align*}
Using $\tr_\vphi\om\leq \frac{1}{n-1}e^{-F}(\tr_{\om}\om_\vphi)^{n-1}$ and $\tri F\leq \tr_\vphi\om\tri_\vphi F$, we have
	\begin{align*}
	\tri_\vphi (e^\phi\tr_{\om}\om_\vphi )
	&\geq -Ce^\phi\tr_\vphi \om\cdot \tr_{\om}\om_\vphi+e^\phi\tri F\\
	&\geq -\frac{C}{n-1}e^{-F+\phi} (\tr_{\om}\om_\vphi)^n-e^{\frac{F}{2} }|\p\bar\p F|^2_\vphi-\frac{1}{4}e^{\frac{-F}{2} +2\phi}(\tr_{\om}\om_\vphi)^2\\
	&\geq -\frac{C}{n-1}e^{-\inf_X F} (\tr_{\om}\om_\vphi)^{n-1}u-e^{\frac{F}{2} }|\p\bar\p F|^2_\vphi\\
	&-\frac{1}{4}e^{\frac{-\inf_XF}{2} +\sup_X\phi}(\tr_{\om}\om_\vphi)u.
	\end{align*}

We also have from the proof of \eqref{gradientW2pinequality},
	\begin{align}
	\tri_\vphi w
	&\geq e^{\frac{F}{2}}(\nabla\tri_\vphi F,\nabla  F)_\vphi-\frac{\sup_X R}{2}u\\&+\frac{1}{n-1}[\frac{\inf_X\Theta}{2}+\inf_XRic(\om) ]\cdot e^{-F}(\tr_{\om}\om_\vphi)^{n-1} u+e^{\frac{F}{2}}|\p\bar\p  F|^2_\vphi.\nonumber
	\end{align}
Putting together and setting the constant $\tilde A$,
we thus obtain the differential inequality.
\end{proof}

\textbf{Step 2: integral inequality.}
We multiply the differential inequality above with $-u^{p-1}$ for some $p\geq 1$ and apply the integration by parts,
\begin{align*}
&\int_X (p-1)u^{p-2}|\p u|^2_\vphi \om_\vphi^n\\
&\leq \int_X [-e^{\frac{F}{2}}(\nabla\tri_\vphi F,\nabla F)_\vphi u^{p-1} -\tilde A (\tr_{\om}\om_\vphi)^{n-1}u^p-\tilde Au^p]\om_\vphi^n\nonumber.
\end{align*}
By the same reason of \eqref{gradient2ndkeyintegralpre}, we estimate the first term on the right hand side and thus obtain that
\begin{align*}
\int_X \frac{p-1}{2}u^{p-2}|\p u|^2_\vphi \om_\vphi^n&\leq \int_X [2pe^{\frac{F}{2}}(\tri_\vphi F)^2 u^{p-1}
+ |\tri_\vphi F | u^{p}]\om_\vphi^n\\
&-\tilde A\int_X [ (\tr_{\om}\om_\vphi)^{n-1}u^p+u^p]\om_\vphi^n.
\end{align*}
Note that
\begin{align*}
|\tri_{\vphi}F|
&=|\tr_{\vphi}\Theta-R|
\leq \|\Theta\|_\infty\tr_{\vphi}\om+\|R\|_\infty\\
&\leq \frac{\|\Theta\|_\infty}{n-1}e^{-F}(\tr_{\om}\om_\vphi)^{n-1}+\|R\|_\infty.
\end{align*}
Substituting into the inequality above and using the assumption $\tr_{\om}\om_\vphi\geq 1$ and $u\geq 1$,
we thus have that there is a constant $C$ depending on $\|F\|_\infty, \|\Theta\|_\infty, \|R\|_\infty$ and $\tilde A$ such that
\begin{align*}
\int_X \frac{p-1}{2} u^{p-2}|\p u|^2_\vphi \om^n
\leq C p \{\int_X [(\tr_{\om}\om_\vphi)^{2n-1}u^p+ u^p]\om^n\}.
\end{align*}
We denote $v=u^{\frac{p}{2}}$ and $h=\tr_{\om}\om_\vphi$, and rewrite this inequality as following
\begin{align*}
\int_X \frac{2(p-1)}{p^2}|\p v|^2_\vphi \om^n
\leq C p \{\int_X [h^{2n-1}v^2+ v^2]\om^n\}.
\end{align*} 

\textbf{Step 3: iteration.}
Similar to \eqref{gradient2ndkeyintegral pre}, we have	
\begin{align}
\|\p v\|_a^2
\leq C\|h\|_{\frac{a}{2-a}}p^2\{\int_X [h^{2n-1}v^2+ v^2]\om^n\}.
\end{align}
The Sobolev inequality with exponent $a^\ast=\frac{2na}{2n-a}$ and $a<2<b<a^\ast$ implies that
\begin{align}\label{Sobolev inequality}
\|v\|_{a^\ast}^2\leq C(\|\p v\|_{a}^2+\|v\|_{a}^2)\leq C(\||\p v|_\vphi\|^2_{2}\|h\|_{\frac{a}{2-a}}+\|v\|_{a}^2).
\end{align}
By the H\"older's inequality, we get
\begin{align*}
\int_X h^{2n-1}v^2\om^n \leq \|h\|_{\frac{(2n-1)b}{b-2}}^{2n-1} \|v\|_b^2, \text{ and }
\int_X v^2\om^n\leq V^{\frac{b-2}{b}}\|v\|_b^2.
\end{align*}
Thus there is a constant $C$ depending on $\|h\|_{\frac{a}{2-a}}$, $\|h\|_{\frac{(2n-1)b}{b-2}}^{2n-1}$, $V$ such that
\begin{align*}
\|v\|_{a^\ast}^2\leq C p^2\|v\|_{b}^2.
\end{align*}
Running the iteration procedure, we obtain the estimate
\begin{align*}
\|u\|_{\infty} \leq \tilde C \|u\|^{1-\frac{1}{b}}_{\infty}\|u\|^{\frac{1}{b}}_{1}.
\end{align*}
The $\|u\|_{1}$ norm is bounded by \eqref{L1tildeF}. Therefore, the proof is completed with the aid of the $W^{2,p}$-estimate (\thmref{w2pestimates}).

\end{proof}


\section{Regularity and uniqueness of log $\chi$-twisted $K$-energy minimisers}\label{Regularity and uniqueness of log twisted energy minimisers}
In this section, we will show the regularity and uniqueness of the log $\chi$-twisted $K$-energy minimisers. They will be used in the next section to prove that the existence of cscK cone metrics implies the properness of the log $K$-energy.
\subsection{The complete space $\mathcal E^1$}

Recall that $\mathcal H$ is the space of smooth K\"ahler potentials in $\Om=[\om_0]$. Let $\vphi$ be a K\"ahler potential in $\mathcal H$, the tangent space of $\mathcal H$ at $\vphi$ is the set of smooth functions $C^{\infty}(X)$.
The $d_1$ metric is defined to be
\begin{align*}
\|\xi\|_1=\int_X |\xi|\om^n_\vphi,\text{ for all }\xi\in T_\vphi\mathcal H.
\end{align*}
We denote by $\mathcal E^1$ the finite energy class \cite{MR2352488},
\begin{align*}
\mathcal E^1=\{\vphi \text{ is $\om$-psh }\vert \int_{X}\om_{\varphi}^n=\int_{X}\om^n,\int_{X}|\vphi|\om_{\varphi}^n<\infty\}.
\end{align*}
It is proved in Theorem 2.3 in \cite{MR3406499} that the metric completion of $\mathcal H$ under $d_1$ is $\mathcal E^1$. When $\vphi_1,\vphi_2\in \mathcal E^1$, their approximation sequence are two decreasing sequence $\vphi^k_i, i=1,2$ which are  converging pointwise to $\vphi_i$. The $d_1$-geodesic between $\vphi_i$ is the limit of Chen's $C^{1,1}$ geodesic between $\vphi^k_i$. The limit is also in $\mathcal E^1$ and independent of the choice of approximation of the boundary values $\vphi_i$.
We also denote by $\mathcal E^1_0$ the functions $\vphi$ in $\mathcal E^1$ with the normalization condition $D(\vphi)=0$.

\begin{defn}We assume $\chi\geq 0$ is a smooth closed $(1,1)$-form and recall that the log $\chi$-twisted $K$-energy is \begin{align*}
	\nu_{\beta,\chi}=\nu_\beta+J_\chi.
	\end{align*}	
	We say $\vphi\in \mathcal E^1$ is a log $\chi$-twisted $K$-energy minimiser, if \begin{align*}
	\nu_{\beta,\chi}(\vphi)=\inf_{\mathcal E^1}\nu_{\beta,\chi}(\cdot).
	\end{align*}
\end{defn}
Given a K\"ahler cone potential $\vphi$, its volume form $\om^n_\vphi$ is $L^p$ for some $p>1$. Thus $\vphi$ is a H\"older continuous function and $\mathcal H_\beta\subset\mathcal E^1$.
The space of K\"ahler cone potential $\mathcal H_\beta$ is dense in $\mathcal E^1$. The results in Section \ref{Log K-energy and convexity} are extended to the log $K$-energy over $\mathcal H_\beta$.
\begin{lem}\label{convexityKenergy}We assume $\chi\geq 0$ is a smooth closed $(1,1)$-form.
	The log $\chi$-twisted $K$-energy $\nu_{\beta,\chi}:\mathcal H_\beta\rightarrow \mathbb R$ could be extended to $\nu_{\beta,\chi}:\mathcal E^1\rightarrow \mathbb R$.
	Furthermore, the log $\chi$-twisted $K$-energy is $d_1$ lower semi-continuous and convex along the $d_1$-geodesic. 
\end{lem}
\begin{proof}
From Definition \eqref{logKenergy}, we recall the formula of the log $\chi$-twisted $K$-energy
\begin{align*}
\nu_\beta(\vphi)
&=\frac{1}{V}\int_M\log\frac{\om^n_\vphi}{\om_\theta^n}\om_\vphi^{n}
+J_{-\theta}(\vphi)+J_\chi(\vphi)+\frac{1}{V}\int_M (-(1-\b)\log |s|_h^2+h_0)\om_0^n.
\end{align*}
Theorem 4.7 in \cite{MR3687111} assume that the twisted term $\tilde\chi=\tilde\chi_0+i\p\bar\p f$ satisfies $\tilde\chi_0$ is a smooth closed positive $(1,1)$-form, $f$ is a $\tilde\chi_0$-psh function and $e^{-f}\in L^1(\om^n_0)$. The twisted $K$-energy in this theorem uses the formula 
\begin{align*}
\nu_{\tilde\chi}=\frac{1}{V}\int_M\log\frac{\om^n_\vphi}{e^{-f}\om_0^n}\om_\vphi^{n}
+J_{-Ric(\om_0)+\tilde{\chi_0}}(\vphi)-\frac{1}{V}\int_Mf\om_0^n.
\end{align*}
In order o apply this theorem, we see that in our case, $f=(1-\b)\log |s|_h^2-h_0$ and $\tilde\chi_0=Ric(\om_0)-\theta+\chi$.
It is direct to see that $e^{-f}=|s|^{2\beta-2}_h\in L^1(\om_0)$ and 
\begin{align*}
\tilde\chi
&=Ric(\om_0)-\theta+\chi+i\p\bar\p((1-\b)\log |s|_h^2-h_0)\\
& =(1-\beta)\Theta_D+\chi+(1-\b)i\p\bar\p\log |s|_h^2\\
&=\chi+(1-\b)i\p\bar\p\log |s|^2\geq 0.
\end{align*}
In the computation above, we use \eqref{h0} and the assumption that $\chi\geq 0$.

\end{proof}

With the same reason to \lemref{convexityKenergy}, we apply Corollary 4.8 in \cite{MR3687111} (which is obtained from a compactness Theorem 2.17 in \cite{MR3956691}) to the log $\chi$-twisted $K$-energy.
\begin{lem}\label{Weakcompactness}We assume $\chi\geq 0$ is a smooth closed $(1,1)$-form.
	Suppose a sequence $\vphi_i\in\mathcal E^1$ satisfies both the log $\chi$-twisted $K$-energy
	$\nu_{\beta,\chi}(\vphi_i)$ and $d_1(0,\vphi_i)$
	are uniformly bounded.
	Then there exists a $d_1$-convergent subsequence.
\end{lem}


\subsection{Regularity of log $\chi$-twisted $K$-energy minimisers}\label{Regularity of minimisers}
\begin{thm}[Regularity of minimisers]\label{Regularity}
Suppose $\chi\geq0$ is a smooth closed $(1,1)$-form.
Then the log $\chi$-twisted $K$-energy minimisers are $D_{\bf w}^{4,\a,\b}(\om_\theta)$.
\end{thm}
\begin{proof}
	Assume that $\vphi_{min}$ is a minimiser of the log $\chi$-twisted $K$-energy $\nu_{\beta,\chi}$. Then there exists a sequence $\vphi_j\in\mathcal H$ (c.f. Lemma 3.1 in \cite{MR3687111}) such that
	\begin{align*}
	d_1(\vphi_j,\vphi_{min})\rightarrow 0 \text{ and } \nu_{\beta,\chi}(\vphi_j)\rightarrow \nu_{\beta,\chi}(\vphi_{min}).
	\end{align*}
	
	\textbf{Smoothing $\vphi_j$ by the twisted cone path $\vphi_j^t$}:
	 We fix $\vphi_j$ and connect it to $\vphi_{min}$ by the path of $\vphi_j^t$ with $\vphi_j^0=\vphi_j$ such that $\vphi_j^t$ minimise the twisted $K$-energy 
\begin{align}\label{Kt}
K_t=t\nu_{\beta,\chi}+(1-t)J_{\om_{\vphi_j}}.
\end{align} The path satisfies the equation
\begin{align}\label{2Twisted cone path}
t[S-\underline S_\b-(\tr_\vphi\chi-\underline \chi)]=(1-t)(\tr_\vphi\om_{\vphi_j}-n).
\end{align}
When $t\in (0,1)$, we have
\begin{align*}
K_t(\vphi)\geq t\nu_{\beta,\chi}(\vphi_{min})+(1-t)J_{\om_{\vphi_j}}(\vphi)
\end{align*} for all $\vphi\in\mathcal H$. Since $J_{\om_{\vphi_j}}$ is proper \cite{MR3698234}, we have $K_t$ is also proper. 
The existence of the solution $\vphi_j^t$ to \eqref{2Twisted cone path} is guaranteed by \thmref{properclosedness}.


\textbf{Getting uniform entropy bound of $\vphi_j^t$}:
Since $\vphi_j^t$ is also the global minimiser of $K_t$ (\lemref{globalminimiser}), we get $K_t(\vphi_j^t)\leq K_t(\vphi_j)$. Then by using $J_{\om_{\vphi_j}}(\vphi_j)\leq J_{\om_{\vphi_j}}(\vphi_j^t)$, we have uniform energy upper bound
\begin{align}\label{nuvphijt}
\nu_{\beta,\chi}(\vphi_j^t)\leq \nu_{\beta,\chi}(\vphi_j).
\end{align}

Similarly, we use $K_t(\vphi_j^t)\leq K_t(\vphi_{min})$ and $
\nu_{\beta,\chi}(\vphi_{min})\leq \nu_{\beta,\chi}(\vphi_j^t)$ to prove
\begin{align}\label{Jvphijt}
J_{\om_{\vphi_j}}(\vphi_j^t)\leq J_{\om_{\vphi_j}}(\vphi_{min}).
\end{align}
According to Lemma 5.2-5.6 in \cite{arXiv:1801.00656}, \eqref{Jvphijt} implies uniform distance upper bound,
\begin{align}
&\sup_{0.1<t<1}I(\vphi_j^t,\vphi_{min})\rightarrow 0,\text{ as }j\rightarrow \infty,\label{Iunuqieness}\\
&\sup_{j,0.1<t<1}d_1(0,\vphi_j^t)\leq C\label{d1vphijt},
\end{align}
for some constant depending on $\sup_{j}d_1(0,\vphi_j)$ and $n$.

By \lemref{d_1andJ}, we have 
\begin{align*}
&|J_\chi(\vphi_j^t)|\leq C(n)\|\chi\|_{\infty}d_1(0,\vphi_j^t),\\
&|J_{-\theta}(\vphi_j^t)|\leq C(n)\|\theta\|_{\infty}d_1(0,\vphi_j^t).
\end{align*} Recall $
\nu_{\beta,\chi}
=E_\beta+J_{-\theta}+J_\chi+\frac{1}{V}\int_M (h_0+\mathfrak h)\om_0^n$ (Definition \ref{logKenergy}). Thus, by \eqref{nuvphijt} and \eqref{d1vphijt}, the entropy is bounded as
\begin{align*}
\sup_{j,0.1<t<1}E_\beta(\vphi_j^t)=\frac{1}{V}\int_M\log\frac{\om^n_{\vphi_{j}^t}}{\om^n_{\theta}}\om_{\vphi_{j}^t}^{n}\leq C.
\end{align*}

\textbf{To conclude the convergence of $\vphi_j^t$}:
According to \corref{properclosedness estimates} with the twisted term $\chi+\frac{1-t}{t}\om_{\vphi_j}\geq 0$, we have the estimates 
\begin{align}\label{vphijt}
\|\vphi_j^t\|_{D^{4,\a,\b}_{\bf w}(\om_\theta)}\leq C,\quad \forall j,\quad0.1<t<1.
\end{align} The constant $C$ depends on 
	\begin{align*}
	\sup_{j,0.1<t<1} E_\beta(\vphi_{j}^t), \quad \|\theta\|_{C^{0,\a,\b}},\quad\|\chi+\frac{1-t}{t}\om_{\vphi_j}\|_{C^{0,\a,\b}},\quad \alpha_1,\quad \alpha_\beta,\quad\underline S_\b,\quad n.
	\end{align*}

After taking $t\rightarrow 1$, we have $\vphi_j^t$ converge to $\vphi^1_j$ in $D^{4,\a,\b}_{\bf w}(\om_\theta)$. Note that \begin{align*}
\|\chi+\frac{1-t}{t}\om_{\vphi_j}\|_{C^{0,\a,\b}}\rightarrow \|\chi\|_{C^{0,\a,\b}},
\end{align*} as $t\rightarrow1$. The $D^{4,\a,\b}_{\bf w}(\om_\theta)$-norm of $\vphi^1_j$ is independent of $j$. Thus a subsequence of $\vphi^1_j$ converges to $\vphi_\infty$ in $D^{4,\a,\b}_{\bf w}(\om_\theta)$.
Due to \eqref{Iunuqieness}, we have \begin{align*}I(\vphi_\infty,\vphi_{min})=0
\end{align*} and conclude that
\begin{align*}
\vphi_{min}=\vphi_\infty\in D^{4,\a,\b}_{\bf w}(\om_\theta).
\end{align*}
\end{proof}

\subsection{Uniqueness of log $\chi$-twisted $K$-energy minimisers}\label{Uniqueness of minimisers}
We close this section by using the regularity theorem to prove the following uniqueness of the log $\chi$-twisted $K$-energy minimiser.

\begin{prop}[Uniqueness of minimisers]\label{weakUniqueness}Suppose $\chi>0$ is a smooth closed $(1,1)$ form.
	Then the log $\chi$-twisted $K$-energy minimiser is unique. 
\end{prop}
\begin{proof}
	According to the regularity theorem \thmref{Regularity}, log $\chi$-twisted $K$-energy minimiser is $D^{4,\a,\b}_{\bf w}(\om_\theta)$. Then this lemma follows from \lemref{uniquechipositive}.
\end{proof}

When considering the smooth case $\beta=1$ and $\chi$ is a K\"ahler form. The uniqueness of the $\chi$-twisted Mabuchi energy in proved in Theorem 4.13 in \cite{MR3687111}.

\section{Existence implies properness and geodesic stability}\label{Existence implies properness and geodesic stability}

The properness conjecture predicts the properness of Mabuchi's $K$-energy is equivalent to the existence of cscK metrics.
We then formulate the properness of an energy functional $F$ in the space of smooth K\"ahler potentials $\mathcal H$ associated with the K\"ahler class $\Om$.

The properness conjecture for cscK metrics is formulated by Chen regarding to the $L^2$-geodesic distance \cite{arXiv:0809.4081}, in Donaldson's framework of space of K\"ahler metrics \cite{MR1736211}; and later was updated in Darvas-Rubinstein \cite{MR3600039} by using the $d_1$-geodesic distance, which is after Darvas' work \cite{MR3702499,MR3406499} on the metric completion $\mathcal E^1_0$ of the space of K\"ahler potentials. We refer the readers to the expository article \cite{arXiv:1401.7857} and the works \cite{MR3406499,MR3702499,MR3687111,arXiv:1602.03114,MR3090260,MR3956691,MR3600039} and references therein on the $d_1$-metric and the weak topology from the rapid growing literature. We would like to bring to the attention of readers that in the K\"ahler-Einstein problem, the properness of Ding functional \cite{MR967024} with respect to Aubin's $I$-functional was introduced by Tian \cite{Ding-Tian,MR1471884} in the 1980s, which is motivated from the Moser-Trudinger inequality on $S^2$.

With the preparation above, we are ready to prove "existence implies properness" and extends the geodesic stability results to cscK cone metrics. The proofs are similar to \cite{arXiv:1801.00656}. We put them as following for readers' convenience.
\subsection{Existence implies properness}\label{Existence implies properness}

\begin{defn}
	We say the log $K$-energy is coercive, if there exists positive constants $A$ and $B$ such that $\nu_\beta(\vphi)\geq A \cdot d_{1,G}(\vphi,0)-B$ for all $\vphi\in \mathcal H$.
\end{defn}
\begin{thm}\label{Existence implies properness discrete}Assume that the automorphism group $Aut(X;D)$ is discrete.
Suppose that there exists a constant scalar curvature K\"ahler cone metric $\om_{cscK}$. Then the log $K$-energy is proper. Moreover, 
the log $K$-energy is coercive.
\end{thm}
\begin{proof}Let $\om_{cscK}=\om_0+i\p\bar\p\vphi_{cscK} $.
	We will prove that there exists a constant $A$ such that 
	\begin{align*}
	\nu_\beta(\vphi)\geq A \cdot d_1(\vphi,\vphi_{cscK})+\nu_\beta(\vphi_{cscK}),\quad \forall\vphi\in \mathcal H.
	\end{align*}
We prove by contradiction method. We assume that there exists a sequence $\vphi_i$ such that
$A_i=\frac{\nu_\beta(\vphi_i)-\nu_\beta(\vphi_{cscK})}{d_1(\vphi_i,\vphi_{cscK})}\rightarrow 0$.

We connect $\vphi_{cscK}$ to $\vphi_i$ by unit speed $d_1$-geodesic $e_i(s)$. We pick up a point $e_i(1)$, that is 
\begin{align*}
d_1(e_i(1),\vphi_{cscK})=1.
\end{align*}
The convexity \lemref{convexityKenergy} implies that 
\begin{align*}
\nu_\beta(e_i(1))\leq (1-\frac{1}{d_1(\vphi_i,\vphi_{cscK})})\nu_\beta(\vphi_{cscK})+\frac{\nu_\beta(\vphi_i)}{d_1(\vphi_i,\vphi_{cscK})}=\nu_\beta(\vphi_{cscK})+A_i.
\end{align*} 
Then by the weak compactness \lemref{Weakcompactness}, there is a $d_1$-convergent subsequence of $e_i(1)$ to $e_\infty(1)$.
Taking $i\rightarrow\infty$, the lower semicontinuity \lemref{convexityKenergy} implies $\nu_\beta(e_\infty(1))\leq \nu_\beta(\vphi_{cscK})$.
As a result, $e_\infty(1)$ is also a log $K$-energy minimiser.

According to the regularity \thmref{Regularity}, $e_\infty(1)$ is a cscK cone metric.
By the assumption of trivial $Aut(X;D)$, the uniqueness \thmref{Uniqueness} implies that $e_\infty(1)$ is the same to $\vphi_{cscK}$, which contradicts to
\begin{align*}
d_1(e_\infty(1),\vphi_{cscK})=1.
\end{align*}
Thus the proof is complete.
\end{proof}

The counterpart for $\chi$-twisted cscK cone metrics is given below.
\begin{thm}Suppose $\chi$ is a smooth closed $(1,1)$-form and $\chi>0$.
Suppose that there exists a $\chi$-twisted constant scalar curvature K\"ahler cone metric $\om_{cscK}=\om_0+i\p\bar\p\vphi_{cscK}$. Then the log $\chi$-twisted $K$-energy $\nu_{\beta,\chi}$ is proper. Moreover, there exists a constant $A$ such that $\nu_{\beta,\chi}(\vphi)\geq A \cdot d_1(\vphi,\vphi_{cscK})+\nu_\beta(\vphi_{cscK})$ for all $\vphi\in \mathcal H$.
\end{thm}
\begin{proof}
Following the argument above, $e_\infty(1)$ is a log $\chi$-twisted $K$-energy minimiser. According to the uniqueness of the $\chi$-twisted $K$-energy minimiser (Proposition \ref{weakUniqueness}), we have that $e_\infty(1)=\vphi_{cscK}$. But the distance between them is 1. Contradiction.
\end{proof}

In the general automorphism group case, the distance $d_1$ needs to be replaced by the distance $d_{1,G}$.
\begin{thm}\label{Existence general}Suppose that there exists a constant scalar curvature K\"ahler cone metric $\om_{cscK}$. Then the log $K$-energy is proper. Moreover, 
the log $K$-energy is coercive.
\end{thm}
\begin{proof}
The proof is similar to the proof given above. The major new input is the use of uniqueness of cscK cone metrics, \thmref{Uniqueness}.
\end{proof}
\subsection{Geodesic stability}\label{Geodesic stability}

The geodesic stability conjecture for cscK metric was proposed in Donaldson \cite{MR1736211}, aiming to detecting the existence of cscK metrics by using the geodesic rays.

\begin{defn}
An $d_1$-geodesic ray $\{e(s);0\leq s\leq \infty\}$ in $\mathcal E^1_0$ starting with $\vphi_0$ is called $\nu_\beta$-\textit{non-increasing} ($\nu_\beta$-strictly-decreasing), if the log $K$-energy $\nu_\beta$ is $\nu_\beta$-non-increasing ($\nu_\beta$-strictly-decreasing resp.) along $e(s)$.
\end{defn}

We extend Chen's definition of geodesic stability \cite{MR2471594} to log $K$-energy. The definition is well-defined according to convexity \lemref{convexityKenergy}.
\begin{defn}Suppose $\rho(t):[0,\infty)\rightarrow \mathcal E^1_0$ is a $d_1$-geodesic ray. The $\mathfrak  F$-invariant is defined along the $d_1$-geodesic ray in term of the log $K$-energy,
\begin{align*}
\mathfrak F=\lim_{k\rightarrow \infty} \nu_\beta(\rho(k+1))- \nu_\beta(\rho(k)).
\end{align*}
\end{defn}

We say a ray $\rho(t)\in  \mathcal E^1_0$ is a \textit{holomorphic ray} starting at $\vphi_0$, if it is generated by a one parameter holomorphic action $\sigma(t)\in Aut(X;D)$, i.e. $\om_{\rho(t)}=\sigma(t)^\ast\om_{\vphi_0}$. 
According to \cite{MR2471594}, two rays $\rho_1(t), \rho_2(t)$ are said to be \textit{parallel} if they have uniformly bounded $d_1$-distance, i.e. $\sup_t d_1(\rho_1(t),\rho_2(t))$ is finite. 
We say a ray is \textit{trivial}, if it is parallel to any holomorphic ray.

\begin{defn}\label{geodesic stability}
Given $\vphi_0\in \mathcal E^1_0$. Let $\rho(t):[0,\infty)\rightarrow \mathcal E^1_0$ be any $d_1$-geodesic starting at $\vphi_0$.
\begin{itemize}
\item
The point $\vphi_0$ is called geodesic semi-stable, if the $\mathfrak F$-invariant semi-positive.
\item
The point $\vphi_0$ is called geodesic stable, if one of the following holds, 
\begin{enumerate}
\item the $\mathfrak F$-invariant is strictly positive along $\rho(t)$, 
\item it vanishes and $\rho(t)$ is trivial.
\end{enumerate}
\end{itemize}

A K\"ahler class $\Om$ is geodesic stable (geodesic semi-stable resp.), if every $\vphi_0\in \mathcal E^1_0$ is geodesic stable (geodesic semi-stable resp.).
\end{defn}


\subsubsection{Discrete automorphism group}\label{Donaldson geodesic stability conjecture: discrete automorphism group}

We extend Donaldson conjecture \cite{MR1736211} to the cscK cone metrics.
\begin{thm}\label{geodesic conjecture}
Assume that the automorphism group $Aut(X;D)$ is discrete.
The following are all equivalent,
\begin{enumerate}
\item there exists no constant scalar curvature K\"ahler cone metric;
\item there exists a $\vphi_0\in\mathcal E^1_0$ and a non-increasing $d_1$-geodesic ray in $\mathcal E^1_0$ starting with $\vphi_0$;
\item for all $\vphi\in\mathcal E^1_0$, there exists a non-increasing $d_1$-geodesic ray in $\mathcal E^1_0$ starting with $\vphi$.
\end{enumerate}
\end{thm}

The proof of this theorem is divided into the following two propositions.
\begin{prop}\label{1implies3}
	$(1)\implies(3)$.
\end{prop}
\begin{proof}

We obtain from the properness \thmref{properclosedness} that, if there is no cscK cone metric, then there exists a decreasing sequence $\vphi_i\in \mathcal H_0$ such that $d_1(0,\vphi_i)\rightarrow \infty$, as $i\rightarrow\infty$, and
\begin{align}\label{boundedsequence}
\nu_\beta(\vphi_i)\leq C.
\end{align}

\textbf{Construction of $d_1$-geodesic ray}: From assumption of $(3)$, we are given a point $\vphi$. We connect it with $\vphi_i$ by a $d_1$-geodesic $e_i(s): [0,d_1(\vphi,\vphi_i)]\rightarrow \mathcal E^1$.

From convexity \lemref{convexityKenergy}, the log $K$-energy is convex along $e_i(s)$. It implies that for $s\in [0, d_1(\vphi,\vphi_i)]$, it holds
\begin{align*}
\nu_\beta(e_i(s))\leq (1-\frac{s}{d_1(\vphi,\vphi_i)})\nu_\beta(\vphi)+\frac{s}{d_1(\vphi,\vphi_i)}\nu_\beta(\vphi_i).
\end{align*}
That is bounded by $\nu_\beta(\vphi)+C$, due to \eqref{boundedsequence}.

Given $s$, $\nu_\beta(e_i(s))$ and $d_1(\vphi,e_i(s))=s$ are uniformly bounded, we could apply the weak compactness \lemref{Weakcompactness} to subtract a subsequence converges to $e(s)$, as $i\rightarrow \infty$. Actually, it is a geodesic ray $\{e(s);s\in[0,\infty)\}$ starting from $\vphi$. 

We then use continuity \lemref{convexityKenergy} to conclude that the log $K$-energy is also convex along $e(s)$. Thus the log $K$-energy function is non-increasing along the geodesic ray $e(s)$.

\end{proof}

\begin{prop}
$(2)\implies(1)$.
\end{prop}
\begin{proof}
We are given an $\mathcal E_0^1$ geodesic ray $\{\vphi(t); 0\leq t\leq \infty\}$ starting from $\vphi_0\in \mathcal E_0^1$ such that the log $K$-energy is non-increasing.

We prove by contradiction method. We assume there exists a cscK cone metric $\vphi$. We choose a non-increasing sequence $\vphi_i=\vphi(t_i)$ along the geodesic ray with $t_i\rightarrow \infty$.

Then following the construction above, we construct a geodesic ray $e(s)$ such that the log $K$-energy is non-increasing along it. Since $e(0)=\vphi$ is a cscK cone metric, $e(s)$ are all log $K$-energy minimiser. According to regularity \thmref{Regularity}, $e(s)$ are all cscK cone metrics. They all equal to $\vphi$ by uniqueness \thmref{Uniqueness}.

Then contradiction comes from the triangle inequality of the distance comparison. For large $s$,
\begin{align*}
s=d_1(\vphi_0,e(s))\leq d_1(\vphi_0,\vphi)+d_1(\vphi,e(s))=d_1(\vphi_0,\vphi).
\end{align*}
\end{proof}

The proof of the twisted cscK cone metric is identical. Let $\nu_{\beta,\chi}$ be the log $\chi$-twisted $K$-energy.
\begin{thm}\label{twisted geodesic conjecture}
Suppose $\chi>0$ is a smooth closed $(1,1)$ form.
The following are all equivalent,
\begin{enumerate}
\item there exists no $\chi$-twisted constant scalar curvature K\"ahler cone metric;
\item there exists a $\vphi_0\in\mathcal E^1_0$ and a $\nu_{\beta,\chi}$-strictly-decreasing $d_1$-geodesic ray in $\mathcal E^1_0$ starting with $\vphi_0$;
\item for all $\vphi\in\mathcal E^1_0$, there exists a $\nu_{\beta,\chi}$-strictly-decreasing $d_1$-geodesic ray $\{e(s);0\leq s\leq \infty\}$ in $\mathcal E^1_0$ starting with $\vphi$.
\end{enumerate}
\end{thm}
\begin{proof}
From the proof of \thmref{geodesic conjecture}, it remains to prove that the log $\chi$-twisted $K$-energy is strictly-decreasing in the statements $(2)$ and $(3)$.
If not, from uniqueness Proposition \ref{weakUniqueness} of the log $\chi$-twisted $K$-energy minimiser, we have that $e(r)=e(s)$ for any $r,s\geq s_0$. This is a contradiction to the distance between them is $|r-s|$.
\end{proof}

Therefore, \thmref{geodesic conjecture} implies the last conclusion in \thmref{cone geodesic conjecture}, that is the geodesic stability conjecture for cscK cone metrics in terms of $\mathfrak  F$-invariant.
\begin{thm}Assume that the automorphism group $Aut(X;D)$ is discrete.
$(M,\Om)$ admits a constant scalar curvature K\"ahler cone metric if and only if it is geodesic stable.
\end{thm}


\subsubsection{General automorphism group}\label{Donaldson geodesic stability conjecture: general case}
In this section, \thmref{Existence implies properness discrete} and \thmref{geodesic conjecture} will be extended to general case, that is the automorphism group is non-trivial.

This section is a generalisation of \cite{arXiv:1801.05907}.
We remark that the new key ingredient in our proofs is the use of uniqueness of cscK cone metrics, which was proved in the previous article in this series \cite{MR4020314}.

\begin{thm}\label{geodesic conjecture general}
The following are all equivalent,
\begin{enumerate}
\item there exists no constant scalar curvature K\"ahler cone metric;
\item either $Fut_\beta\neq0$, or there exists a $\vphi_0\in\mathcal E^1_0$ and a nontrivial, $\nu_\beta$-non-increasing, $d_1$-geodesic ray in $\mathcal E^1_0$ starting with $\vphi_0$;
\item either $Fut_\beta\neq0$, or for all $\vphi\in\mathcal E^1_0$, there exists a nontrivial, $\nu_\beta$-non-increasing, $d_1$-geodesic ray in $\mathcal E^1_0$ starting with $\vphi$.
\end{enumerate}
\end{thm}
\begin{proof}
We first prove $(1)\implies(3)$. We assume that there is no cscK cone metric. 
It is sufficient to show the second criterion in $(3)$, when $Fut_\beta=0$. Actually, vanishing of the log Futaki invariant implies that the log $K$-energy $\nu_\beta$ is $G$-invariant. 

We then consider two situations.
If $\nu_\beta$ is not bounded below, then $\nu_\beta$ is strictly deceasing. This situation is already contained in $(3)$. Then we turn to the situation, when $\nu_\beta$ is bounded below. Since we assume there is no cscK cone metric, properness \thmref{Properness implies existence general} implies that there is a sequence $\vphi_i$ such that $\nu_\beta(\vphi_i)\leq C$ and
\begin{align*}
d_{1,G}(0,\vphi_i)\rightarrow \infty,\text{ as }i\rightarrow\infty.
\end{align*}
Let $\om_{\psi_i}=\sigma_i^\ast\om_{\vphi_i}$. We have by the $G$-invariant of $\nu_\beta$,
\begin{align*}
d_{1}(0,\psi_i)\rightarrow \infty,\quad \nu_\beta(\psi_i)\leq C.
\end{align*}
The construction in Proposition \ref{1implies3} gives us a geodesic ray$\{e(s);s\in[0,\infty)\}$ such that $\psi_i(s)\rightarrow e(s), i\rightarrow \infty$ starting from any given potential $\vphi_0$ such that the log $K$-energy function is non-increasing along the geodesic ray $e(s)$. 

Actually, $e(s)$ is nontrivial. Otherwise, $d_{1,G}(0,e(s))$ is bounded. But
\begin{align*}
d_{1,G}(0,e(s))&\geq d_{1,G}(0,\vphi_i)-d_{1,G}(\vphi_i,e(s))\\
&=d_{1,G}(0,\vphi_i)-d_{1}(\psi_i,e(s))\rightarrow\infty.
\end{align*}
Contradiction! So, we obtain a non-increasing, nontrivial ray $e(s)$ as we need.
	
We then prove $(2)\implies(1)$. We prove by the contradiction method. We assume $\om_\vphi$ is a cscK cone metric. Then $Fut_\beta(\om_\vphi)=0$. According to $(2)$, there exists a $\vphi_0$ and a nontrivial, non-increasing, geodesic ray $\vphi(t)$ with $\vphi(0)=\vphi_0$. 

By the construction in Proposition \ref{1implies3}, the segments connecting $\vphi_0$ and $\vphi(t)$ $d_1$-converge to a non-increasing geodesic ray $\{e(s);s\in[0,\infty)\}, i\rightarrow \infty$ starting from $\vphi$, as $t\rightarrow\infty$. 
Since $\om_\vphi$ is a cscK cone metric, $e(s)$ are all log $K$-energy minimisers. Due to regularity \thmref{Regularity} and  uniqueness \thmref{Uniqueness}, $e(s)$ is trivial, i.e. $\om_{e(s)}=\sigma(s)^\ast \om_\vphi,\forall s$. 
According to Lemma 4.6 in \cite{arXiv:1801.05907}, $\vphi(t)$ is parallel to $e(s)$. That is $\vphi(t)$ is trivial, which is a contradiction.
\end{proof}

Then \thmref{geodesic conjecture general} is restated as following,
\begin{thm}\label{geodesic conjecture general chen}
$(M,\Om)$ admits a constant scalar curvature K\"ahler cone metric if and only if it is geodesic stable.
\end{thm}
\begin{proof}
\textbf{$\Longrightarrow$.}
We assume $\om_\vphi$ is a cscK cone metric and $\rho(t)$ be the geodesic ray starting with $\vphi_0$. We need to check geodesic stability condition. Since the log $K$-energy $\nu_\beta$ is convex along $\rho(t)$, the $\mathfrak F$-invariant is strictly positive, if $\nu_\beta$ is unbounded above. 

In the case, when $\nu_\beta$ is bounded above, the same argument above works. The construction in Proposition \ref{1implies3}, convexity and continuity \lemref{convexityKenergy} gives us a geodesic ray consisting of log $K$-energy minimisers. Then regularity \thmref{Regularity} and uniqueness \thmref{Uniqueness}, $e(s)$ is trivial. Furthermore, the $\mathfrak F$-invariant vanishes and $e(s)$ is parallel to $\rho(t)$.

\textbf{$\Longleftarrow$.}
We run the continuity path
\eqref{2Twisted cone path general}
\begin{align*}
t(S-\underline S_\b)=(1-t)(\tr_\vphi\om_{\vphi_j}-n).
\end{align*}
By Proposition \ref{path except 1}, the path is solvable before $t=1$. 
Following the same argument in Lemma 4.1 and 4.3 in \cite{arXiv:1801.05907}, the log $K$-energy is $G$-invariant and bounded below. As shown in Section \ref{General case}, it is sufficient to show that
	\begin{align}\label{d1g bounded}
	\sup_{0.1< t<1} d_{1,G}(0, \vphi_t)\leq C.
	\end{align} 

We prove by the contradiction method. 
By definition, we let 
\begin{align*}
\om_{\psi_t}:=\sigma_t^\ast\om_{\vphi_t}=\om+i\p\bar\p \tilde\vphi_t,\quad d_{1,G}(0, \vphi_t)=d_{1}(0, \psi_t).
\end{align*}
We assume that
$
	\sup_{0.1< t<1} d_{1}(0,  \psi_t)\rightarrow \infty.
$
By \eqref{proper existence K upper bound} and $G$-invariant of the log $K$-energy,
	\begin{align*}
	\sup_{0.1< t<1}\nu_\beta( \psi_t)\leq \inf_{\mathcal H_\beta} \nu_\b+1.
	\end{align*}
Let $\vphi_0$ be the potential given in the geodesic stability condition.
The construction in Proposition \ref{1implies3} gives us a geodesic ray$\{e(s);s\in[0,\infty)\}$ starting with $\vphi_0$. Moreover, the log $K$-energy is non-increasing along $e(s)$. This contradicts to $(1)$ in the geodesic stability Definition \ref{geodesic stability}.

The argument of $(1)\implies(3)$ in the proof of \thmref{geodesic conjecture general}, tells us that 
\begin{align*}
d_{1,G}(0,e(s))\rightarrow\infty, \text{ as }s\rightarrow \infty.
\end{align*}
So, $e(s)$ is not trivial, which contradicts to $(2)$ in Definition \ref{geodesic stability}.

We thus obtain \eqref{d1g bounded} and complete the proof.
\end{proof}

\section{CscK cone path}\label{CscK cone path}
\begin{defn}\label{cscK cone path defn}
	We say the path $\{\om_{\varphi^\beta}=\om_0+i\p\bar \p\varphi^\beta ;0< \beta\leq 1\} $ is a cscK cone path, if $\om_{\varphi^\beta}$ is a cscK cone metric with cone angle $\beta$ along the smooth divisor $D$.
\end{defn}
We recall the equation of the cscK cone metric in Definition \ref{csckconemetricdefn},
\begin{equation}\label{cscK cone equation openness}
\left\{
\begin{aligned}
&\frac{(\om_{\varphi^\beta})^n}{\om_{\theta^\beta}^n}=e^F,\\
&\tri_{\varphi^\beta} F=\tr_{\varphi^\beta}\theta^\beta-\underline S_\b.
\end{aligned}
\right.
\end{equation}
In which, the smooth $(1,1)$-form $\theta^\beta\in C_1(X,D)$ and $\om_{\theta^\beta}$ satisfies the equation in the distribution sense,
\begin{align*}
Ric(\om_\theta)=\theta^\beta+2\pi (1-\b)[D].
\end{align*}
The topological constant
\begin{align}\label{topological condition openness}
\underline S_\b=\frac{(C_1(X,D))[\om_0]^{n-1}}{[\om_0]^n}.
\end{align}

\subsection{Deform cone angle of the cscK cone path}

\begin{thm}\label{openess}
	Assume that $Aut(X;D)$ is discrete and $C_1(L_D)> 0$.
The cscK cone path is open when $\beta>0$. Precisely, if $\om_{\varphi^{\beta_0}}$ is a cscK cone metric with cone angle $\beta_0\in (0,1)$ along the smooth divisor $D$, then there is a $\delta>0$ such that for all $\beta\in (\beta_0-\delta,\beta_0+\delta)$, there exists a cscK cone metric $\om_{\varphi^{\beta}}$ with cone angle $\beta$ along $D$.
\end{thm}
\begin{proof}
From assumption, we have a cscK cone metric $\om_{\varphi^{\beta_0}}$ with cone angle $\beta_0$. According to the necessary part of the properness theorem "existence implies properness" (\thmref{Existence implies properness discrete}), the corresponding log $K$-energy $\nu_{\beta_0}$ is proper.

\textbf{Approximate the cscK cone metric $\om_{\varphi^{\beta_0}}$.}
Now we apply the approximation scheme (Proposition \ref{twisted approximation thm})
to conclude that the cscK cone metric $\om_{\varphi^{\beta_0}}$ has a smooth approximation of the twisted cscK metric $\{\vphi^{\beta_0}_\eps,\eps\in(0,1]\}$ satisfying the twisted cscK equation
	\begin{align}\label{approximate at beta0}
S(\om_{\varphi^{\beta_0}_\eps})=\tr_{\vphi^{\beta_0}_\eps}(Ric(\om_{\theta^{\beta_0}_\eps})-\theta^{\beta_0})+\underline S_{\b_0}.
\end{align}

\textbf{Deform the cone angle of the approximate equation \eqref{approximate at beta0}.}
We now fix the parameter $\eps$ and prove the deformation of \eqref{approximate at beta0} in terms of the parameter $\beta_0$.
In order to apply the implicit function theorem, we need to prove that the kernel of the linearisation operator is trivial. We compute that the linearisation operator of \eqref{approximate at beta0} at $\om_{\varphi^{\beta_0}_\eps}$ is
\begin{align*}
L_{\varphi^{\beta_0}_\eps}(u)=\tri^2_{\varphi^{\beta_0}_\eps}u+u^{i\bar j}[Ric(\om_{\varphi^{\beta_0}_\eps})-Ric(\om_{\theta^{\beta_0}_\eps})+\theta^{\beta_0}]_{i\bar j}.
\end{align*}
Suppose $u$ is the kernel of $L_{\varphi^{\beta_0}_\eps}(u)$. We rewrite the linearisation operator in the integral form, 
\begin{align*}
\int_X u L_{\varphi^{\beta_0}_\eps}(u)\om^n_{\varphi^{\beta_0}_\eps}
&=\int_X u\tri^2_{\varphi^{\beta_0}_\eps}u+u^{i\bar j}[Ric(\om_{\varphi^{\beta_0}_\eps})-Ric(\om_{\theta^{\beta_0}_\eps})+\theta^{\beta_0}]_{i\bar j}\om^n_{\varphi^{\beta_0}_\eps}\\
&=\int_X |\p\p u|^2_{\varphi^{\beta_0}_\eps}+u^{i}[Ric(\om_{\theta^{\beta_0}_\eps})-\theta^{\beta_0}]_{i\bar j}u^{\bar j}\om^n_{\varphi^{\beta_0}_\eps}.
\end{align*}
Since $Ric(\om_{\theta^{\beta_0}_\eps})>\theta^{\beta_0}$ (\lemref{thetaepstheta}), we see that $\int_X u L_{\varphi^{\beta_0}_\eps}(u)\om^n_{\varphi^{\beta_0}_\eps}=0$ implies $u=0$.
Then we are able to conclude that there exists a $\delta_{\eps}>0$ such that for all $\beta\in (\beta_0-\delta_{\eps},\beta_0+\delta_{\eps})$, there exists a smooth twisted cscK metric $\om_{\varphi_\eps^{\beta}}$ satisfying
	\begin{align}\label{approximate at beta}
S(\om_{\varphi^{\beta}_\eps})=\tr_{\vphi^{\beta}_\eps}(Ric(\om_{\theta^{\beta}_\eps})-\theta^{\beta})+\underline S_{\b}.
\end{align}
If we could take $\eps\rightarrow 0$ in an appropriate sense to have the limit of $\delta_\eps\rightarrow \delta_0>0$, we obtain a cscK cone metric $\om_{\varphi^{\beta}}$ with cone angle $\beta$ in $(\beta_0-\delta_{0},\beta_0+\delta_{0})$ and thus prove the theorem. It then boils down to obtain a priori estimates of the sequence $\varphi^{\beta}_\eps$ in the following steps.

\textbf{Uniform distance bound.} We will prove that there exists a small positive constant $\eps_0$ such that for all $\eps\in (0,\eps_0]$ and $\beta\in (\beta_0-\delta_{\eps_0},\beta_0+\delta_{\eps_0})$, we have
\begin{align}\label{Uniform distance bound}
d_1(\varphi^{\beta}_\eps,\varphi^{\beta_0})\leq 1.
\end{align}
We leave the proof in the next Proposition \ref{distacen bound}.

\textbf{Take convergent subsequence.} Now we fix the cone angle $\beta$, and will take convergent subsequence of $\varphi^{\beta}_\eps$ such that it converges to a cscK cone metric required.

The procedure is similar to what we did in Section \ref{Solving approximate equations}.
In order to apply \thmref{a prioriestimates approximation} to get
	\begin{align*}
	\|\varphi^{\beta}_\eps\|_\infty, \quad \|F_\eps=\frac{\om_{\varphi^{\beta}_\eps}^n}{\om_{\theta^{\beta}_{\eps}}^n}\|_\infty, \quad \sup_X \|\p F_\eps\|^2_{\om_{\theta^\beta_\eps}}, \quad\sup_X \tr_{\om_{\theta^\beta_\eps}}\om_{\varphi^{\beta}_\eps}\leq C,
	\end{align*} 
it is sufficient to verify that the entropy $E_\beta^\eps=\frac{1}{V}\int_X\log\frac{\om^n_{\varphi^{\beta}_\eps}}{\om^n_{\theta^\beta_\eps}}\om_{\varphi^{\beta}_\eps}^{n}$ is bounded uniformly for all $\eps\in(0,\eps_0]$. 
This is because
	\begin{align*}	
	E_{\beta}^{\eps}
	=\nu^{\eps}_{\beta}(\vphi^{\beta}_\eps)-J_{-\theta^{\beta}_{\eps}}(\vphi^{\beta}_\eps)
	-\frac{1}{V}\int_M [-(1-\b)\log(|s|_h^2+\eps)+h^\beta_0]\om_0^n-c_{\beta},
\end{align*}
and also $\nu^{\eps}_{\beta}(\vphi^{\beta}_\eps)\leq \nu^{\eps}_{\beta}(0)=0$ and $|J_{-\theta^{\beta}_{\eps}}(\vphi^{\beta}_\eps)|\leq C(n)\max_X \|\theta^{\beta}_{\eps}\|_{\om_0} d_1(0,\vphi^{\beta}_\eps)$ as above and the distance bound \eqref{Uniform distance bound}. Hence a subsequence of the approximate sequence $\vphi^{\beta}_\eps$ converges to $\vphi_\infty$ in $C^{0,\alpha}$ in $X$ and smoothly outside $D$. Moreover, $\vphi_\infty$ satisfies \eqref{cscK cone equation openness} in the regular part $M$ and the second order estimate $C^{-1}\om_{\theta}\leq \om_{\vphi}\leq C\om_{\theta}$. Then the Schauder estimate gives $\vphi_\infty\in C^{2,\a,\b}$ and the regularity \thmref{Geometric asymptotic} implies $\vphi_\infty\in D_{\bf w}^{4,\a,\b}(\om_\theta)$.
 
\end{proof}

\begin{prop}\label{distacen bound}
	Given the approximate sequence $\varphi^{\beta}_\eps$ in the proof of \thmref{openess}.
	There exists a small positive constant $\eps_0$ such that for all $\eps\in (0,\eps_0]$ and $\beta\in (\beta_0-\delta_{\eps_0},\beta_0+\delta_{\eps_0})$, we have
	\begin{align}\label{Uniform distance bound}
	d_1(\varphi^{\beta}_\eps,\varphi^{\beta_0})\leq 1.
	\end{align}
\end{prop}
\begin{proof}
	We consider $\beta\in [\beta_0,\beta_0+\delta_{\eps_0})$, the other side is similar.
	We prove by contradiction method. Assume that there is a sequence $\eps_i\rightarrow 0$, such that $\delta_{\eps_i}\rightarrow 0$ and 
	\begin{align}\label{contradiction assumption openness}
	d_1(\varphi^{\beta_i}_{\eps_i},\varphi^{\beta_0})= 1.
	\end{align}
	In which, $\beta_i=\beta_0+\delta_{\eps_i}$ is the first cone angle moving forward $\beta_0$ when the identity above achieves.
	
	We will need several steps to achieve the contradiction.
	We set 
	\begin{align*}
	\vphi_i=\varphi^{\beta_i}_{\eps_i},\quad F_i=\log\frac{\om_{\vphi_i}^n}{\om_{\theta^{\beta_i}_{\eps_i}}^n}, \quad f_i=\log\frac{\om_{\theta^{\beta_i}_{\eps_i}}^n}{\om_{\theta^{\beta_0}_{\eps_i}}^n}, 
	\end{align*} and then rewrite \eqref{approximate at beta} into two equations and make use of the reference metric $\om_{\theta^{\beta_0}_{\eps_i}}$,
	\begin{equation}\label{approximate equation openness}
	\left\{
	\begin{aligned}
	\tilde F_i&=F_i+f_i=\log\frac{\om_{\vphi_i}^n}{\om_{\theta^{\beta_0}_{\eps_i}}^n},\\
	\tri_{\vphi_i} \tilde F_{\eps_i}&=\tr_{\vphi_i}(\theta^{\beta_i}+i\p\bar\p f_i)-\underline S_{\beta_i}.
	\end{aligned}
	\right.
	\end{equation}
	
	\textbf{Step 1: Uniform a priori estimates of $\vphi_i$}
	We will prove that 
	\begin{claim}\label{correct uniform bound openness claim}
	There is a constant $C$ independent of $i$ such that for any $p>1$,
	\begin{align}\label{correct uniform bound openness}
	\|\vphi_i\|_{\infty}, \quad \|i\p\bar \p{\vphi_i}\|_{L^p(\om_{\theta^{\beta_0}})}, \quad \|\tilde F_i\|_{W^{1,2p}(\om_{\theta^{\beta_0}})}\leq C,
	\end{align}
\end{claim}
	We start to prove this claim.
	According to \thmref{a prioriestimates approximation}, there is a constant $C$ such that 
	\begin{align}\label{uniform bound openness}
	\|\vphi_i\|_{\infty}, \quad \|\tilde F_i\|_{\infty}, \quad \sup_X \|\p \tilde F_i\|^2_{\vphi_i}, \quad \sup_X\|\tr_{\om_{\theta^{\beta_0}_{\eps_i}}}\om_{\vphi_i}\|_{L^p(\om_{\theta^{\beta_0}_{\eps_i}})}\leq C,
	\end{align} where $C$ depends on $\alpha_1,\alpha_\beta, n$, $\|\frac{\om_{\theta^{\beta_0}_{\eps_i}}^n}{\om_0^n}\|_{L^q(\om_0)}$ for some $q>1$ and the following quantities
	\begin{align}
	&\|e^{-f_i}\|_{L^{p_0}(\om_0)},\quad \|\underline S_{\beta_i}\|_\infty, \quad \|\theta^{\beta_i}\|_\infty,\quad\inf_XRic(\om_{\theta^{\beta_0}_{\eps_i}}),\quad
	E_{\beta_0}^{\eps_i}=\frac{1}{V}\int_X\tilde F_i\om_{\vphi_i}^{n}.\label{a prioriestimates approximation dependence open}	\end{align}
	In which, $p_0$ is sufficiently large and depends on $n$ and $p$.

	We now verify the constant dependence \eqref{a prioriestimates approximation dependence open} is uniformly for all $i$. That the scalar curvature average $\sup_i \underline S_{\beta_i}$ is bounded follows from the formula \eqref{topological condition openness}.
	The smooth closed $(1,1)$-form $\theta^{\beta_i}\in C_1(X,D)$ is chosen to smoothly depend on the parameter $\beta$, so $\sup_i\|\theta^{\beta_i}\|_\infty$
	is bounded. By $Ric(\om_{\theta^{\beta_0}_{\eps_i}})>\theta^{\beta_0}$ (\lemref{thetaepstheta}), we have the bound of $\inf_XRic(\om_{\theta^{\beta_0}_{\eps_i}})$ for all $i$.
	
	Then we examine $\|\frac{\om_{\theta^{\beta_0}_{\eps_i}}^n}{\om_0^n}\|_{L^q(\om_0)}$, $\|e^{-f_i}\|_{L^{p_0}(\om_0)}$ and $E_{\beta_0}^{\eps_i}$ as following.
	
	\textbf{1.} From the equation which \eqref{thetaeps}  $\om_{\theta^{\beta_0}_{\eps_i}}$ satisfies, we have
	\begin{align*}
	\frac{\om_{\theta^{\beta_0}_{\eps_i}}^n}{\om^n_0}=\frac{e^{h^{\beta_0}_0+c_{\beta_0}}}{(|s|^2_h+\eps_i)^{1-\beta_0}}\leq \frac{e^{h^{\beta_0}_0+c_{\beta_0}}}{|s|_h^{2(1-\beta_0)}}.
	\end{align*}	
	Recall that $h^{\beta_0}_0$ satisfies 
	$
	Ric(\om_0)=\theta^{\beta_0}+(1-\beta)\Theta_D+i\p\bar\p h^{\beta_0}_0,
	$
	and
	the normalisation constant $c$ satisfies that
	$
	e^{c_{\beta_0}}=\frac{\int_Me^{h^{\beta_0}_0}|s|_h^{2(\beta_0-1)}\om^n_0}{\int_Me^{h^{\beta_0}_0}(|s|^2_h+\eps)^{\beta_0-1}\om^n_0},
	$ which is bounded
	$
	1\leq e^{c_{\beta_0}}\leq 
	\frac{\int_Me^{h^{\beta_0}_0}|s|_h^{2(\beta_0-1)}\om^n_0}{\int_Me^{h^{\beta_0}_0}(|s|^2_h+1)^{\beta_0-1}\om^n_0}.
	$
	Therefore, both $h^{\beta_0}_0$ and $c_{\beta_0}$ are bounded, so the volume ratio $\frac{\om_{\theta^{\beta_0}_{\eps_i}}^n}{\om^n_0}$ is in $L^{q}(\om_0)$ for $1<q<\frac{1}{1-\beta_0}$.

	\textbf{2.} Since both $\om_{\theta^{\beta}_{\eps_i}}$ and $\om_{\theta^{\beta_0}_{\eps_i}}$ satisfy \eqref{thetaeps}, we have 
	\begin{align}\label{fi}
	e^{-f_i}=\frac{\om_{\theta^{\beta_0}_{\eps_i}}^n}{\om_{\theta^{\beta_i}_{\eps_i}}^n}=e^{h_0^{\beta_0}+c_{\beta_0}-h_0^{\beta_i}-c_{\beta_i}}(|s|^2_h+\eps_i)^{\beta_0-\beta_i}.
	\end{align}
	Since $c_{\beta_0}$, $c_\beta$, $h_0^{\beta_0}$, $h_0^{\beta}$ are bounded, we have $e^{-f_i}$ is $L^{p_0}(\om_0)$ if $p_0<\frac{1}{\beta-\beta_0}$. Consequently, $p_0$ could be very large, if $\beta$ is close to $\beta_0$.
	
	\textbf{3.} 
	From \eqref{Approximation of the log twisted K energy},
	we have
	\begin{align*}	
	E_{\beta_0}^{\eps_i}
	&=\frac{1}{V}\int_X(F_i+f_i)\om_{\vphi_i}^{n}
	= \frac{1}{V}\int_Xf_i\om_{\vphi_i}^{n}+\nu^{\eps_i}_{\beta_i}(\vphi_i)-J_{-\theta^{\beta_i}_{\eps_i}}(\vphi_i)\\
	&-\frac{1}{V}\int_M [-(1-\b_i)\log(|s|_h^2+\eps_i)+h_0]\om_0^n-c_{\beta_i},
	\end{align*}
	that is bounded independent of $i$, since $\nu^{\eps_i}_{\beta_i}(\vphi_i)\leq \nu^{\eps_i}_{\beta_i}(0)=0$ (c.f. \eqref{approximate K energy upper bound}) and $|J_{-\theta^{\beta_i}_{\eps_i}}(\vphi_i)|\leq C(n)\max_X \|\theta^{\beta_i}_{\eps_i}\|_{\om_0} d_1(0,\vphi_i)$ (\lemref{d_1andJ}) and \begin{align*}
	d_1(0,\vphi_i)\leq d_1(0,\varphi^{\beta_0})+d_1(\varphi^{\beta_0},\vphi_i)\leq d_1(0,\varphi^{\beta_0})+1
	\end{align*} by the assumption \eqref{contradiction assumption openness}.
	
	In conclusion, the constant $C$ in \eqref{uniform bound openness} is independent of $i$. Now we show that how to obtain \eqref{correct uniform bound openness} from \eqref{uniform bound openness}. We get $\|\p\tilde F_i\|_{L^{2p}(\om_{\theta^{\beta_0}_{\eps_i}})}\leq C$, from
	$|\p\tilde F_i|_{\om_{\theta^{\beta_0}_{\eps_i}}}^2\leq |\p\tilde F_i|^2_{\vphi_i}\tr_{\om_{\theta^{\beta_0}_{\eps_i}}}\om_{\vphi_i}$. Furthermore, we note that in the cone chart,  $\om_0\leq \om_{\theta^{\beta_0}_{\eps_i}}\leq C \om_{\theta^{\beta_0}}$. So we have
		\begin{align*}
\int_X	|\p\tilde F_i|_{\om_{\theta^{\beta_0}}}^{2p} \om^n_{\theta^{\beta_0}}\leq \int_X |\p\tilde F_i|^{2p}_{\om_{\theta^{\beta_0}_{\eps_i}}}(\tr_{\om_{\theta^{\beta_0}}}\om_{\theta^{\beta_0}_{\eps_i}})^p \cdot \frac{\om^n_{\theta^{\beta_0}}}{\om^n_0}\om^n_0.
	\end{align*}
Since $\frac{\om^n_{\theta^{\beta_0}}}{\om^n_0}$ is $L^q(\om_0)$ for some $q>1$, we have for some $\tilde p>p$,
			\begin{align*}
	\int_X	|\p\tilde F_i|_{\om_{\theta^{\beta_0}}}^{2p} \om^n_{\theta^{\beta_0}}\leq C\int_X |\p\tilde F_i|^{2\tilde p}_{\om_{\theta^{\beta_0}_{\eps_i}}} \om^n_0\leq C\int_X |\p\tilde F_i|^{2\tilde p}_{\om_{\theta^{\beta_0}_{\eps_i}}} \om^n_{\theta^{\beta_0}_{\eps_i}}.
	\end{align*}
	The estimate of $ \|i\p\bar \p{\vphi_i}\|_{L^p(\om_{\theta^{\beta_0}})}$ is similar.
	
\textbf{Step 2: Take a energy convergent subsequence $\vphi_{i_j}$ of $\vphi_i$}
	
Thanks to Claim \ref{correct uniform bound openness claim}, we could take a subsequence of $\vphi_i$ such that
\begin{itemize}
	\item $\vphi_{i_j}$ converges to $\vphi_\infty$ in $C^{1,\a,\b}$ and $i\p\bar \p\vphi_{i_j}\rightarrow \vphi_\infty$ weakly in $L^p(\om_{\theta^{\beta_0}})$;
	\item $\tilde F_{i_j}$ converges to $F_\infty$ in $C^{0,\a,\b}$ and $\p \tilde F_{i_j}\rightarrow \p F_\infty$ weakly in $L^p(\om_{\theta^{\beta_0}})$.
\end{itemize}
We are going to show that  $\vphi_\infty$ is a cscK cone metric. Then due to uniqueness (\thmref{Uniqueness}), we have	
\begin{align}\label{0distance openness}
d_1(\varphi^{\beta_0},\vphi_\infty)= 0. 
\end{align}

The proof is a limit procedure by applying a priori estimates from Section \ref{a priori estimates}.

\paragraph{a)}We will prove
\begin{align}\label{a openess}
\om^n_{\vphi_\infty}=e^{ F_\infty}\om^n_{\theta^{\beta_0}}\text{ weakly in } L^p.
\end{align}
	From \eqref{fi}, we know that
\begin{align*}
\int_X |f_i| \om_0^n\leq \|h_0^{\beta_0}+c_{\beta_0}-h_0^{\beta_i}-c_{\beta_i}\|_\infty+(\beta_i-\beta_0)\int_X |\log(|s|^2_h+\eps_i)| \om_0^n\rightarrow 0.
\end{align*}
Moreover, we have $\int_X |f_i|^p \om_0^n\rightarrow 0$ and $e^{-f_i}\rightarrow 1$ in $L^p(\om_{\theta^{\beta_0}})$ for some large $p$, as $i\rightarrow\infty$.	Then we are able to take limit on the both side of 
$
\om^n_{\vphi_i}=e^{F_i}\om^n_{\theta^{\beta_0}}=e^{\tilde F_i}e^{-f_i}\om^n_{\theta^{\beta_0}}
$
to get \eqref{a openess}.
	
\paragraph{b)} We choose a sequence of smooth $F_i$ converges to $F_\infty$ in $W^{1,p}(\om_{\theta^{\beta_0}})$ and approximate \eqref{a openess} by the complex Monge-Amp\`ere equation
\begin{align}\label{a approximate openess}
\om^n_{\vphi_i}=e^{ F_i}\om^n_{\theta_{\eps_i}^{\beta_0}}.
\end{align}
The existence of smooth solution $\vphi_i$ is guaranteed by \cite{MR2993005}.	The second order estimate of $\vphi_i$ is obtained in the following Section \ref{second order estimate w1p}. Then after taking $\eps_i\rightarrow 0$ in \eqref{a approximate openess} in $X\setminus D$, we have a K\"ahler cone metric $\om_{\vphi_1}$ which solves \eqref{a openess}. Moreover, $F_\infty\in C^{0,\a,\b}$, $\vphi_1$ is $C^{2,\a,\b}$ by Evans-Krylov estimates. Actually, we have $\vphi_1=\vphi_{\infty}$ by uniqueness of the solution of \eqref{a openess}.

\paragraph{c)}Then we derive the equation of $F_\infty$. We choose $\eta$ be a smooth function and rewrite \eqref{approximate equation openness} in the integral form
\begin{align*}
\int_X	i\p\bar \p \eta \tilde F_{i}\wedge \om_{\vphi_i}^{n-1}&=\int_X\eta  \theta^{\beta_i}\wedge \om_{\vphi_i}^{n-1}+\int_Xf_i i\p\bar\p \eta  \wedge \om_{\vphi_i}^{n-1}-\underline S_{\beta_i} \int_X\eta\om_{\vphi_i}^{n}.
\end{align*}
Since $\tilde F_{i}\rightarrow F_\infty$ in $C^{0,\a,\b}$, $\om_{\vphi_i}\rightarrow \om_{\vphi_\infty}$ in weakly $L^p$, $f_i\rightarrow 0$ in $L^p$, $\theta^{\beta_i},\underline S_{\beta_i}\rightarrow \theta^{\beta_0},\underline S_{\beta_0}$ smoothly, the integral identity above converges to
	\begin{align*}
	\int_X	i\p\bar \p \eta  F_{\infty}\wedge \om_{\vphi_\infty}^{n-1}&=\int_X\eta  \theta^{\beta_0}\wedge \om_{\vphi_\infty}^{n-1}-\underline S_{\beta_0} \int_X\eta\om_{\vphi_\infty}^{n}.
	\end{align*}
Since $\vphi_{\infty}$ is $C^{2,\a,\b}$, we could solve $\tri_{\vphi_\infty} F_1=\tr_{\vphi_\infty}\theta^{\beta_0}-\underline S_{\beta_0}$ to get a solution $F_1\in C^{2,\a,\b}$. But the smooth functions are dense in $W^{1,p}(\om_{\theta^{\beta_0}})$  and $F_\infty \in W^{1,p}(\om_{\theta^{\beta_0}})$, we choose $\eta=F_1-F_{\infty}$ in the integral identity above such that  
	\begin{align*}
\int_X	|\p  (F_1-F_{\infty})|^2_{\vphi_\infty} \om_{\vphi_\infty}^{n}=-\int_X	i\p\bar \p \eta \cdot (F_1-F_{\infty})\wedge \om_{\vphi_\infty}^{n-1}&=0.
\end{align*}
Thus $F_\infty$ coincides with $F_1$ and is also $C^{2,\a,\b}$.	 
	
	In conclusion, $\vphi_\infty\in D_{\bf w}^{4,\a,\b}(\om_{\theta^{\beta_0}})$ is a cscK cone metric.

\textbf{Step 3: Take a distance convergent subsequence of $\vphi_{i_j}$}
We are going to subtract a subsequence in the previous Step such that the $d_1$-distance $d_1(\varphi^{\beta_0},\vphi_\infty)$ converges to 1. Therefore we arrive at the contradiction to \eqref{0distance openness}
and therefore the distance bound \eqref{Uniform distance bound} is proved.

From assumption \eqref{contradiction assumption openness}, we already have the bound of the $d_1$-distance.
In order to apply \lemref{Weakcompactness} to get the $d_1$-convergent subsequence required, it is sufficient to verify that the log $K$-energy is bounded.

	Recall the formula of the log $K$-energy \eqref{log K energy},
	\begin{align*}
	\nu_\beta(\vphi)
	&=E_\beta(\vphi)
	+J_{-\theta^\beta}(\vphi)+\frac{1}{V}\int_M (-(1-\b)\log |s|_h^2+h^\beta_0)\om_0^n.
	\end{align*}
	We compare the log $K$-energy between different cone angle $\beta$ and $\beta_0$,
	\begin{align*}
	\nu_{\beta_i}(\vphi)-\nu_{\beta_0}(\vphi)
	&=\frac{1}{V}\int_M\log\frac{\om_{\theta^{\beta_0}}^n}{\om_{\theta^{\beta_i}}^n}\om^n_\vphi
	+J_{\theta^{\beta_0}-\theta^{\beta_i}}(\vphi)\\
	&+\frac{1}{V}\int_M [(\beta_i-\b_0)\log |s|_h^2+h^{\beta_i}_0-h^{\beta_0}_0]\om_0^n\\
	&=\frac{1}{V}\int_M [{h_0^{\beta_0}+c_{\beta_0}-h_0^{\beta_i}-c_\beta} + (\beta_0-\beta_i)\log(|s|^2_h+\eps)]\om^n_\vphi\\
	&+J_{\theta^{\beta_0}-\theta^{\beta_i}}(\vphi)
	+\frac{1}{V}\int_M [(\beta_i-\b_0)\log |s|_h^2+h^{\beta_i}_0-h^{\beta_0}_0]\om_0^n\\
	&\leq C(n)\|\theta^{\beta_0}-\theta^{\beta_i}\|_{\om_0} d_1(0,\vphi_i)C |\beta_{i}-\beta_0|.
	\end{align*}
	Here, we use the continuity of $h^\beta_0$ and $c_\beta$ on $\beta$ and \lemref{d_1andJ}.
	We see that $\nu_{\beta}(\vphi)$ is continuous on the cone angle $\beta$ for any $\vphi.$	
	As a consequence, we make use of \lemref{approximation proper} to obtain that the log $K$-energy is bounded, that is
	\begin{align*}
	\nu_\beta(\vphi_i)\leq \nu_{\beta_i}(\vphi_i)+C |\beta_{i}-\beta_0|
	\leq \nu^{\eps_i}_{\beta_i}(\vphi_i)+C+C |\beta_{i}-\beta_0|\leq 2C.
	\end{align*}
	Then the proof of this step is complete, and we thus obtain the distance uniform bound in the proposition.

\end{proof}





\end{document}